\documentclass[a4paper,11pt]{article}
\usepackage[utf8]{inputenc} 
\usepackage[english]{babel}
\usepackage[T1]{fontenc}   
\usepackage[normalem]{ulem}
\usepackage{mathpazo}
\usepackage{graphicx}
\usepackage{comment}
\usepackage{amsthm,amsmath,amsfonts,stmaryrd,mathrsfs,amssymb}
\usepackage{mathtools}
\usepackage{bbm}
\usepackage{enumitem}
\usepackage{chemarrow}
\usepackage{tikz}
\usepackage[top=2.5cm, bottom=2.5cm, left=2cm, right=2cm]{geometry}
\usepackage{subfig}
\usepackage{multirow}
\usepackage{physics}
\usepackage{array}
\usepackage{bigints}
\usepackage{etoolbox}
\usepackage{mdframed}
\usepackage{color}
\usepackage{booktabs}
\usepackage{textgreek}
\usepackage[unicode,colorlinks=true,linkcolor=blue,citecolor=red,pdfencoding=auto,psdextra]{hyperref}
 \usepackage{soul}
 \usepackage[font=sf, labelfont={sf,bf}, margin=1cm]{caption}

\newlist{compitem}{itemize}{4}
\setlist[compitem,1]{nolistsep,label=$\bullet$}
\colorlet{link}{red!60!black}
\makeatletter
\newcommand{\labeltext}[3][]{%
    \@bsphack%
    \csname phantomsection\endcsname
    \def\tst{#1}%
    \def\labelmarkup{\textcolor{link}}
    \def\refmarkup{}%
    \ifx\tst\empty\def\@currentlabel{\refmarkup{#2}}{\label{#3}}%
    \else\def\@currentlabel{\refmarkup{#1}}{\label{#3}}\fi%
    \@esphack%
    \labelmarkup{#2}
}
\makeatother

\setenumerate[1]{label=(\roman*), font = \normalfont}

\newcommand{\Ec}[1]{\mathbb{E} \left[#1\right]}

\newcommand{\Pp}[1]{\mathbb{P} \left(#1\right)}
\newcommand{\Pndt}[1]{\mathbb{P}_{n,\delta,T} \left(#1\right)}
\newcommand{\Ppp}[2]{\mathbb{P}_{#1} \left(#2\right)}
\newcommand{\Epc}[2]{\mathbb{E}_{#1} \left[#2\right]}

\newcommand{\Ecsq}[2]{\mathbb{E} \left[#1\mathrel{}\middle|\mathrel{}#2\right]}
\newcommand{\Epcsq}[3]{\mathbb{E}_{#1} \left[#2\mathrel{}\middle|\mathrel{}#3\right]}

\newcommand{\Ppsq}[2]{\mathbb{P} \left(#1\mathrel{}\middle|\mathrel{}#2\right)}
\newcommand{\Ppsqndt}[2]{\mathbb{P}_{n,\delta,T} \left(#1\mathrel{}\middle|\mathrel{}#2\right)}
\newcommand{\Ecp}[2]{\mathbb{E}_{#1} \left[#2\right]}

\newcommand{\intervalle}[4]{\mathopen{#1}#2
                            \mathclose{}\mathpunct{},#3
                            \mathclose{#4}}
\newcommand{\intervalleff}[2]{\intervalle{[}{#1}{#2}{]}}
\newcommand{\intervalleof}[2]{\intervalle{(}{#1}{#2}{]}}
\newcommand{\intervallefo}[2]{\intervalle{[}{#1}{#2}{)}}
\newcommand{\intervalleoo}[2]{\intervalle{(}{#1}{#2}{)}}
\newcommand{\intervalleentier}[2]{\intervalle\llbracket{#1}{#2}
                               \rrbracket}
\newcommand{\enstq}[2]{\left\lbrace#1\mathrel{}:\mathrel{}#2\right\rbrace}
\newcommand{\ind}[1]{\mathbbm{1}_{\lbrace #1 \rbrace}}

\DeclareMathOperator{\haut}{ht}

\newcommand{\cL}{\mathcal{L}}
\newcommand{\bL}{\mathbb{L}}

\newcommand{\T}{\mathcal{T}} 
\newcommand{\Height}{\mathsf{Height}} 
\newcommand{\R}{\mathbb{R}} 
\renewcommand{\P}{{\mathbb{P}}} 
\newcommand{\E}{{\mathbb{E}}}
\newcommand{\Z}{{\mathbb{Z}}}
\newcommand{\N}{{\mathbb{N}}}
\newcommand{\e}{{\varepsilon}}
\newcommand{\Xb}{\mathrm{\mathbf{x}}} 
\newcommand{\Sb}{\mathrm{\mathbf{s}}} 
\newcommand{\X}{\mathrm x} 
\newcommand{\XXb}{\mathrm{\mathbf{X}}}
\newcommand{\XX}{\mathrm X} 

\newcommand{\cvloi}[1][n]{\enskip\mathop{\longrightarrow}^{(d)}_{#1 \to \infty}\enskip}

\newmdtheoremenv{theorem}{Theorem}[section]
\newmdtheoremenv{proposition}[theorem]{Proposition}
\newtheorem{lemma}[theorem]{Lemma}

\newtheorem{corollary}[theorem]{Corollary}

\DeclareSymbolFont{extraup}{U}{zavm}{m}{n}
\DeclareMathSymbol{\vardspade}{\mathalpha}{extraup}{81}
\DeclareMathSymbol{\varheart}{\mathalpha}{extraup}{86}
\DeclareMathSymbol{\vardiamond}{\mathalpha}{extraup}{87}
\DeclareMathSymbol{\varclub}{\mathalpha}{extraup}{84}

\makeatletter
\renewcommand*{\@fnsymbol}[1]{\ensuremath{\ifcase#1\or  \vardspade \or \varheart \or \vardiamond\or \varclub \or \bigstar \or
   \mathsection\or \mathparagraph\or \|\or **\or \dagger\dagger   \or \ddagger\ddagger \else\@ctrerr\fi}}
\makeatother

\author{
\qquad  
Emmanuel Kammerer 
\thanks{CMAP, \'Ecole polytechnique, Institut Polytechnique de Paris, 91120 Palaiseau, France, \textsf{emmanuel.kammerer@polytechnique.edu}
}
 \qquad  
Igor Kortchemski 
\thanks{CNRS \& DMA, École normale supérieure, PSL University, 75005 Paris, France, \textsf{igor.kortchemski@math.cnrs.fr}
} 
\qquad
Delphin Sénizergues 
\thanks{MODAL'X, UMR CNRS 9023, Universit\'e Paris Nanterre, 92000 Nanterre,  France, \textsf{dsenizer@parisnanterre.fr}
} \qquad  
}

\begin{document}
\title{The height of the infection tree}
\date{}
\maketitle

\begin{abstract}
We are interested in the geometry of the ``infection tree'' in a stochastic SIR (Susceptible-Infectious-Recovered) model, starting with a single infectious individual. 
This tree is constructed by drawing an edge between two individuals when one infects the other. 
We focus on the regime where the infectious period before recovery follows an exponential distribution with rate $1$, and infections occur at a rate  
$\lambda_{n} \sim \frac{\lambda}{n}$ where $n$ is the initial number of healthy individuals with $\lambda>1$. 
We show that provided that the infection does not quickly die out, the height of the infection tree is asymptotically $\kappa(\lambda) \log n$ as $n \rightarrow \infty$, where $\kappa(\lambda)$ is a continuous function in $\lambda$ that undergoes a second-order phase transition at $\lambda_{c}\simeq 1.8038$. 
Our main tools include a connection with the model of uniform attachment trees with freezing and the application of martingale techniques to control profiles of random trees. 
\end{abstract}

\begin{figure}[!ht]
\centering
\includegraphics[width=0.48\linewidth]{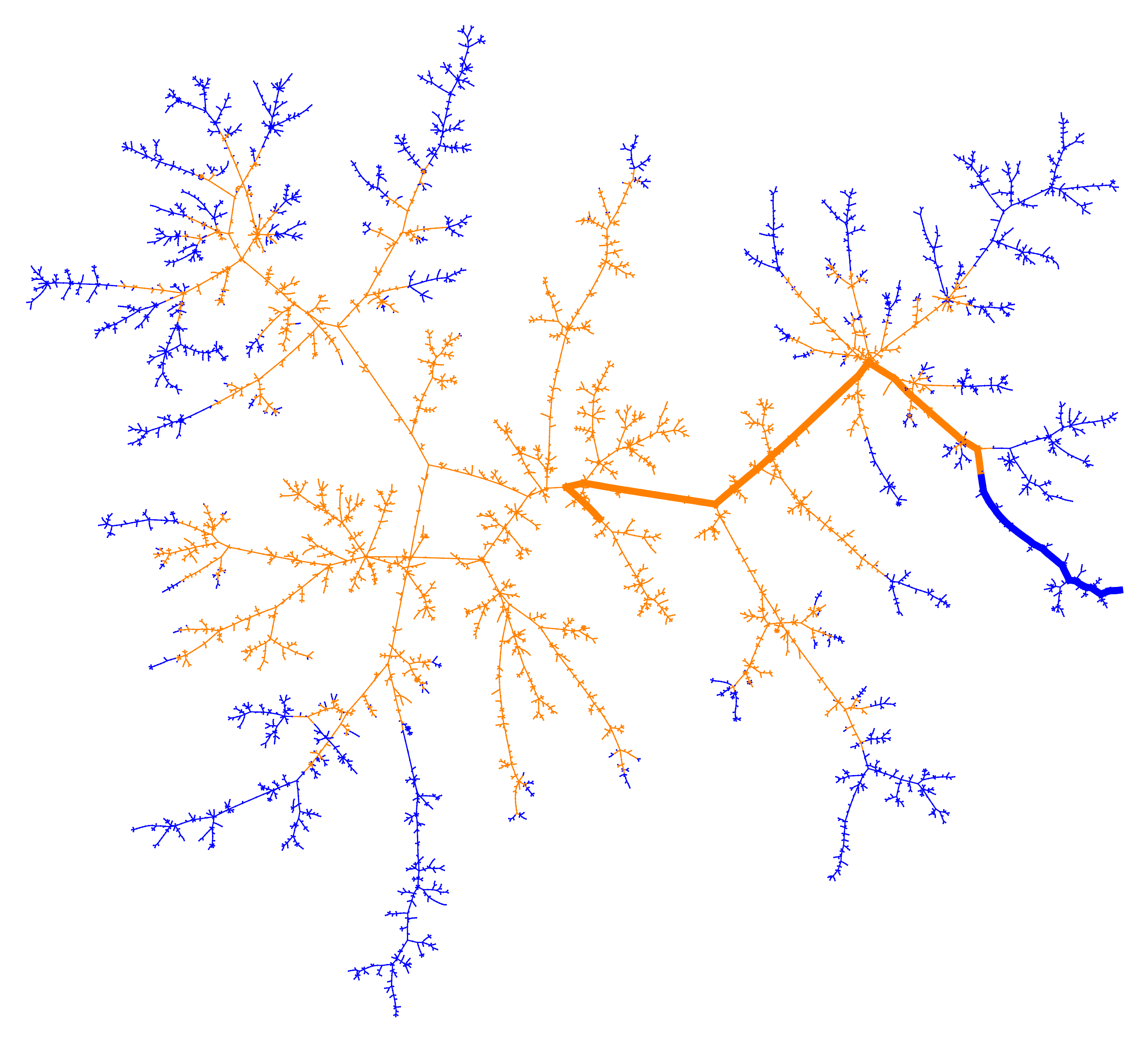}%
\quad  \includegraphics[width=0.48\linewidth]{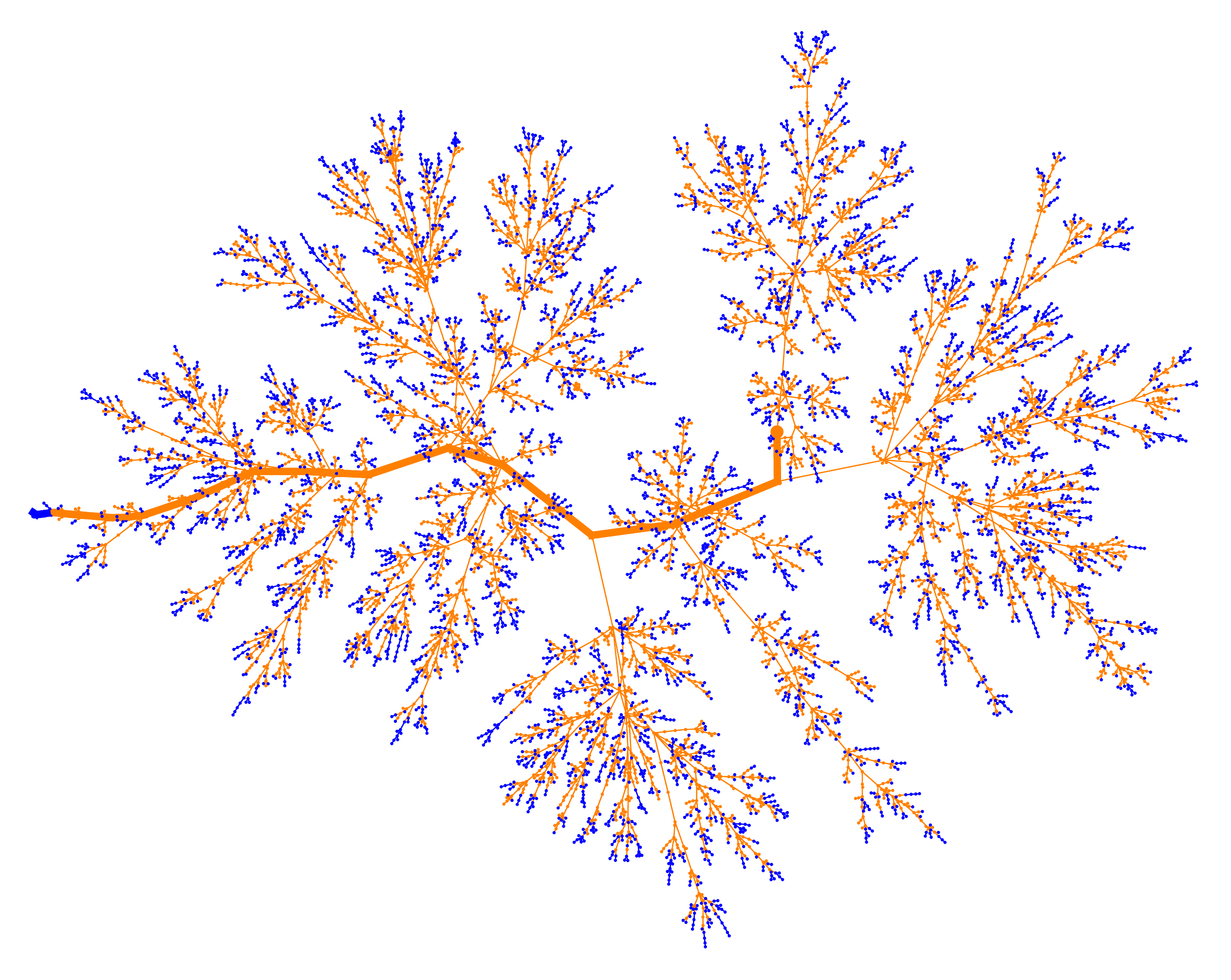} 
\caption{\label{fig:ssimus_intro}Simulations of  large   infection trees for $\lambda=1.1$ (left) and $\lambda=5$ (right). The trees both have  $10000$ vertices, and the orange vertices represent  the first half of the vertices (in order of appearance). The bold path is the shortest path from the root  to the vertex furthest away from the root. In the first case, the orange and blue trees both macroscopically contribute  to the length of this path, while in the second case only the orange tree macroscopically contributes.
}
\end{figure}

\pagebreak

\tableofcontents

\section{Introduction}

Many growth processes that involve real-world networks, such as the spread of disease in a human population, the proliferation
of rumors on social media, the spread of computer viruses on computer networks, and the
development of social structures among individuals, can be modeled using random graphs. 
We are interested here in the stochastic SIR (Susceptible-Infectious-Recovered) model, which is a classical model for the evolution of epidemics  (for background on stochastic  epidemic models, see \cite{AB12,BPBLST19}). 
It has been extensively studied in multiple directions, we mention some of them in connection with random graphs:  fluid limit for the density processes 
of an 
SIR dynamics on a complete graph \cite[Sec.~5.5]{AB12}, study of the SIR epidemic on a random graph with given degrees \cite{JLW14}. 

 In this work, we focus on the so-called \emph{infection tree}, obtained by keeping track of the infections in a SIR process on a complete graph  where an edge is present between two individuals if one has infected the other. 
 In the context of contact-tracing this tree (sometimes also called 
 ``transmission tree") is natural to consider \cite{KFH06,CSBV16}, yet to the best of our knowledge  its mathematical study has only been first considered very recently in \cite{BBKK23+}. 
  
Denote by $\mathcal{T}^{n}$ the random infection tree obtained with $1$ infectious individual (called ``patient zero'') and $n$ susceptible individuals, where infectious individuals recover at rate $1$ and where infections occur at a rate $\lambda_{n} \sim \frac{\lambda}{n}$, see Section~\ref{ssec:SIR} for a formal definition. 
We let $\Height(\T^{n})$ be the maximal graph distance between patient zero and any other vertex of $ \mathcal{T}^{n}$. 
To state our main limit theorem concerning  $\Height(\T^{n})$  we need to introduce some notation.
  
Let $W$ be the principal branch of the Lambert function, which satisfies $W(x) e^{W(x)}=x$ for $x \geq -\frac{1}{e}$.  
Observe that $W(x)<0$ when $-\frac{1}{e} \leq x<0$. 
For all $z>0$ and $\lambda>1$,
set
\begin{equation}
\label{eq:param}f_{\lambda}(z)\coloneqq  1+ \frac \lambda{\lambda-1}(e^{z}-1-z e^{z}), \quad  z_{\lambda}\coloneqq\inf \{t >0: f_{\lambda}(t)=0\},
 \quad m_{\lambda}\coloneqq-W(-\lambda e^{-\lambda}) \in (0,1).
\end{equation}
 The quantity $z_{\lambda}$ is well-defined since $f_{\lambda}(0)=1$ and $f_{\lambda}$ is decreasing on $\R_{+}$ (since $f'_{\lambda}(z)=- \frac{\lambda}{\lambda-1} ze^{z}$).
 Using the definition of $W$, it is a simple matter to check that $z_{\lambda}=1+W(-\frac{1}{e \lambda})$. 
 Finally, let $\lambda_{c}$ be the unique value of $\lambda >1$ that solves the equation
\begin{equation}\label{eq:lc}
m_{\lambda}=e^{-z_{\lambda}}.
\end{equation}  
 The existence and uniqueness of $\lambda_{c}$ will be justified later (see Proposition~\ref{prop:lc}). 
 Numerically, $\lambda_{c}\simeq~ 1.8038$.

  \begin{theorem}
  \label{thm:main}
Assume that $\lambda_n \sim\frac{\lambda}{n}$ as $n\rightarrow \infty$ for some $\lambda> 1$. 
Let $\mathcal{B}$ be a Bernoulli random variable with parameter $1-\frac{1}{\lambda}$. Then
\begin{equation*}
\frac{\Height(\T^{n})}{\log n}  \quad \mathop{\longrightarrow}^{(d)}_{n \rightarrow \infty} \quad \kappa(\lambda) \cdot \mathcal{B},
\qquad 
\text{where}
\qquad
	\kappa(\lambda)= 
\begin{cases}
	\left(   \frac{\lambda}{(\lambda-1)m_{\lambda}}+  \frac{f_{\lambda}(-\log m_{\lambda})}{-\log m_{\lambda}}   \right) & \textrm{if } \lambda \leq \lambda_{c},\\
	\frac{\lambda}{\lambda-1} e^{z_{\lambda}} &  \textrm{if } \lambda \geq \lambda_{c}. 
\end{cases}
\end{equation*}
\end{theorem}

Using the explicit expressions of $z_{\lambda}$ and $m_{\lambda}$, the expression for $\kappa(\lambda)$ can be alternatively be written as
\begin{equation}
\label{eq:kappa}
\kappa(\lambda)=\begin{cases}
 \left(  \frac{\lambda}{(1-\lambda) W \left(  -  \lambda e^{-\lambda}\right)}+  \frac{f_{\lambda}\left(-\log (-W(-\lambda e^{-\lambda}))\right)}{-\log (-W(-\lambda e^{-\lambda}))}   \right) & \textrm{if } \lambda \leq \lambda_{c},\\
 \frac{1}{(1-\lambda) W \left(  - \frac{1}{e \lambda}\right)} &  \textrm{if } \lambda \geq \lambda_{c}. 
\end{cases}
\end{equation}
It is not difficult to check that ${f_{\lambda}(-\log m_{\lambda})}=0$ for $\lambda=\lambda_{c}$, so that the two limiting quantities coincide at $\lambda=\lambda_{c}$. Further, their derivatives coincide at $\lambda=\lambda_{c}$ as well, but not their second order derivatives: the height of the infection tree thus undergoes a {second}-order phase transition at $\lambda_{c}$.

\begin{figure}[!ht]
\label{fig:plot}
\centering
\includegraphics[scale=0.7]{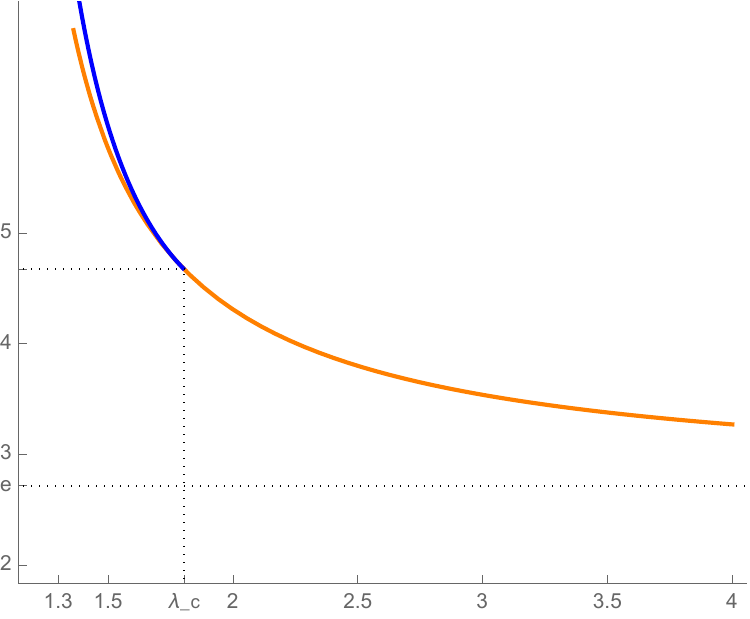}%
\caption{In orange, plot of the function appearing in the limit of $\Height(\widehat{\T}^{n})/\log n$, where $\widehat{\T}^{n}$ is the infection tree after the early stages of the epidemic and before the late stages of the epidemic (this is the orange tree in Figure~\ref{fig:ssimus_intro}). 
The function $\kappa$ appearing in \eqref{eq:kappa} is the blue curve for $\lambda \leq \lambda_{c}$ and the orange curve for $\lambda \geq \lambda_{c}$: when $\lambda<\lambda_{c}$, the late stages of the epidemic have an influence on the total height of the infection tree, but not when $\lambda>\lambda_{c}$.
}
\end{figure}

The reason why we focus on the regime $\lambda_n \sim\frac{\lambda}{n}$ for some $\lambda> 1$ is that it is the  remaining delicate case which was not covered in \cite[Theorem~23 \& 24]{BBKK23+}. 
Indeed, when $\lambda \leq 1$ the infection tree converges locally in distribution towards a finite Bienaymé tree, while in the case $\lambda_n \gg \frac{1}{n}$ we have  ${\Height(\T^{n})}/{\log n} \rightarrow  e$ in probability. 
Informally speaking, in the latter case $\T^{n}$ behaves {``as''} 
a random recursive tree with $n$ vertices, which corresponds to the case where there is no recovery.

Several further comments are in order. 
At the very early stages of the epidemic, the infection tree roughly 
grows like a Bienaymé random tree with geometric offspring distribution with parameter $\frac{1}{1+\lambda}$, which has a probability $1-\frac{1}{\lambda}$ of survival. 
This explains the presence of $ \mathcal{B}$. 
Let us explain the intuition behind the phase transition (this will be made precise later). 
After the early stages of the epidemic and before the late stages of the epidemic (i.e.\ when the population contains a positive fraction of infectious individuals as well as a positive fraction of healthy individuals), the height of the infection tree is of order $\frac{\lambda}{\lambda-1} e^{z_{\lambda}}\log n$. 
Between this moment and the end of the epidemic, the infection process will continue from each of the active vertices in the tree, resulting in additional subtrees hanging off of those vertices in the final tree $\mathcal{T}^n$. 
It turns out that these outgrowths macroscopically contribute to the total height of the infection tree when $\lambda>\lambda_{c}$, but not when $\lambda<\lambda_{c}$, see Figure~\ref{fig:ssimus_intro} for an illustration.

 More precisely, the intuition behind Theorem~\ref{thm:main} is the following: 
 When $\lambda>1$, with high probability either the epidemic dies out quickly (this corresponds to $\mathcal{B}=0$), 
 or it dies out after $\simeq t_{\lambda} n$ steps for a certain $t_{\lambda}>0$ (this corresponds to $ \mathcal{B}=1$).  
 For $\delta>0$ small enough, denote by $ \mathcal{T}^{n}_{\delta}$ the infection tree after $\lfloor (t_{\lambda}-\delta) n \rfloor$ steps of infection or recovery, conditionally given the fact that the epidemic has not died out yet. 
 Then, setting $\gamma=\frac{\lambda}{\lambda-1}$, with high probability:
 \begin{enumerate}[label=(\roman*)]
 \item The tree $\mathcal{T}^{n}_{\delta}$ has height of order $\gamma e^{z_{\lambda}} \log n$ and
 for any $z\in \intervalleoo{0}{z_{\lambda}}$
 there are of order $n^{f_{\lambda}(z)}$ active vertices at height $\gamma e^{z} \log n$.
 \item  For $\delta>0$ small, the outgrowths in $\mathcal{T}^n$ hanging off of each vertex that was active in $\mathcal{T}^{n}_{\delta}$
 are roughly independent subcritical Bienaymé trees with geometric offspring distribution with mean $m_{\lambda}<1$.
 \item The height of a forest of $n^{f_{\lambda}(z)}$ such Bienaymé trees is of order $\frac{f_{\lambda}(z)}{-\log m_{\lambda}}\log n$.
 \item It follows that the height of $ \mathcal{T}^{n}$ is of order
 \begin{align*}  \sup_{0 \leq z \leq z_{\lambda}}  \left( \gamma  e^{z}+ \frac{f_{\lambda}(z)}{-\log m_{\lambda}} \right) \log n,
 \end{align*} 
 which entails the desired result (the supremum in the above display is reached at $z=z_{\lambda}$ 
  if and only if $\lambda  \geq \lambda_{c}$, see Proposition~\ref{prop:lc}). 
 \end{enumerate}
 
 The main technical challenge is to prove step (i), whose rigorous statement can be found in Theorem~\ref{thm:profil} stated in Section~\ref{subsec:profile of the infection tree} below.  
 This result ensures that in the tree $ \mathcal{T}^{n}_{\delta}$, at height $\gamma e^{z} \log n \in \N$, there are roughly
 \begin{align*}
 	\frac{n^{f_\lambda(z)}}{\sqrt{\log n}} \cdot e^{O_\P(1)}
 \end{align*}
active vertices, simultaneously for $z\in \intervalleoo{0}{z_{\lambda}}$. 
 In order to get access to the profile of $\mathcal{T}^{n}_{\delta}$ we rely on a strategy that has proved its efficacy in the context of growing random trees (binary search tree \cite{CDJ01,CKMR05}, uniform recursive tree \cite{KMS17}, plane-oriented recursive tree \cite{KMS17}, weighted recursive trees \cite{Sen21} and others): 
 we {first} study the behaviour of the Laplace transform of the profile with the help of appropriately defined martingales indexed by $z\in \mathbb{C}$, {then we} prove that for $z$ in a given domain of the complex plane they converge in $L^p$ (following the ideas introduced by Biggins \cite{Big92} in the context of the branching random walk), and then finally we use this control on the Laplace transform to recover the profile using a Fourier inversion {argument}.
 Let us mention in particular the work \cite{KMS17}, which gives a very strong control, known as the Edgeworth expansion, on the profile of a vast family of growing trees. 
 
 {Unfortunately}, {the results from \cite{KMS17}} are not directly applicable here. 
 Indeed, 
 contrary to all the models cited above, 
 our sequence $(\mathcal{T}^{n}_{\delta})_{n\geq 1}$ is not itself growing as $n$ changes and each tree $\mathcal{T}^{n}_{\delta}$ is rather defined through its own process of growing trees, 
 on its own probability space. 
 This creates additional difficulties, which we circumvent using various couplings.

\paragraph{Outline.} In Section~\ref{sec:tree} we recall the definition of the model of uniform attachment with freezing and explain why  the infection tree is a uniform attachment tree with freezing. Section~\ref{sec:height} contains the proof of our main Theorem~\ref{thm:main}, assuming a limit theorem for the profile of the infection tree whose proof is the content of Section~\ref{sec:profile}. Section~\ref{sec:polya} and Section~\ref{sec:lambert} contain several technical results, the first one concerning time-dependent P\'olya urns and the second one concerning bounds for the Lambert function.

\paragraph{Acknowledgments.} We thank \'Etienne Bellin and Arthur Blanc-Renaudie for stimulating discussions at early stages of this work. E.K.\@ acknowledges the support of the ERC consolidator grant 101087572 ``SuPerGRandMa''.

  \section{The infection tree is a uniform attachment tree with freezing}
  \label{sec:tree}
  
 In this section, we define the infection tree and describe it as a uniform attachment tree with freezing. We also provide for future use a table of notation (Table \ref{tab:notation}).
  \begin{table}[htbp]\caption{Table of the main notation and symbols.}
  	\centering
  	\begin{tabular}{c c p{12cm} }
  		\toprule
  		$\N$ && the set $\{1,2,3,\dots\}$ of all positive integers\\
  		\hline
  		$\Xb = (\X_n)_{n\in\N}$ && a sequence of elements of $\{-1,1\}$\\
  		$(\T_n(\Xb))_{n\ge 0}$ && the sequence of uniform attachment trees with freezing built from $\Xb$\\
  		\hline
  		$\tau_n$ && the number of steps made when the epidemic ceases starting with $n$ susceptible individuals\\
  		$(\T^n_k)_{0 \le k \le \tau_n}$ && the infection tree after $k$ steps\\
  		$\T^n = \T^n_{\tau^n}$ && the infection tree when the epidemic ceases\\
  		$(H^n_k,I^n_k)_{k\ge 0}$ &&Markov chain of susceptible and infectious individuals defined in \eqref{eq:transitions}\\
  		\hline
  		$f_\lambda(z)$ && $ 1 +(\lambda/(\lambda-1))(e^z-1-ze^z)$\\
  		$W$ && the principal branch of the Lambert function\\
  		$m_\lambda$ &&$  -W(-\lambda e^{-\lambda})$\\
  		$z_\lambda$ && $ \inf\{t>0  :  f_\lambda(t) = 0\} = 1+ W(-1/(e\lambda))$\\
  		$g_\lambda(t)$&& $ (1/\lambda)W(\lambda e^\lambda e^{-\lambda t})$\\
  		$t_\lambda$ && $  \inf\{t\ge 0  : 2-2g_\lambda(t)-t =  0\}$\\
  		$\gamma$&&$\lambda/(\lambda-1)$\\
  		\hline
  		$\mathbb{A}_{k}^{n}(h)$&&$\#\{\text{active vertices at height $h$ at time $k$ of $\mathcal{T}^{n}_{k}$}\}$\\
  		$\bL_k^n(h)$ &&${\mathbb{A}_{k}^{n}(h)}/{I_k^n}$\\
  		\hline
  		$\Height(\T)$ && height of a tree $\T$\\
  		$\haut(v)$ && height of a vertex $v$\\
  		\bottomrule
  	\end{tabular}
  	\label{tab:notation}
  	\end{table}

\subsection{Uniform attachment with freezing} \label{subsec:uniform attachment with freezing}

Let $\Xb=(\X_i)_{i \geq 1}\in \{-1,+1\}^\N$. 
In what follows, it will be useful to define more generally a sequence of random forests (i.e.~sequences of trees) built by uniform attachment with freezing. Such forests will be made of  rooted and vertex-labelled trees. The label of a vertex is either  ``$f$'' if it is frozen, or  ``$a$'' if it is still active.

\paragraph{\labeltext[1]{Algorithm 1.}{algo1}} For an integer $r \geq 1$:
\begin{compitem}
	\item Start with a forest $ \mathcal{F}^{r}_{0}=(\T^{1}_0(\Xb), \ldots, \T^{r}_{0}(\Xb))$ of $r$ trees which are all made of a single root vertex labelled $a$.
	\item For every $n \geq 1$, if $\mathcal{F}^r_{n-1}(\Xb)$ has no vertices labelled $a$, then set $\mathcal{F}^r_n(\Xb) \coloneqq \mathcal{F}^r_{n-1}(\Xb)$. Otherwise let $V_n$ be a random uniform active vertex of $\mathcal{F}^r_{n-1}(\Xb)$, chosen independently from the previous ones. Then:
	\begin{compitem}
		\item[--] If $\X_n=-1$, build $\mathcal{F}^r_n(\Xb)$ from $\mathcal{F}^r_{n-1}(\Xb)$  
		by replacing the label $a$ of $V_n$ with the label $f$;
		\item[--] If $\X_n=1$, build $\mathcal{F}^r_n(\Xb)$ from $\mathcal{F}^r_{n-1}(\Xb)$ by adding an edge between $V_n$ and a new vertex labeled $a$. 
	\end{compitem}
\end{compitem}

When    $\Xb=(\X_i)_{1 \leq i  \leq n}\in \{-1,+1\}^n$ has finite length, we  build $(\mathcal{F}^r_k(\Xb))_{0 \leq k \leq n}$ in the same way, and set $\mathcal{F}^r_{k}(\Xb)\coloneqq \mathcal{F}^r_{n}(\Xb)$ for $k > n$. 
We set $\mathcal{F}^r_{\infty}(\Xb)\coloneqq\lim_{n \rightarrow \infty} \mathcal{F}^r_{n}(\Xb)$, where the limit makes sense since the sequence $(\mathcal{F}_{n}(\Xb))_{n \geq 0}$ is weakly increasing.  

For every $n \in \Z_{+} \cup \{+\infty\} $, we denote by  $(\T^{1}_n(\Xb), \ldots, \T^{r}_{n}(\Xb))$ the $r$ trees of the forest $\mathcal{F}^{r}_{n}$. 
When $r=1$ and $n \in \Z_{+}\cup \{+\infty\}$, to simplify notation, we write $\T_{n}(\Xb)$ for the only  tree  $\T^{1}_{n}(\Xb)$ of $ \mathcal{F}^{1}_{n}(\Xb)$.

In the sequel, for all $s \in (0,1]$, we denote by $\mathsf{G}(s)$ a random variable with geometric  law on $\Z_{+}$ with parameter $s$, with law given by $\P(\mathsf{G}(s)=k)=s (1-s)^{k}$ for $k \geq 0$. 
By abuse of notation, we will use the symbol $\mathsf{G}(s)$ to denote the law of this random variable.

\subsection{The infection tree of a SIR epidemic}
\label{ssec:SIR}

Here we formally define the SIR epidemic process together with its infection tree, and explain the connection with uniform attachment trees with freezing.
 
 We assume that initially there is $1$ infectious individual and $n$ susceptible individuals. 
 The {duration of the} infectious periods of different infectious individuals are i.i.d.\ exponential random variables of parameter~$1$. 
 During its infectious period, an infectious individual {comes into contact} with any other given individual {at a set of times distributed as} a time-homogeneous Poisson process with intensity $\lambda_{n}$. 
 At such a time of contact, if the 
 other individual was susceptible, then it becomes infectious and is immediately able to infect other individuals. 
 An individual is considered \emph{removed} once its infectious period is over, and is then immune to new infections, playing no further part in the epidemic spread. 
 The epidemic ceases as soon as there are no more infectious individuals present in the population. 
 All Poisson processes are assumed to be independent of each other; they are also independent {of the duration} of infectious periods.

We call a \emph{step} of the process an event where either a susceptible individual becomes infectious, or where an individual's infectious period terminates. 
Denote by $\tau_{n}$  the number of steps made when the epidemic ceases. 
For $0 \leq k \leq \tau_{n}$, let $\mathcal{T}^{n}_{k}$ be the infection tree after $k$ steps, in which the vertices are individuals and where an edge is present between two individuals if one has infected the other {at some point during the process}.  
We are interested in the shape of the full infection tree $\T^n\coloneqq\mathcal{T}^{n}_{\tau_{n}}$  when the epidemic ceases.

The  connection with uniform attachment trees with freezing is established by first  choosing the sequence $\Xb \in \{-1,+1\}^\N$ appropriately at random. Specifically, let $(H^{n}_{k}, I^{n}_{k})_{k \geq 0}$ be a Markov chain with initial state $(H^{n}_{0},I^{n}_{0})=(n,1)$ and  transition probabilities given by
\begin{equation}
\label{eq:transitions}(H^{n}_{k+1},I^{n}_{k+1}) = \begin{cases}
(H^{n}_{k}-1,I^{n}_{k}+1)  &\textrm{with probability } \frac{\lambda_{n} H^{n}_{k}}{1+ \lambda_{n} H^{n}_{k}}\\
 (H^{n}_{k},I^{n}_{k}-1)  &\textrm{with probability } \frac{1}{1+\lambda_{n} H^{n}_{k}}
 \end{cases}
\end{equation}
with {set} $ \{(k,0) : 0 \leq k \leq n\}$ {of} absorbing states.  
Observe that the number of susceptible individuals and the number of infectious individuals in the SIR epidemic evolve according to this Markov chain. 
Then define the random sequence $\XXb^{n}=(\XX^{n}_{i})_{1 \leq i \leq \tau'_{n}}$ of $\pm 1$ as follows:  let $\tau'_{n}$ be the absorption time of the Markov chain, and for $1 \leq i \leq \tau'_{n}$ set $\XX^{n}_{i}=I^{n}_{i}-I^{n}_{i-1}$.

Then by construction, it is clear that
\begin{align}\label{eq couplage freezing}
(\mathcal{T}^{n}_{k})_{0 \leq k \leq \tau_{n}}  \quad \mathop{=}^{(d)}  \quad  (\mathcal{T}_{k}(\XXb^n))_{ 0 \leq k \leq \tau'_{n}}. 
\end{align}
The above equality also holds in terms of labeled trees: the active vertices correspond to the infectious individuals and the frozen vertices to the ``removed'' individuals. 
In the sequel, we will often implicitly make this identification.

\subsection{Coupling uniform attachment trees with freezing and  Bienyamé trees}

We 
{construct below}
a coupling between uniform attachment trees with freezing and  Bienyamé trees with geometric offspring distribution. 
We first introduce some notation.

Let $\XXb=(\XX_{k})_{k \geq 1}$ be a sequence of $ \{\pm1\}$-valued random variables.   For all $k \in \N$, for all $\X_1, \ldots, \X_{k-1}\in \{\pm1\}$ {for which $\P(\XX_1=\X_1, \ldots, \XX_{k-1}= \X_{k-1})>0$,}  set 
  \begin{equation*}
   	r_{k}(\X_1, \ldots, \X_{k-1})\coloneqq\Ppsq{\XX_{k}=-1}{\XX_1=\X_1, \ldots, \XX_{k-1}= \X_{k-1}}.
  \end{equation*}
  	For every $r \geq 1$, set $\tau_{r}(\XXb)= \inf \{ n \geq  1 : \XX_{1}+ \cdots+\XX_{n}=-r \} \in \N \cup \{+\infty\}$.

\begin{lemma}
	\label{lem:couplage}
	Let $N \geq 1$ be an integer. 
	Let $p,q \in (\frac{1}{2},1)$ with $p \leq q$. 
	Let $ \mathcal{E}$ be the event defined {as}
\begin{equation*}
	\mathcal{E}\coloneqq \{ \forall k  \in {\intervalleentier{1}{\tau_{N}(\XXb)-1}, \ } 
p \leq r_{k}(\XX_1, \ldots, \XX_{k-1}) \leq q \}, 
\end{equation*}
with the convention  $\intervalleentier{1}{\tau_{N}(\XXb)-1}= \N$ when  $\tau_{N}(X)=\infty$.
The following assertions hold.
\begin{enumerate}[label=(\roman*)]
		\item\label{it:lem:couplage couplage} We can couple $\XXb$ and $ \mathcal{F}_{\infty}^N(\XXb)$ with two  families of finite  trees $ (\underline{ \mathcal{T}}^{i})_{1 \leq i \leq N}$ and $(\overline{\mathcal{T}}^{i})_{1 \leq  i \leq  N}$, such  that    $ (\underline{ \mathcal{T}}^{i})_{1 \leq i \leq N}$ are i.i.d.~Bienaymé trees  with  offspring distribution $\mathsf{G}(q)$ 
		  and $(\overline{ \mathcal{T}}^{i})_{1 \leq i \leq N}$ are i.i.d.~Bienaymé trees with offspring distribution 
		$\mathsf{G}(p)$  
		and  such that on the event $ \mathcal{E}$ we have $\underline{ \mathcal{T}}^{i} \subset \mathcal{T}^{i}_{\infty}(\XXb) \subset \overline{\mathcal{T}}^{i}$ for every $1 \leq i \leq N$.
		\item\label{it:lem:couplage queue hauteur} There exists a constant $C>0$ depending only on $p$ and $q$ such that for every $1 \leq i \leq N$ and $h \geq 0$,
		\begin{align*}  \frac{1}{C} \left(  \frac{1}{q}-1 \right)^{h}  \leq\P \left(\Height\left(\underline{\mathcal{T}}^{i}\right) \geq h  \right) \leq \P \left(\Height\left(\overline{\mathcal{T}}^{i}\right) \geq h  \right) \leq  C \left( \frac{1}{p}-1\right)^{h} .
		\end{align*} 
	\end{enumerate}
\end{lemma}

  \begin{proof}
  To simplify notation and to avoid unnecessary details, we prove the result for $N=1$.

  	Write $ \mathcal{A}_{T}$ for the set of all active vertices of a tree $T$. 
Below 
  	we build by induction a sequence of trees $(\underline{\mathcal{T}_{n}}, \mathcal{T}_{n}, \overline{\mathcal{T}_{n}} )_{n \geq 0}$, a sequence $\widetilde{\XXb}=(\widetilde{\XX}_k)_{k\ge 1}$  and a {non-decreasing} sequence of integers $(\sigma(n))_{n \geq 0}$  such that 
if we define
	\begin{equation*}
	\widetilde{\mathcal{E}} \coloneqq \{ \forall k  \in \intervalleentier{1}{\tau_{{1}}(\widetilde{\XXb})-1}
	, \  p \leq r_{k}(\widetilde{\XX}_1, \ldots, \widetilde{\XX}_{k-1}) \leq q \},
	\end{equation*}
then the following properties hold:
  	\begin{enumerate}[label=(\alph*)]
			  		\item\label{it:seq}  The two sequences $\widetilde{\XXb}$ and  ${\XXb}$ have the same law.
  		\item\label{it:tree sandiwch2} For every $n \geq0$ we have $ \underline{\mathcal{T}_{n}} \subset   \overline{\mathcal{T}_{n}}$ and $ \mathcal{A}_{\underline{\mathcal{T}_{n}}} \subset  \mathcal{A}_{\overline{\mathcal{T}_{n}}}$.
		  		\item\label{it:tree sandiwch} On the event $\widetilde{\mathcal{E}}$ for every $n \geq0$ we have $ \underline{\mathcal{T}_{n}} \subset \mathcal{T}_{n} \subset  \overline{\mathcal{T}_{n}}$ and $ \mathcal{A}_{\underline{\mathcal{T}_{n}}} \subset \mathcal{A}_{\mathcal{T}_{n}} \subset  \mathcal{A}_{\overline{\mathcal{T}_{n}}}$.

  		\item\label{it:non-decreasing trees} The sequences  $( \underline{\mathcal{T}_{n}})_{n \geq 0}$,  $( {\mathcal{T}_{n}})_{n \geq 0}$ and  $( \overline{\mathcal{T}_{n}})_{n \geq 0}$ are non-decreasing, so their limits $\underline{\T_\infty}$, $\T_\infty$ and $\overline{\T_\infty}$ are well defined.
  		\item\label{it:upper tree is Bienaymé} The tree $\overline{\T_\infty}$ has the law of a Bienaymé tree with offspring distribution $\mathsf{G}(p)$. 
  		\item\label{it:middle tree has the right law} We have $\sigma(n) \rightarrow\infty$ as $n \to \infty$ and $(\widetilde{\XX}_k,\mathcal{T}_{\sigma^{-1}(k)})_{k \geq 1}$ has the same law as $(\XX_{k},\mathcal{T}_{k}({\XXb}))_{k \geq 1}$, where $\sigma^{-1}(k)=\inf \{n \geq 0 : \sigma(n) \geq k\}$. This entails that $(\widetilde{\XXb},\T_\infty)$ has the same law as $(\XXb,\mathcal{T}_{\infty}({\XXb}))$.
  		\item\label{it:lower tree is Bienaymé} The tree $\underline{\T_\infty}$ has the law of a Bienaymé tree with offspring distribution $\mathsf{G}(q)$. 
  	\end{enumerate}
  	Point~\ref{it:lem:couplage couplage} {will then follow from the above properties}: by \ref{it:tree sandiwch} and \ref{it:non-decreasing trees}, the trees $\underline{\T_\infty}$, $\T_\infty$ and $\overline{\T_\infty}$ are constructed on the same probability space in such a way that on the event $\widetilde{\mathcal{E}}$ we have $\underline{\T_\infty}\subset \T_\infty \subset \overline{\T_\infty}$; by \ref{it:upper tree is Bienaymé}, \ref{it:middle tree has the right law}, \ref{it:lower tree is Bienaymé}, those trees have the desired distributions. 
  
  		Let us now focus on proving properties \ref{it:seq} through \ref{it:lower tree is Bienaymé}.
		  	Along with the trees, the sequence $\widetilde{\XXb}$ and the sequence $\sigma$, the construction will build an auxiliary  sequence $(C_{n})_{n \geq 0}$ of  $\{0,1\}$-valued random variables (which, roughly speaking, allows to monitor {whether} the condition $  p \leq r_{k}(\widetilde{\XX}_1, \ldots, \widetilde{\XX}_{k-1}) \leq q$ holds).
  	
  	To start with, $\underline{\mathcal{T}_{0}}, \mathcal{T}_{0}, \overline{\mathcal{T}_{0}}$ are all made of a {single} active vertex, $\sigma(0)=0$, $C_{0}=1$ and $(\widetilde{\XX}_k)_{1\leq k\leq \sigma(0)}$ is then just the empty sequence. 
  	Let $(U_{n})_{n \geq 1}$ be a sequence of i.i.d.\ uniform random variables on $[0,1]$.  
  	For $n \geq0$, assuming that $( \underline{\mathcal{T}_{m}},\mathcal{T}_{m}, \overline{\mathcal{T}_{m}} )_{0 \leq m \leq n}$ and $(\sigma(m))_{0 \leq m \leq n}$ and $(\widetilde{\XX}_k)_{1\leq k\leq \sigma(n)}$ have been constructed, we proceed as follows. 

  	\begin{itemize}[label=--]
		\item[(I)]If $C_{n}=1$ and $r_{\sigma(n)+1} (\widetilde{\XX}_1, \ldots, \widetilde{\XX}_{\sigma(n)}) \in  [p,q]$, set $C_{n+1}=1$ and build $( \underline{\mathcal{T}_{n+1}},\mathcal{T}_{n+1}, \overline{\mathcal{T}_{n+1}} )$ as follows:
  	  	\begin{itemize}
			\item[(A)] If $\overline{\mathcal{T}_{n}}$  has at least one active vertex, choose an active vertex $\mathcal{V}_{n}$ of $\overline{\mathcal{T}_{n}} $  uniformly at random, independently of all other choices.
  		Then build $( \underline{\mathcal{T}_{n+1}},\mathcal{T}_{n+1}, \overline{\mathcal{T}_{n+1}} )$ as follows:
  		\begin{itemize}[label=--]
  			\item[$(\alpha)$] If $ U_{n+1} <p$: freeze $ \mathcal{V}_{n}$ in  $\overline{\mathcal{T}_{n}} $;
  			
  			If $ U_{n+1}  \geq p$: attach a new active vertex to  $ \mathcal{V}_{n}$ in $\overline{\mathcal{T}_{n}} $;
  			
			  			\item[$(\beta)$]  If  $ \mathcal{V}_{n }$ is present and active in $ \underline{\mathcal{T}_{n}}$:
  			
  			\quad If $ U_{n+1} <q$: freeze $ \mathcal{V}_{n}$ in  $\underline{\mathcal{T}_{n}} $;
  			
  			\quad If $ U_{n+1}  \geq q$: attach a new active vertex to  $ \mathcal{V}_{n}$ in $\underline{\mathcal{T}_{n}} $;

  			\item[$(\gamma)$] If   $ \mathcal{V}_{n}$ is not present or not active  in $ \mathcal{T}_{n}$: set $\sigma(n+1)\coloneqq\sigma(n)$;
  			
  			If  $ \mathcal{V}_{n }$ is present and active  in $ \mathcal{T}_{n}$: set $\sigma(n+1)\coloneqq\sigma(n)+1$,
  			  $\widetilde{\XX}_{\sigma(n)+1}\coloneqq2 \mathbbm{1}_{\{U_{n+1} \geq  r_{\sigma(n)+1} (\widetilde{\XX}_1, \ldots, \widetilde{\XX}_{\sigma(n)})\}}-1$ and  perform the following actions:
  			
  			\quad If $  U_{n+1} < r_{\sigma(n)+1} (\widetilde{\XX}_1, \ldots, \widetilde{\XX}_{\sigma(n)})$:  freeze $ \mathcal{V}_{n}$ in $ \mathcal{T}_{n}$;
  			
  			\quad If $U_{n+1} \geq r_{\sigma(n)+1}(\widetilde{\XX}_1, \ldots, \widetilde{\XX}_{\sigma(n)})  $: attach a new active vertex to $ \mathcal{V}_{n}$  in $ \mathcal{T}_{n}$.
  			
  		\end{itemize}
  		\item[(B)]If $\overline{\mathcal{T}_{n}}$  has no active vertices, set  $( \underline{\mathcal{T}_{n+1}},\mathcal{T}_{n+1}, \overline{\mathcal{T}_{n+1}} )\coloneqq ( \underline{\mathcal{T}_{n}},\mathcal{T}_{n}, \overline{\mathcal{T}_{n}} )$
  		and $\sigma(n+1)\coloneqq\sigma(n)+1$
  		and  $\widetilde{\XX}_{\sigma(n)+1}\coloneqq2 \mathbbm{1}_{\{U_{n+1} \geq  r_{\sigma(n)+1} (\widetilde{\XX}_1, \ldots, \widetilde{\XX}_{\sigma(n)})\}}-1$.
		\end{itemize}
			\item[(II)] Otherwise set $C_{n+1}=0$ and build $( \underline{\mathcal{T}_{n+1}},\mathcal{T}_{n+1}, \overline{\mathcal{T}_{n+1}} )$ as follows:
			  	  	\begin{itemize}
			\item[(A)]  If $\overline{\mathcal{T}_{n}}$  has no active vertices, set  $( \underline{\mathcal{T}_{n+1}}, \overline{\mathcal{T}_{n+1}} )\coloneqq ( \underline{\mathcal{T}_{n}}, \overline{\mathcal{T}_{n}} )$. Otherwise, choose an active vertex $\mathcal{V}_{n}$ of $\overline{\mathcal{T}_{n}} $  uniformly at random, independently of all other choices.
  		Then build $( \underline{\mathcal{T}_{n+1}}, \overline{\mathcal{T}_{n+1}} )$ as follows:
  		\begin{itemize}[label=--]
  			\item[$(\alpha)$] If $ U_{n+1} <p$: freeze $ \mathcal{V}_{n}$ in  $\overline{\mathcal{T}_{n}} $;
  			
  			If $ U_{n+1}  \geq p$: attach a new active vertex to  $ \mathcal{V}_{n}$ in $\overline{\mathcal{T}_{n}} $;
  			  			
  			\item[$(\beta)$]  If  $ \mathcal{V}_{n }$ is present and active in $ \underline{\mathcal{T}_{n}}$:
  			
  			\quad If $ U_{n+1} <q$: freeze $ \mathcal{V}_{n}$ in  $\underline{\mathcal{T}_{n}} $;
  			
  			\quad If $ U_{n+1}  \geq q$: attach a new active vertex to  $ \mathcal{V}_{n}$ in $\underline{\mathcal{T}_{n}} $;
  		\end{itemize}
  		\item[(B)] If ${\mathcal{T}_{n}}$  has no active vertices, set  $\mathcal{T}_{n+1}  \coloneqq \mathcal{T}_{n}$,
  		 $\sigma(n+1)\coloneqq\sigma(n)+1$
  		and  $\widetilde{\XX}_{\sigma(n)+1}\coloneqq2 \mathbbm{1}_{\{U_{n+1} \geq  r_{\sigma(n)+1} (\widetilde{\XX}_1, \ldots, \widetilde{\XX}_{\sigma(n)})\}}-1$. Otherwise,  choose an active vertex $\mathcal{W}_{n}$ of ${\mathcal{T}_{n}} $  uniformly at random, independently of all other choices.
 Set  $\widetilde{\XX}_{\sigma(n)+1}\coloneqq2 \mathbbm{1}_{\{U_{n+1} \geq  r_{\sigma(n)+1} (\widetilde{\XX}_1, \ldots, \widetilde{\XX}_{\sigma(n)})\}}-1$ and  perform the following actions:
  			
  			\quad If $  U_{n+1} < r_{\sigma(n)+1} (\widetilde{\XX}_1, \ldots, \widetilde{\XX}_{\sigma(n)})$:  freeze $ \mathcal{W}_{n}$ in $ \mathcal{T}_{n}$;
  			
  			\quad If $U_{n+1} \geq r_{\sigma(n)+1}(\widetilde{\XX}_1, \ldots, \widetilde{\XX}_{\sigma(n)})  $: attach a new active vertex to $ \mathcal{W}_{n}$  in $ \mathcal{T}_{n}$.
		\end{itemize}
  	\end{itemize}

  	Properties \ref{it:tree sandiwch2}, \ref{it:tree sandiwch} and \ref{it:non-decreasing trees} hold by construction. 
  	Also, observe that by (II), at any time $n\geq 1$ such that ${\mathcal{T}_{n}}$ still has at least one active vertex, $\sigma(n)$ represents the number of times an action (freezing or attachment) has modified $(\mathcal{T}_{m})_{0 \leq m \leq n}$. 
  	In particular, 
$(\mathcal{T}_{\sigma^{-1}(k)})_{k \geq 0}$ encodes the evolution of  $(\mathcal{T}_{n})_{n \geq 0}$ at steps when it changes and then remains constant after its number of active vertices reaches $0$.

  	We first check \ref{it:upper tree is Bienaymé}.
  	By construction,  $(\overline{\mathcal{T}_{n}} )_{ n \geq0}$ has the same law as $(\T_n(\overline{\XXb}))_{n\geq 0}$ with $\overline{\XX}_{n}\coloneqq2 \mathbbm{1}_{U_{n} \geq p}-1$ for $n \geq 1$. 
  	Since  $(\overline{\XX}_{n})_{n \geq1}$ are i.i.d. with $\P(\overline{\XX}_{1}=1)=1-p$,  by Theorem~2 in \cite{BBKK23+}, $\overline{\mathcal{T}_{\infty}}$ is a Bienaymé tree with offspring distribution $\mathsf{G}(p)$. 
  	Also observe that since $p>1/2$, the tree $\overline{\mathcal{T}_{\infty}}$ is almost surely finite and has no active vertices.
  	
  	Now let us establish \ref{it:middle tree has the right law}.  
  	First, the a.s.~limit $\sigma(n) \rightarrow\infty$ comes from the fact that  there exists $n \geq1$ such that $\overline{\mathcal{T}_{n}}$ has no active vertices.  Indeed,  $\sigma(n+1) =\sigma(n)$ can happen only when $\overline{\mathcal{T}_{n}}$  has at least one active vertex, and otherwise $\sigma(n+1) =\sigma(n)+1$.
  	We then show by induction on $k$ that $(\mathcal{T}_{\sigma^{-1}(i)},\widetilde{\XX}_{i})_{1 \leq i \leq k}$ and $(\mathcal{T}_{i}({\XXb}),\XX_{i})_{1 \leq i \leq k}$ have same law. 

  \emph{Base case.} 
  	Since $\underline{\mathcal{T}_{0}}, \mathcal{T}_{0}, \overline{\mathcal{T}_{0}}$ are all made of an active vertex, $ \mathcal{V}_{0}$ is that vertex, so we have $\sigma(1)=1$,  $\widetilde{\XX}_{1}=2 \mathbbm{1}_{U_{1} \geq \P(X_{1}=-1)}-1$. Thus, if $\tau_{0}$ is the tree made of a frozen vertex and $\tau_{1}$ is the tree made of two active vertices, we have $\P(( \mathcal{T}_{1},\widetilde{\XX}_{1})=(\tau_{0},-1))=\P((\mathcal{T}_{1}({\XXb}),\XX_{1})=(\tau_{0},-1))=\P(\XX_{1}=-1)$ and  $\P(( \mathcal{T}_{1},\widetilde{\XX}_{1})=(\tau_{1},1))=\P((\mathcal{T}_{1}({\XXb}),\XX_{1})=(\tau_{1},1))=\P(\XX_{1}=1)$ so that $( \mathcal{T}_{1},\widetilde{\XX}_{1})$ and $(\mathcal{T}_{1}({\XXb}),\XX_{1})$ have same law.
  	
  	\emph{Induction step.} Assume that $(\mathcal{T}_{\sigma^{-1}(i)},\widetilde{\XX}_{i})_{1 \leq i \leq k}$ and $(\mathcal{T}_{i}({\XXb}),\XX_{i})_{1 \leq i \leq k}$ have same law. 	Fix some sequence of trees $(\tau_{i})_{1 \leq i \leq k+1}$, 
  	some tree $\overline{\tau}_k$,
  	some sequence $(\X_{i})_{1 \leq i \leq k+1} \in \{-1,1\}^{k+1}$, 
  	some $c\in \{0,1\}$ 
  	and some integer $n\ge 1$, and let $E $ be the event
\begin{align*}
	E \coloneqq {
	\{C_{k}=c\} \cap
	\{\sigma^{-1}(k)=n \} \cap
	\{(\mathcal{T}_{\sigma^{-1}(i)},\widetilde{\XX}_{i})= (\tau_{i},\X_{i}) \textrm{ for } 1 \leq i \leq k \} \cap
	\{\overline{\mathcal{T}_{\sigma^{-1}(k)}}=\overline{\tau_{k}}\}.}
\end{align*}
 We {first} show that  	\begin{eqnarray}
  		&&\P \left( (\mathcal{T}_{\sigma^{-1}(k+1)},\widetilde{\XX}_{k+1})= (\tau_{k+1},\X_{k+1}) \mid E  \right)\notag
  		\\
  		&& \qquad  \qquad\qquad = \P \left( (\mathcal{T}_{k+1}({\XXb}),\XX_{k+1})= (\tau_{k+1},\X_{k+1}) \mid  (\mathcal{T}_{i}({\XXb}),\XX_{i})= (\tau_{i},\X_{i}) \textrm{ for } 1 \leq i \leq k\right) \label{eq:probas},
  	\end{eqnarray}
  provided that the events involved in the conditioning have positive probability. 
If \eqref{eq:probas} holds, it is then straightforward to check that
  	\begin{eqnarray*}
  		&&\P \left( (\mathcal{T}_{\sigma^{-1}(k+1)},\widetilde{\XX}_{k+1})= (\tau_{k+1},\X_{k+1}) \mid    (\mathcal{T}_{\sigma^{-1}(i)},\widetilde{\XX}_{i})= (\tau_{i},\X_{i}) \textrm{ for } 1 \leq i \leq k  \right)\notag
  		\\
  		&& \qquad  \qquad = \P \left( (\mathcal{T}_{k+1}({\XXb}),\XX_{k+1})= (\tau_{k+1},\X_{k+1}) \mid  (\mathcal{T}_{i}({\XXb}),\XX_{i})= (\tau_{i},\X_{i}) \textrm{ for } 1 \leq i \leq k\right) \notag,
  	\end{eqnarray*}
  	which in turn 
  	 {using the induction hypothesis}
  	 implies that $(\mathcal{T}_{\sigma^{-1}(i)},\widetilde{\XX}_{i})_{1 \leq i \leq k+1}$ and $(\mathcal{T}_{i}({\XXb}),\XX_{i})_{1 \leq i \leq k+1}$ have same law.

  	To establish \eqref{eq:probas},
  	we start with the case where $ \tau_{k}$ has no active vertices. 
  	Then the two probabilities in \eqref{eq:probas} are $0$ unless $\tau_{k+1}=\tau_{k}$.  
  	Also, by construction, on the event $E $, we have $\sigma^{-1}(k+1)=n+1$, $\mathcal{T}_{n+1}=\mathcal{T}_{n}$,  $\widetilde{\XX}_{n+1}=2 \mathbbm{1}_{U_{n+1} \geq  r_{k+1} (\X_1, \ldots, \X_{k})}-1$. 
  	Since $U_{n+1}$ is independent of $E $, it follows that
  	\begin{eqnarray*}
  		&&\P \left( (\mathcal{T}_{\sigma^{-1}(k+1)},\widetilde{\XX}_{k+1})= (\tau_{k},\X_{k+1}) \mid E  \right)\notag
  		\\
  		&& \qquad  \qquad = \P \left( 2 \mathbbm{1}_{U_{n+1} \geq  r_{k+1} (\X_1, \ldots, \X_{k})}-1=\X_{k+1}\right) =\Ppsq{\XX_{k+1}=\X_{k+1}}{\XX_1=\X_1, \ldots, \XX_{k}= \X_{k}},
  	\end{eqnarray*}
  	which is precisely equal to $ \P \left( (\mathcal{T}_{k{+1}}({\XXb}),\XX_{k+1})= (\tau_{k+1},\X_{k+1}) \mid  (\mathcal{T}_{i}({\XXb}),\XX_{i})= (\tau_{i},\X_{i}) \textrm{ for } 1 \leq i \leq k\right)$ by Algorithm \ref{algo1}.
  	
  	Now assume that $\tau_{k}$ has at least one active vertex. 
  	First,  if $c=1$ and $r_{k+1} ({\X}_1, \ldots, {\X}_{k}) \in  [p,q]$ (case (I)), observe that on the event $E $, {the tree} $\overline{\tau_{k}}$ has at least one active vertex (since by construction $\mathcal{A}_{\mathcal{T}_{n}} \subset  \mathcal{A}_{\overline{\mathcal{T}_{n}}}$ if $C_{n}=1$), so we are in step (I) (A). 
  	In particular, we have $\sigma^{-1}(k+1)= \min \{ i \geq n+1 : \mathcal{V}_{i} \in \mathcal{A}_{\tau_{k}}\}$ and $\widetilde{\XX}_{k+1}=2 \mathbbm{1}_{U_{ \sigma^{-1}(k+1)} \geq  r_{k+1} ({\X}_1, \ldots, {\X}_{k})}-1$. 
  	In addition, conditionally given $E $, by rejection sampling  $ \mathcal{V}_{\sigma^{-1}(k+1)}$ follows the uniform distribution on $ \mathcal{A}_{\tau_{k}}$ and $U_{\sigma^{-1}(k+1)} $ is a uniform random variable on $[0,1]$ independent of   $ \mathcal{V}_{\sigma^{-1}(k+1)}$. 
  	Thus, by step $(\gamma$), $\P ( (\mathcal{T}_{\sigma^{-1}(k+1)},\widetilde{\XX}_{k+1})= (\tau_{k+1},1) \mid E  )$ is the probability  that $\tau_{k+1}$ is obtained by attaching an active vertex to a random uniform active vertex of $\tau_{k}$ times the probability $\Ppsq{\XX_{k+1}=1}{\XX_1=\X_1, \ldots, \XX_{k}= \X_{k}}$, and $\P ( (\mathcal{T}_{\sigma^{-1}(k+1)},\widetilde{\XX}_{k+1})= (\tau_{k+1},-1) \mid E  )$ is the probability  that $\tau_{k+1}$ is obtained by freezing a random uniform active vertex of $\tau_{k}$ times  the probability $\Ppsq{\XX_{k+1}=-1}{\XX_1=\X_1, \ldots, \XX_{k}= \X_{k}}$. 
  	This is precisely \eqref{eq:probas}.

	Second, if $c=0$, or if $c=1$ and $r_{k+1} ({\X}_1, \ldots, {\X}_{k}) \not\in  [p,q]$, by step (II) (B), we have $\sigma^{-1}(k+1)=n+1$ and $\P ( (\mathcal{T}_{\sigma^{-1}(k+1)},\widetilde{\XX}_{k+1})= (\tau_{k+1},1) \mid E  )$ is the probability  that $\tau_{k+1}$ is obtained by attaching an active vertex to a random uniform active vertex of $\tau_{k}$ times the probability $\Ppsq{\XX_{k+1}=1}{\XX_1=\X_1, \ldots, \XX_{k}= \X_{k}}$, and $\P ( (\mathcal{T}_{\sigma^{-1}(k+1)},\widetilde{\XX}_{k+1})= (\tau_{k+1},-1) \mid E  )$ is the probability  that $\tau_{k+1}$ is obtained by freezing a random uniform active vertex of $\tau_{k}$ times  the probability $\Ppsq{\XX_{k+1}=-1}{\XX_1=\X_1, \ldots, \XX_{k}= \X_{k}}$, which is again \eqref{eq:probas}.
{This finishes the proof of the induction step, and hence that of \ref{it:middle tree has the right law}.}

Now \ref{it:lower tree is Bienaymé} is established in the same way as \ref{it:middle tree has the right law}, by constructing a sequence  $\pi(n) \rightarrow\infty$ and $(\underline{\mathcal{T}_{\pi^{-1}(k)}})_{k \geq 1}$ has the same law as $(\mathcal{T}_{k}(\underline{\XXb}))_{k \geq 1}$, where $\pi^{-1}(n)=\inf \{k \geq 0 : \pi(k) \geq n\}$ and  $(\underline{\XX}_{k})_{k \geq1}$ are i.i.d. with $\P(\underline{\XX}_{1}=1)=1-q$.
This finishes the proof of \ref{it:lem:couplage couplage}.
  	
  Now, \ref{it:lem:couplage queue hauteur} follows from \ref{it:lem:couplage couplage} and the fact that if $ \mathcal{T}$ is a Bienaymé tree with  offspring distribution $\mathsf{G}(r)$ with $r>1/2$, we have $\P(\Height(\mathcal{T}) > n)= \frac{1-s_{0}}{m^{n}-s_{0}} \cdot m^{n}$ with $m=\mathbb{E}[\mathsf{G}(r)]=1/r-1$ and $s_{0}=r/(1-r)$, see \cite[p. 9]{Har63}.
  \end{proof}

\section{Height of the infection tree}
\label{sec:height}

{In this section,}
we shall prove our main result, Theorem~\ref{thm:main}, assuming a limit theorem for the profile of the infection tree (Proposition~\ref{prop: profil faible}, stated in Section \ref{subsec:profile of the infection tree}).
We first gather some useful ingredients pertaining to the asymptotic behavior of the evolution of susceptible and infectious individuals (Section~\ref{subsec:fluid limit}) and analytic properties concerning the Lambert function (Section~\ref{ssec:crit}).

\subsection{Fluid limit}
\label{subsec:fluid limit}

The so-called \emph{fluid limit} of the processes $I^{n}$ and $H^{n}$ involves the  solution $g_\lambda$ of the ordinary differential equation $g_{\lambda}'(t)=-\frac{\lambda g_{\lambda}(t)}{1+\lambda g_{\lambda}(t)}$ with $g_\lambda(0)=1$. 
Recall that $W$ is the principal branch of the Lambert function, which satisfies $W(x) e^{W(x)}=x$ for $x \geq -\frac{1}{e}$. 
It is also the solution of the differential equation $W'(t)= \frac{W(t)}{t(1+W(t))}$ with $W'(0)=1$. 
This readily implies that
\begin{equation}
\label{eq:glambda}
g_{\lambda}(t)= \frac{1}{\lambda} W \left(  \lambda  e^{\lambda}e^{- \lambda t} \right), \qquad t \geq 0.
\end{equation}Set
\begin{align*} t_{\lambda} \coloneqq \inf \{t \geq 0: 2-2g_{\lambda}(t)-t=0 \}.
\end{align*} 
The fact that $t_{\lambda}$ is well defined comes e.g.~from the fact that $h(t)\coloneqq2-2g_\lambda(t)-t$ for $t \geq 0$ defines a concave function (this can be seen by differentiating) with $h'(0)=2 \lambda-1>0$ and $h(t) \rightarrow -\infty$ as $t \rightarrow \infty$ (since $W(0)=0$).

Recall that $\mathcal{B}$  is a Bernoulli random variable of parameter $1-\frac{1}{\lambda}$. 
In the proof of Theorem~24 in \cite{BBKK23+} the following is established for every $\delta \in (0,1)$:
\begin{eqnarray}
&& \left(
\left(\frac{I^n_{\lfloor nt \rfloor}}{n}: t \ge 0\right),\left(\frac{H^n_{\lfloor nt \rfloor}}{n} :  0 \leq t \leq t_{\lambda} \right) , {\mathbbm{1}_{\tau'_{n} \geq (1-\delta) t_{\lambda} n}}
\right) \notag \\
&& \qquad \qquad \qquad \qquad   \cvloi 
\left(
\left( \max(2-2g_{\lambda}(t)-t,0)  \mathcal{B} : t \geq 0\right),
( g_{\lambda}(t) \mathcal{B} : {0 \leq t \leq t_{\lambda}}), { \mathcal{B} }
\right),\label{eq:limflu}
\end{eqnarray}
 where the functional convergence is understood for the topology of uniform convergence on compact sets. In particular, $t_{\lambda}$ can be thought of as the extinction time of the fluid limit of $I^{n}$.

\subsection{The critical value of $\lambda$}
\label{ssec:crit}

Here we establish several analytical properties, including  the existence of $\lambda_{c}$ defined by \eqref{eq:lc}. 
 The proof of Proposition~\ref{prop:lc} is analytical and technical, and can be skipped at the first reading. 
 Recall from \eqref{eq:param} the definitions of $m_{\lambda}$ and $z_{\lambda}$:
\begin{align*} 
m_{\lambda}=-W(-\lambda e^{-\lambda}), \qquad z_{\lambda}=\inf \{t >0: f_{\lambda}(t)=0\}=1+W \left( -\frac{1}{e \lambda}\right).
\end{align*} 
 To simplify notation, for $x \geq 0$ we set
\[
h_{\lambda}(x)\coloneqq \frac{ f_{\lambda}(x)}{-\log m_{\lambda}}.
\]

\begin{proposition}
\label{prop:lc}
The following assertions hold.
\begin{enumerate}[label=(\roman*)]
\item \label{it:prop:lc:existence critical lambda}
 There exists a unique $\lambda \in \intervalleoo{1}{\infty}$ that satisfies $m_{\lambda}=e^{-z_{\lambda}}$, which we denote by $\lambda_c$.
In addition, $\lambda < \lambda_{c}$ implies  $m_{\lambda}>e^{-z_{\lambda}}$ and $\lambda>\lambda_{c}$ implies $ m_{\lambda}<e^{-z_{\lambda}}$.
\item\label{it:prop:lc:value kappa} We have
\[
 \sup_{0 \leq s \leq z_{\lambda}}  \left( \frac{\lambda}{\lambda-1} e^{s}+ h_{\lambda}(s) \right)= \begin{cases}
    \frac{\lambda}{(\lambda-1)m_{\lambda}}+  {h_{\lambda}(-\log m_{\lambda})}  & \textrm{if } \lambda \leq \lambda_{c},\\
\frac{\lambda}{\lambda-1} e^{z_{\lambda}} &  \textrm{if } \lambda \geq \lambda_{c}. 
\end{cases}
  \]
 \end{enumerate}
\end{proposition}

We will use the following lower bound on the Lambert function.

\begin{lemma}
\label{lem:Lambert}
The following assertions hold.
\begin{enumerate}[label=(\roman*)]
\item \label{it:lem:Lambert minoration W(x)}  For every $-1/e \leq x \leq 0$ we have
\begin{align*} 
W(x) \geq -1+\sqrt{2e}\sqrt{x+\frac{1}{e}}- \frac{2}{3}e \left( x+ \frac{1}{e}\right).
\end{align*} 
\item \label{it:lem:Lambert minoration W(lambda e lambda)} For every $\lambda \geq 1$,
\begin{align*} W(-\lambda e^{-\lambda }) \geq (\lambda -1) \sqrt{2-2\lambda +\lambda ^{2}}-2+2\lambda -\lambda ^{2}.\end{align*} 
\end{enumerate}

\end{lemma}
{In Lemma~\ref{lem:Lambert}}, the right-hand side of \ref{it:lem:Lambert minoration W(x)} corresponds to the first three terms of the asymptotic expansion of $W$ at $-\frac{1}{e}$ (see e.g.~\cite[Eq.~(4.22)]{CGHJK96}). 
The bound 
\ref{it:lem:Lambert minoration W(x)} is also better than the bound $W(x) \geq \sqrt{ex+1}-1$ obtained in \cite[Theorem~2.2]{RS20} in the vicinity of $-\frac{1}{e}$, see Figure~\ref{fig:bounds}. 
The bound of \cite[Theorem~2.2]{RS20} is not good enough for the proof of Proposition~\ref{prop:lc}. 
The proof of Lemma~\ref{lem:Lambert} is quite technical and is deferred to the appendix.

\begin{figure}
\centering
\includegraphics[scale=0.4]{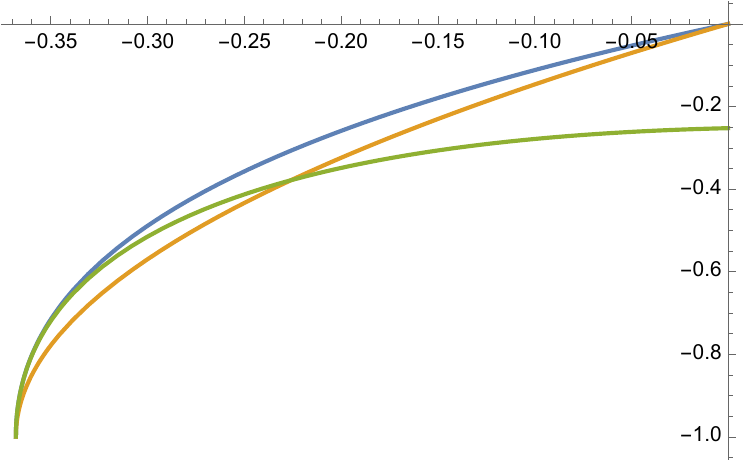}%
\caption{\label{fig:bounds}In blue the Lambert function, in orange the lower bound of \cite[Theorem~2.2]{RS20} and in green the lower bound of Lemma~\ref{lem:Lambert}.}
\end{figure}

\begin{proof}[Proof of Proposition~\ref{prop:lc}]
We start with the proof of \ref{it:prop:lc:existence critical lambda}. 
Consider the function $u$ defined on $\intervallefo{1}{\infty}$ by 
\begin{align*}
u(\lambda) \coloneqq \frac{1}{m_{\lambda}}-e^{z_{\lambda}} \underset{\eqref{eq:param}}{=} \frac{1}{- W(-\lambda e^{-\lambda})} - e^{1+ W(-\frac{1}{e\lambda})}.  
\end{align*}
We start by showing that $u$ is convex on $\intervallefo{1}{\infty}$.
First  we compute 
\begin{align*} 
	\frac{\mathrm{d^{2}}}{\mathrm{d} \lambda^{2}}  \left( - e^{z_{\lambda}} \right) = \frac{2+3 W \left( -\frac{1}{e \lambda}\right)+2 W \left( -\frac{1}{e \lambda}\right)^{2}}{\lambda^{3} \left(1+W \left( -\frac{1}{e \lambda}\right)\right)^{3}}.
\end{align*} 
The denominator is positive for all $\lambda>1$ and the discriminant of the polynomial $2+3X+2X^2$ is negative so the numerator is always non-negative, hence the second derivative that is considered here is always non-negative.  
Now, we write
\begin{align*} 
	\frac{\mathrm{d^{2}}}{\mathrm{d} \lambda^{2}}  \frac{1}{m_{\lambda}}
	=
	\frac{e^{\lambda+W(-\lambda e^{-\lambda }})} {\lambda ^{3} \left(1+W(-\lambda e^{-\lambda }) \right)^{3}} 
	\left(\left(W(-\lambda e^{-\lambda })+2-2 \lambda + \lambda ^{2}\right)^{2}-(\lambda -1)^{2}(2-2\lambda +\lambda ^{2})\right),
\end{align*} 
and we can check that the RHS is non-negative for any $\lambda>1$ using  Lemma~\ref{lem:Lambert}\ref{it:lem:Lambert minoration W(lambda e lambda)}.
Combining the two previous displays, we get that $\frac{\mathrm{d^{2}}}{\mathrm{d} \lambda^{2}} u(\lambda) \geq 0$ for all $\lambda>1$, so that $u$ is convex. 

Now, from the fact that $W(\frac{-1}{e})=-1$ and $W(0) = 0$ we can check that 
\begin{equation}\label{eq:values u 0 and infinity}
	u(1)=0 \qquad \text{and} \qquad u(\lambda) \underset{\lambda \rightarrow \infty}{\rightarrow} \infty.
\end{equation}
Also, from the expansion $W(x)=-1+\sqrt{2} \sqrt{1+e x} + o(1+ex)$ as $x \rightarrow -1/e$ (see e.g.~\cite[Eq.~(4.22)]{CGHJK96}) which yields $u(1+h)=-\sqrt{2h}+o(\sqrt{h})$ as $h \rightarrow 0$, we get that $u'(1)=-\infty$. 
This combined with \eqref{eq:values u 0 and infinity} and the fact that $u$ is convex ensures that there exists a unique $\lambda=\lambda_c$ that satisfies $u(\lambda)=0$, and which is so that $\lambda <\lambda_c$ implies $u(\lambda)<0$ and $\lambda >\lambda_c$ implies $u(\lambda)>0$. 
This easily implies \ref{it:prop:lc:existence critical lambda}.

We now turn to the proof of \ref{it:prop:lc:value kappa}. 
Observe that 
\[
\frac{\mathrm{d}}{\mathrm{d}s}  \left( \frac{\lambda}{\lambda-1} e^{s}+ h_{\lambda}(s)\right)=\frac{\lambda e^{s} (s+\log m_{\lambda})}{({\lambda-1}) \log m_{\lambda}},
\]
and that $m_\lambda<1$. 
This ensures
 that the function $s \mapsto \frac{\lambda}{\lambda-1} e^{s}+ h_{\lambda}(s)$ is increasing on $[0,-\log m_{\lambda}]$ and decreasing on $[-\log m_{\lambda},\infty)$,
 so that its supremum over the interval $\intervalleff{0}{z_\lambda}$ is  either attained for $s= -\log m_{\lambda}$ in the case where $z_\lambda \geq -\log m_{\lambda}$, or for $s=z_\lambda$ in the case where $z_\lambda \leq -\log m_{\lambda}$.
Hence
\[
 \sup_{0 \leq s \leq z_{\lambda}}  \left( \frac{\lambda}{\lambda-1} e^{s}+ h_{\lambda}(s) \right)= \begin{cases}
 \left(   \frac{\lambda}{(\lambda-1)m_{\lambda}}+  {h_{\lambda}(-\log m_{\lambda})} \right) & \textrm{if } (-\log m_{\lambda}) \leq z_{\lambda},\\
\frac{\lambda}{\lambda-1} e^{z_{\lambda}} &  \textrm{if } (-\log m_{\lambda}) \geq z_{\lambda}. 
\end{cases}
 \]
 and the conclusion follows by \ref{it:prop:lc:existence critical lambda}.
\end{proof}

Observe that the proof of Proposition~\ref{prop:lc} shows that when  $\lambda<\lambda_{c}$,
\begin{equation}
\label{eq:inegl}
\frac{\lambda}{(\lambda-1)m_{\lambda}}+  {h_{\lambda}(-\log m_{\lambda})}\geq \frac{\lambda}{\lambda-1} e^{z_{\lambda}}.
\end{equation}
Indeed, when $\lambda<\lambda_{c}$ we have $ -\log m_{\lambda} \leq z_{\lambda}$, and since  $s \mapsto \frac{\lambda}{\lambda-1} e^{s}+ h_{\lambda}(s)$ is decreasing on {$\intervallefo{-\log m_{\lambda}}{\infty}$}, we have
\[
\frac{\lambda}{(\lambda-1)m_{\lambda}}+  {h_{\lambda}(-\log m_{\lambda})} \leq\frac{\lambda}{(\lambda-1)}e^{z_{\lambda}}+  {h_{\lambda}( z_{\lambda })}= \frac{\lambda}{\lambda-1} e^{z_{\lambda}},
\]
where we have used the fact that $h_{\lambda}(z_{\lambda})=0$.

\subsection{Profile of the infection tree}\label{subsec:profile of the infection tree}
An important tool in the proof of Theorem~\ref{thm:main} will be a limit theorem for the profile of the infection tree.
For $n \geq1$ and $k,h \geq0$, we let
\begin{align*}
\mathbb{A}_{k}^{n}(h)\coloneqq\#\{\text{active vertices at height $h$ at time $k$ of $\mathcal{T}^{n}_{k}$}\} \qquad \textrm{and} \qquad 	\bL_k^n(h)\coloneqq\frac{\mathbb{A}_{k}^{n}(h)}{I_k^n}
\end{align*}
be respectively the \emph{active profile} of the tree $ \mathcal{T}^{n}_{k}$ and its normalized version. For all $a,b \ge 0$, we write $\mathbb{A}^n_k([a,b]) = \sum_{a \le h \le b} \mathbb{A}^n_k(h)$ and $\mathbb{L}^n_k([a,b]) = \sum_{a \le h \le b} \mathbb{L}^n_k(h)$. 
We also set $\mathbb{A}^n_k([a,\infty]) = \sum_{a \le h } \mathbb{A}^n_k(h)$ and $\mathbb{L}^n_k([a,\infty]) = \sum_{a \le h } \mathbb{L}^n_k(h)$.

Recall from \eqref{eq:param} the notation \begin{align*} f_{\lambda}(z)=  1+ \frac{\lambda}{1-\lambda} (e^{z}-1-z e^{z}),\end{align*} 
and  $z_{\lambda}=\inf \{t >0: f_{\lambda}(t)=0\}$. 
{The following proposition ensures a rough control on the profile that will be sufficient to prove our main result. A stronger local version is stated in Section~\ref{sec:profile} (Theorem~\ref{thm:profil}).}

\begin{proposition}
\label{prop: profil faible}
	Let $\lambda>1$. 
	Assume that $\lambda_n \sim \lambda/n$ as $n \to \infty$. 
	Fix $t \in (0,t_{\lambda})$. 
	Set $\gamma=\lambda/(\lambda-1)$. 
	Then, for all $0<x< z_\lambda$ and $y \in (x,\infty]$ the following convergence holds in probability as $n \rightarrow\infty$:
		\[\ind{\tau'_{n} \geq  \lfloor n t\rfloor}\cdot \left(
	\frac{\log \mathbb{A}^{n}_{\lfloor nt \rfloor} ([\gamma e^x \log n, \gamma e^y \log n])}{\log n} - f_{\lambda}(x) \right)   \quad \mathop{\longrightarrow}^{(\P)}_{n \rightarrow \infty} \quad 0.
	\]
\end{proposition}

This is the main input to establish Theorem~\ref{thm:main}: taking  Proposition~\ref{prop: profil faible}   for granted, we shall now  see how this implies Theorem~\ref{thm:main}.

\subsection{Height of the dangling trees}

The idea is to control separately on the one hand the height of the infection tree at a time after the early stages of the epidemic and before the late stages of the epidemic (when, at the same time, a positive fraction of infectious individuals and a positive fraction of healthy individuals remain), and on the other hand the heights of the {outgrowths from all the active vertices of the tree, which contain all the vertices that joined the tree after that time.}

The height of the infection tree after the early stages of the epidemic and before the late stages of the epidemic  is given by the following convergence, where we recall that $\gamma = \lambda/(\lambda-1)$.
\begin{lemma}
\label{lem:limHbefore} 
For every $\delta \in (0,1)$ we have
\[
 \left(\frac{\Height( \mathcal{T}_{\lfloor (1-\delta) t _{\lambda} n \rfloor}(\XXb^ {n} ))}{\log n}, {\mathbbm{1}_{\tau'_{n} \geq  \lfloor (1-\delta) t_{\lambda} n  \rfloor}} \right)  \quad \mathop{\longrightarrow}^{(d)}_{n \rightarrow \infty} \quad   \left(\gamma e^{z_{\lambda}} \mathcal{B}, {\mathcal{B}}\right).
\]
\end{lemma}
\begin{proof}
The proof of Theorem~24(1)(b) in \cite{BBKK23+} shows that $(\Height(\T^{n}_{ \lfloor (1-\delta) t _{\lambda} n \rfloor} )/ (\log n),{\mathbbm{1}_{\tau'_{n} \geq  \lfloor (1-\delta) t_{\lambda} n  \rfloor}})$ converges in distribution to  $( \frac{1+c}{2c} u(c) \mathcal{B},{\mathcal{B}})$, where $u(c)$ is the unique solution of $u(c) (\log u(c)-1)= (c-1)/(c+1)$ with $c= (\lambda-1)/(\lambda+1)$.  
Observing that $ (1+c)/(2c)=\lambda/(\lambda-1)$ and $(c-1)/(c+1)=-1/\lambda$, we check that $u(c)=e^{z_{\lambda}}$ by showing that $e^{z_{\lambda}}$ is solution of $x( \log(x)-1)=-1/\lambda$.  
This readily comes from the fact that by definition
\begin{align*} 1+ \frac \lambda{\lambda-1}(e^{z_{\lambda}}-1-z_{\lambda} e^{z_{\lambda}})=0,
\end{align*} 
which implies $e^{z_{\lambda}}(z_{\lambda}-1)=-\frac{1}{\lambda}$.
This completes the proof.
\end{proof}

\uwave{\hspace{0.9\linewidth}}

\textbf{In the sequel, we assume that all the random variables depending on $n$ are defined on a same probability space.}

\raisebox{0.8em}{\uwave{\hspace{0.9\linewidth}}}

To control the trees grafted on $\T^{n}_{ \lfloor (1-\delta) t_{\lambda} n\rfloor} $ for fixed $\delta \in (0,1/2)$  when $\tau'_{n} \geq  \lfloor (1-\delta) t_{\lambda}  n \rfloor$, we denote by $ \mathcal{A}^{n}_{\delta}$ the set of all active vertices of  $ \mathcal{T}_{\lfloor (1-\delta) t _{\lambda} n \rfloor}(\XXb^ {n} )$ and observe that $\#  \mathcal{A}^{n}_{\delta}=I^n_{\lfloor (1-\delta) t _{\lambda} n \rfloor}$. 
For every $u \in  \mathcal{A}^{n}_{\delta}$, we denote by $T^{n}_{\delta}(u)$ the tree made of 
{the vertex $u$ together with all the descendants of $u$ in  $\mathcal{T}_{\tau'_{n}}(\XXb^n)$ that were added to the tree after time $\lfloor (1-\delta) t_{\lambda} n\rfloor$.}  

For $x,y \in \R \cup \{ {\pm} \infty\}$ with $x<y$, to simplify notation let
  \[ \mathcal{A}^{n}_{\delta}(x,y)\coloneqq \{u \in \mathcal{A}^{n}_{\delta}:    \gamma e^{x} \log n   \leq {\haut(u)} \leq   \gamma e^{y} \log n\}, \quad \textrm{and} \quad  
  H^{n}_{\delta}(x,y)\coloneqq \max_{ u \in \mathcal{A}^{n}_{\delta}(x,y) } \Height(T^{n}_{\delta}(u))
  \]
  be the maximal height of a tree grafted on an active vertex of $ \mathcal{A}^{n}_{\delta}$ with height  belonging to $[\gamma e^x  \log (n), \gamma e^y  \log (n)]$. 
  Finally, to simplify notation, set
 \begin{align*}
  \mathcal{T}^{n,\delta}\coloneqq\mathcal{T}_{\lfloor (1-\delta) t _{\lambda} n \rfloor}(\XXb^ {n} ).  
\end{align*}  
  Recall from \eqref{eq:param} the definitions of $f_{\lambda}$ and of $m_{\lambda}$ and from \eqref{eq:glambda} the definition of $g_{\lambda}$.  Note that $m_{\lambda}<1$ (this comes from the explicit expression of $m_{\lambda}$). Also {observe that the event $ \{\tau'_{n} \geq  \lfloor (1-\delta) t_{\lambda} n  \rfloor\}$  is measurable with respect to $ \mathcal{T}^{n,\delta}$.} 
  
    \begin{lemma}
  \label{lem:dangling}
For every $x \in (0,z_{\lambda})$ and {$y \in (x,\infty]$}, for all $\e>0$, for every $\delta \in \intervalleoo{0}{1}$ small enough,
\begin{align*}
\mathbbm{1}_{\tau'_{n} \geq  \lfloor (1-\delta) t_{\lambda} n  \rfloor} \cdot \Ppsq{
	\left| \frac{1}{\log n}   H^{n}_{\delta}(x,y) -  \frac{ f_{\lambda}(x)}{-\log m_{\lambda}}  \right| \geq  \varepsilon
}{ \mathcal{T}^{n,\delta}}  \quad \mathop{\longrightarrow}^{(\P)}_{n \rightarrow \infty} \quad 0
\end{align*}
and
\begin{align*}
\mathbbm{1}_{\tau'_{n} \geq  \lfloor (1-\delta) t_{\lambda} n  \rfloor} \cdot \Ppsq{
	\left|  \frac{1}{\log n}   H^{n}_{\delta}(-\infty,x)  -   \frac{ 1}{-\log m_{\lambda}}  \right| \geq  \varepsilon
}{ \mathcal{T}^{n,\delta}}  \quad \mathop{\longrightarrow}^{(\P)}_{n \rightarrow \infty} \quad 0.
\end{align*}
  \end{lemma}

We will need the following relation between $m_{\lambda}$, $g_{\lambda}$ and $t_{\lambda}$.

\begin{lemma}
\label{lem:mt}
We have $m_{\lambda}=\lambda g(t_{\lambda})$.
\end{lemma}

\begin{proof}
Using the identities $W \left( x e^{-2 W(-x)}\right)=-W(-x)$ for $-1/e \leq x \leq 0$ and $g_{\lambda}(t_{\lambda})=1-t_{\lambda}/2$, we readily get that 
\[
t_{\lambda}= 2+ \frac{2}{\lambda} W(-\lambda e^{-\lambda}) \qquad \textrm{and thus} \qquad m_{\lambda}=-W(-\lambda e^{-\lambda})=\lambda g_{\lambda}(t_{\lambda}).
\]
This completes the proof.
\end{proof}

Before proving Lemma~\ref{lem:dangling} we need to introduce some more notation. 
Take $\delta \in (0,1/2)$. 
Let $\mathbb{T}^{n,\delta}$ be the support of the random variable $ \mathcal{T}^{n,\delta}$ conditionally given $\{\tau'_{n} \geq  \lfloor (1-\delta) t_{\lambda} n  \rfloor\}$.  
Now fix $T \in {\mathbb{T}}^{n,\delta}$. 
Denote by $\P_{n,\delta,T}$  the conditional probability distribution $\P( \, \cdot  \mid \mathcal{T}^{n,\delta}=T)$. 
We set $\XX^{n,\delta}_{i}= X_{\lfloor (1-\delta) t_{\lambda} n  \rfloor+i}$ for $i \geq 1$ and $\XXb^{n,\delta}=(\XX^{n,\delta}_{i})_{i \geq 1}$. 
For all $k \in \N$, for all $\X_1, \ldots, \X_{k-1}\in \{\pm1\}$ 
such that $\P_{{n,\delta,T}}(\XX^{n,\delta}_1=\X_1, \ldots, \XX^{n,\delta}_{k-1}= \X_{k-1})>0$, 
  set 
\[
r^{n,\delta,T}_{k}(\X_1, \ldots, \X_{k-1})\coloneqq\Ppsqndt{\XX^{n,\delta}_{ k}=-1}{\XX^{n,\delta}_{ 1}=\X_1, \ldots, \XX^{n,\delta}_{ k-1}= \X_{k-1}}.
\]
Finally, for $\eta>0$ set
	\[
	\mathcal{E}^{n,\delta,T,\eta}\coloneqq \left\{ \forall k  \in \llbracket 1,\tau(\XXb^{n,\delta})-1 \rrbracket, \   \frac{1}{1+m_{\lambda} +\eta} \leq r^{n,\delta,T}_{k}(\XX^{n,\delta}_1, \ldots, \XX^{n,\delta}_{k-1}) \leq  \frac{1}{1+ m_{\lambda}  -\eta}  \right\}.  
\]

\begin{lemma}
\label{lem:eventE}
For every $\eta>0$ and for every $\delta \in (0,1/2)$ small enough, 
there exists a subset $\overline{\mathbb{T}}^{n,\delta} \subset \mathbb{T}^{n,\delta}$ such that
\begin{equation*}
\Ppsq{\mathcal{T}^{n,\delta} \in  \overline{\mathbb{T}}^{n,\delta}}{ \tau'_{n} \geq  \lfloor (1-\delta) t_{\lambda} n \rfloor }  \quad \mathop{\longrightarrow}_{n \rightarrow \infty} \quad 1 
 \qquad \textrm{and} \qquad 
 \min_{T \in \overline{\mathbb{T}}^{n,\delta}}  \P_{n,\delta,T}(\mathcal{E}^{n,\delta,T,\eta})  \quad \mathop{\longrightarrow}_{n \rightarrow \infty} \quad 1.
\end{equation*} 
\end{lemma}

\begin{proof}
Fix $\eta>0$. By using Lemma~\ref{lem:mt} and the continuity of $g_{\lambda}$ at $t_{\lambda}$, choose $\delta \in (0,1/2)$ such that $\lambda g_{\lambda}((1-\delta)t_{\lambda}) \geq m_{\lambda}-\eta/2$.

Observe that for every $T\in \mathbb{T}^{n,\delta}$, under $\P_{n,\delta,T}$, we have
\begin{align*}
r^{n,\delta,T}_{k}(\XX^{n,\delta}_1, \ldots, \XX^{n,\delta}_{k-1}) 
&= \frac{1}{1+\lambda_n H^n_{\lfloor (1-\delta) t_{\lambda} n\rfloor+k-1}}.
\end{align*}
Thus
\begin{multline}
	\Ppsq{ \forall k \in \intervalleentier{1}{\tau(\XXb^{n,\delta})-1} ,\ \frac{1}{1+m_{\lambda} +\eta} \leq \frac{1}{1+\lambda_n H^n_{\lfloor (1-\delta) t_{\lambda} n\rfloor+k-1}} \leq \frac{1}{1+m_{\lambda} -\eta} }{\tau'_{n} \geq  \lfloor (1-\delta) t_{\lambda} n  \rfloor}\\
	= \frac{1}{\P(\tau'_{n} \geq  \lfloor (1-\delta) t_{\lambda} n  \rfloor)}\sum_{T\in \mathbb{T}^{n,\delta}} \Pndt{ \mathcal{E}^{n,\delta,T,\eta} } \Pp{\mathcal{T}^{n,\delta}=T}
	\label{eq:inegH}
\end{multline}
By the fluid limit result \eqref{eq:limflu} under the conditional probability $\P \left( \, \cdot \mid \tau'_{n} \geq  \lfloor (1-\delta) t_{\lambda} n  \rfloor \right)$ we have the convergence
$$
\left(  \frac{H^n_{\lfloor (1-\delta) t_{\lambda} n\rfloor+\lfloor nt \rfloor }}{n} : 0 \leq t \leq \delta t_{\lambda} \right)  \quad \mathop{\longrightarrow}^{(\P)}_{n \rightarrow \infty} \quad g_{\lambda}((1-\delta) t_{\lambda} + t : 0 \leq t \leq \delta t_{\lambda}),
$$
so by our choice of $\delta$, the LHS of \eqref{eq:inegH} tends to $1$.

Now, for a fixed $\varepsilon >0$, consider the set 
\begin{align*}
	\widetilde{\mathbb{T}}^{n,\delta,\varepsilon} 
	= \enstq{T \in \mathbb{T}^{n,\delta}}{ 	\P_{n,\delta,T}(\mathcal{E}^{n,\delta,T,\eta}) \geq 1-\varepsilon}
\end{align*}
and write
\begin{align*}
	 \sum_{T\in \mathbb{T}^{n,\delta} }  \Ppp{n,\delta,T}{\mathcal{E}^{n,\delta,T,\eta} } \cdot \Pp{\mathcal{T}^{n,\delta}=T} 
	&\leq   \Pp{\mathcal{T}^{n,\delta}\in \widetilde{\mathbb{T}}^{n,\delta,\varepsilon} } + (1-\varepsilon) \cdot \Pp{\mathcal{T}^{n,\delta}\in \mathbb{T}^{n,\delta} \setminus \widetilde{\mathbb{T}}^{n,\delta,\varepsilon}}\\
	& = \Pp{\mathcal{T}^{n,\delta}\in \mathbb{T}^{n,\delta}} -  \varepsilon \Pp{\mathcal{T}^{n,\delta}\in \mathbb{T}^{n,\delta} \setminus \widetilde{\mathbb{T}}^{n,\delta,\varepsilon}}\\
	&=\P(\tau'_{n} \geq  \lfloor (1-\delta) t_{\lambda} n  \rfloor) -  \varepsilon \Pp{\mathcal{T}^{n,\delta}\in \mathbb{T}^{n,\delta} \setminus \widetilde{\mathbb{T}}^{n,\delta,\varepsilon}}.
\end{align*}
We have already seen that the quantity \eqref{eq:inegH} converges to $1$ as $n \rightarrow \infty$, and since $\P(\tau'_{n} \geq  \lfloor (1-\delta) t_{\lambda} n  \rfloor) \rightarrow 1-1/\lambda$, we conclude that the term $\P(\mathcal{T}^{n,\delta}\in \mathbb{T}^{n,\delta} \setminus \widetilde{\mathbb{T}}^{n,\delta,\varepsilon})$ has to go to $0$ for $\varepsilon>0$ fixed. 
We can then choose a sequence $\varepsilon_n \rightarrow 0$ so that  $\P(\mathcal{T}^{n,\delta}\in \mathbb{T}^{n,\delta} \setminus \widetilde{\mathbb{T}}^{n,\delta,\varepsilon_n}) \rightarrow 0$ as $n\rightarrow \infty$. 
This ensures that the choice $\overline{\mathbb{T}}^{n,\delta} \coloneqq \widetilde{\mathbb{T}}^{n,\delta,\varepsilon_n}$ satisfies the statement of the lemma.
\end{proof}

We are now ready to establish Lemma~\ref{lem:dangling}.

\begin{proof}[Proof of Lemma~\ref{lem:dangling}] Fix  $x \in (0,z_{\lambda})$, {$y \in (x,\infty]$} and $\varepsilon>0$. Let $\eta>0$ be such that
\begin{equation}
\label{eq:eta}
\frac{-\log(m_\lambda - \eta)}{-\log(m_\lambda)}(f_{\lambda}( x)-3 \varepsilon) \leq f_\lambda(x)-2\e \quad\textrm{ and } \quad   f_\lambda(x)+2\e \leq \frac{-\log(m_\lambda + \eta)}{-\log(m_\lambda)}(f_{\lambda}( x)+3 \varepsilon) .
\end{equation}
Take  $\delta \in (0,1/2)$  small enough so that the conclusion of Lemma~\ref{lem:eventE} holds with the subset $\overline{\mathbb{T}}^{n,\delta} \subset \mathbb{T}^{n,\delta}$.

For every tree $T$, recall that $ \mathcal{A}_{T}$ stands for the set of all active vertices of $T$, and define
 $$\mathcal{A}^{n}_{T}(x,y)= \{u \in \mathcal{A}_{T}:    \gamma e^{x} \log n   \leq {\haut(u)} \leq   \gamma e^{y} \log n\}.$$
By  Proposition~\ref{prop: profil faible} and Lemma~\ref{lem:eventE}, if we define
\[
\widehat{\mathbb{T}}^{n,\delta}= \left\{T \in \overline{\mathbb{T}}^{n,\delta}  :  n^{f_{\lambda}(x)-\varepsilon}  \leq \# \mathcal{A}^{n}_{T}(x,y) \leq n^{f_{\lambda}(x)+\varepsilon} \right\},
\]
then 
\begin{equation}
\label{eq:That}
 \Ppsq{\mathcal{T}^{n,\delta} \in  \widehat{\mathbb{T}}^{n,\delta} }{\tau'_{n} \geq  \lfloor (1-\delta) t_{\lambda} n } \quad \mathop{\longrightarrow}_{n \rightarrow \infty} \quad 1  
 \qquad \textrm{and} \qquad 
 \min_{T \in \widehat{\mathbb{T}}^{n,\delta}}  \P_{n,\delta,T}\left(\mathcal{E}^{n,\delta,T,\eta}\right)  \quad \mathop{\longrightarrow}_{n \rightarrow \infty} \quad 1.
\end{equation}

Now take $T \in \widehat{\mathbb{T}}^{n,\delta}$. 
Set $N\coloneqq \#\mathcal{A}_{T}$ and let $\mathcal{A}_{T}=\{u_1,u_2,\dots, u_{N}\}$ be the enumeration of the active vertices in their order of appearance in the tree.  
Under $\P_{n,\delta,T}$, for every $u \in  \mathcal{A}_{T}$, recall that we denote by $T^{n}_{\delta}(u)$ the tree made of  
the vertex $u$ together with all the descendants of $u$ in  $\mathcal{T}_{\tau'_{n}}(\XXb^n)$ that were added to the tree after time $\lfloor (1-\delta) t_{\lambda} n\rfloor$.  
Note that under $\P_{n,\delta,T}$ we have the following equality in distribution for forests
	\begin{align*}
		\left(T^{n}_{\delta}(u_i) :  \ i \in \{1,2,\dots ,N\}\right) \text{ has the same distribution as }	\mathcal{F}^N_\infty(\XXb^{n,\delta}), 
	\end{align*}
	as defined by Algorithm~\ref{algo1}. 
	Thus, by Lemma~\ref{lem:couplage},  under $\P_{n,\delta,T}$, we can couple $(T^{n}_{\delta}(u) : u \in \mathcal{A}_T)$ with  two families of independent Bienaymé trees
  $(\overline{T}^{n}_{\delta}(u) : u \in \mathcal{A}_{T} )$ and $(\underline{T}^{n}_{\delta}(u) : u \in  \mathcal{A}_{T} )$ with respective offspring distributions  $\mathsf{G}(\frac{1}{1+m_{\lambda} +\eta})$  and  $\mathsf{G}(\frac{1}{1+m_{\lambda} -\eta})$ 
    such that on the event $\mathcal{E}^{n,\delta,T,\eta}$,  we have $\underline{T}^{n}_{\delta}(u) \subset T^{n}_{\delta}(u) \subset \overline{T}^{n}_{\delta}(u)$ for every $u  \in \mathcal{A}_{T}$.

  For the first statement, we show that the convergence 
  \begin{equation}
  \label{eq:cvT}
\P_{n,\delta,T}\left(
	 \left| \frac{1}{\log n}   H^{n}_{\delta}(x,y) -  \frac{ f_{\lambda}(x)}{-\log m_{\lambda}}   \right| >   \frac{ 3\varepsilon}{-\log m_{\lambda}} 
\right)  \quad \mathop{\longrightarrow}_{n \rightarrow \infty} \quad \quad 0.
\end{equation}
holds uniformly in $T \in \widehat{\mathbb{T}}^{n,\delta}$, which implies the desired result.
 
Take   $T \in \widehat{\mathbb{T}}^{n,\delta}$.  
We first show the lower bound. 
By Lemma~\ref{lem:couplage}\ref{it:lem:couplage queue hauteur}, for every $u \in \mathcal{A}_{T}$,
\begin{align*}
\Pndt{\Height(\underline{T}^{n}_{\delta}(u)) \geq   \frac{f_{\lambda}(x)  - 3\varepsilon}{-\log m_{\lambda}} \cdot \log n}
\geq  \frac{1}{C}n^{- \frac{-\log(m_\lambda - \eta)}{-\log(m_\lambda)}(f_{\lambda}( x)-3 \varepsilon)} \geq  \frac{1}{n^{f_\lambda(x)-2\e}}
\end{align*}
for $n$ large enough (uniformly in $T \in \widehat{\mathbb{T}}^{n,\delta}$).
Then, using the fact that $n^{f_{\lambda}(x)-\varepsilon}  \leq \# \mathcal{A}^{n}_{T}(x,y)$ and that on the event $\mathcal{E}^{n,\delta,T,\eta}$ we have   $\Height(\underline{T}^{n}_{\delta}(u)) \leq \Height({T}^{n}_{\delta}(u)) $ for every $u \in \mathcal{A}_{T}$ , 
write for $n$ large enough
\begin{eqnarray*}
&& \Pndt{
	 \frac{1}{\log n}   H^{n}_{\delta}(x,y) \geq     \frac{f_{\lambda}(x)  - 3\varepsilon}{-\log m_{\lambda}} 
} \\
& &
\qquad\qquad \geq \Pndt{
	\left\lbrace \forall u \in A_{T}, \  \frac{1}{\log n}   \Height(\underline{T}^{n}_{\delta}(u)) \geq     \frac{f_{\lambda}(x)  - 3\varepsilon}{-\log m_{\lambda} } \right\rbrace  \cap 
	 \mathcal{E}^{n,\delta,T,\eta} } 
\\  
&& \qquad\qquad \geq \Pndt{
	 \forall u \in A_{T} , \ \frac{1}{\log n}   \Height(\underline{T}^{n}_{\delta}(u)) \geq     \frac{f_{\lambda}(x)  - 3\varepsilon}{-\log m_{\lambda} }} - \P_{n,\delta,T}((\mathcal{E}^{n,\delta,T,\eta})^{c})  \\
&&\qquad\qquad  \geq   1- \left(1-\frac{1}{n^{f_{\lambda}( x)-2\varepsilon}} \right)^{n^{f_{\lambda}(x)-\varepsilon}}- \P_{n,\delta,T}((\mathcal{E}^{n,\delta,T,\eta})^{c})
\end{eqnarray*}
which goes to $1$ uniformly in $T \in \widehat{\mathbb{T}}^{n,\delta}$ by \eqref{eq:That}.

We continue with the upper bound of the first statement. 
By Lemma~\ref{lem:couplage}\ref{it:lem:couplage queue hauteur}, for every $u \in \mathcal{A}_{T}$,
\begin{align*}
	\Pndt{\Height(\overline{T}^{n}_{\delta}(u)) \geq   \frac{f_{\lambda}(x)  + 3\varepsilon}{-\log m_{\lambda}} \cdot \log n
}
\le
C n^{- \frac{-\log(m_\lambda + \eta)}{-\log(m_\lambda)}(f_{\lambda}( x)+3 \varepsilon)} \leq \frac{1}{n^{f_\lambda(x)+2\e}}
\end{align*}
for $n$ large enough (uniformly in $T \in \widehat{\mathbb{T}}^{n,\delta}$).  Thus  using the fact that  $ \# \mathcal{A}^{n}_{T}(x,y) \leq n^{f_{\lambda}(x)+\varepsilon} $
 and that on the event $\mathcal{E}^{n,\delta,T,\eta}$ we have   $\Height(\underline{T}^{n}_{\delta}(u)) \leq \Height({T}^{n}_{\delta}(u)) $ for every $u \in \mathcal{A}_{T}$, for $n$ large enough
\begin{eqnarray*}
&& \Pndt{
	 \frac{1}{\log n}   H^{n}_{\delta}(x,y) \geq     \frac{f_{\lambda}(x)  + 3\varepsilon}{-\log m_{\lambda}} 
} \\
& &
\qquad\qquad  \leq \Pndt{
	 \left\lbrace\frac{1}{\log n}   H^{n}_{\delta}(x,y) \geq     \frac{f_{\lambda}(x)  + 3\varepsilon}{-\log m_{\lambda}} \right\rbrace
 \cap 
\mathcal{E}^{n,\delta,T,\eta}}
+\P_{n,\delta,T}((\mathcal{E}^{n,\delta,T,\eta})^{c})
\\  
& &
\qquad\qquad  \leq  \Pndt{
	   \left\lbrace 
	   \exists \in A_{T} , \  \frac{1}{\log n}   \Height(\underline{T}^{n}_{\delta}(u)) \geq     \frac{f_{\lambda}(x)  + 3\varepsilon}{-\log m_{\lambda} }  
	   \right\rbrace
	    \cap 
	   \mathcal{E}^{n,\delta,T,\eta}} +\P_{n,\delta,T}((\mathcal{E}^{n,\delta,T,\eta})^{c})
\\  
& &
\qquad\qquad  \leq  \Pndt{
	  \exists \in A_{T} , \ \frac{1}{\log n}   \Height(\underline{T}^{n}_{\delta}(u)) \geq     \frac{f_{\lambda}(x)  + 3\varepsilon}{-\log m_{\lambda} } } + 2\P_{n,\delta,T}((\mathcal{E}^{n,\delta,T,\eta})^{c})
\\  &&\qquad\qquad  \leq    1- \left(1-\frac{1}{n^{f_{\lambda}( x)+2\varepsilon}} \right)^{n^{f_{\lambda}(x)+\varepsilon}}+2\P_{n,\delta,T}((\mathcal{E}^{n,\delta,T,\eta})^{c})
\end{eqnarray*}
which goes to $0$ uniformly in $T \in \widehat{\mathbb{T}}^{n,\delta}$ by \eqref{eq:That}. 

  The second statement is proved in the same way, by using the fact that Proposition~\ref{prop: profil faible}  entails that for every $\varepsilon>0$ and every $\delta\in \intervalleoo{0}{1}$,
 we have
\[
\mathbb{P}
\left(
\textrm{if }\tau'_{n} \geq  \lfloor (1-\delta) t_{\lambda} n  \rfloor \textrm{ then }  \left| \frac{\# \mathcal{A}^{n}_{\delta}(-\infty,x)}{n}-1 \right| \leq \varepsilon 
\right)
\quad \mathop{\longrightarrow}_{n \rightarrow \infty} \quad 1.
\]
This completes the proof.
\end{proof}

   \subsection{Proof of Theorem~\ref{thm:main}}
   \label{ssec:proof}
   
   We are now ready to establish our main result.

\begin{proof}[Proof of Theorem~\ref{thm:main}.] 
First, we note that for every $\delta\in \intervalleoo{0}{1}$, the random variable $\Height(\T_{\tau'_n}(\XXb^n))\ind{\tau'_n < \lfloor(1-\delta) t_\lambda n\rfloor}$ converges in law as $n\to \infty$ to a finite random variable (see the proof of Theorem~24 in \cite{BBKK23+}). 
		Since we know from \eqref{eq:limflu} that $\ind{\tau'_{n} \geq  \lfloor (1-\delta) t_{\lambda} n \rfloor}$ converges in distribution towards the random variable $\mathcal{B}$ that appears in the statement of the theorem, it is enough to show that 
		for every $\varepsilon>0$, for every $\delta \in \intervalleoo{0}{1}$ small enough, 
\begin{equation}
\label{eq:toprove} 
\mathbb{P}\left(
\ind{\tau'_{n} \geq  \lfloor (1-\delta) t_{\lambda} n \rfloor}\left|
\frac{\mathsf{Height}(\mathcal{T}_{\tau'_{n}}(\XXb^n))}{\log n}   - \kappa(\lambda)
\right| \geq \varepsilon
\right)  \quad \mathop{\longrightarrow}_{n \rightarrow \infty} \quad 0
\end{equation}
with $\kappa(\lambda)=\gamma/m_{\lambda}+h_{\lambda}(-\log m_{\lambda})$ for $\lambda \leq \lambda_{c}$ and $\kappa(\lambda)=\gamma e^{z_{\lambda}}$ for $\lambda \geq \lambda_{c}$,
where we recall that $\gamma=\lambda/(\lambda-1)$. {To simplify notation, we let $E_{n}$ be the event $ \{\tau'_{n} \geq  \lfloor (1-\delta) t_{\lambda} n \rfloor\}$.}

Fix $\eta>0$. Set $N_{\eta}\coloneqq \lfloor 1/\eta \rfloor$, $x_{i}\coloneqq \eta i z_{\lambda}$ for $1 \leq i \leq N_{\eta}$, $x_{0}\coloneqq-\infty$ and $x_{N_{\eta}+1}\coloneqq\infty$.
For $0 \leq i \leq N_{\eta}$ set
\[
\mathsf{H}^{n}_{\delta}(i)= \max \enstq{\haut(u)+\Height(T^{n}_{\delta}(u))}{ u \in \mathcal{A}^{n}_{\delta}(x_{i},x_{i+1}) },
\]
where we recall the notation $\mathcal{A}^{n}_{\delta}(x,y)= \{u \in \mathcal{A}^{n}_{\delta}:    \gamma e^{x} \log n   \leq {\haut(u)} \leq   \gamma e^{y} \log n\}$.
Observe that
\[
\mathsf{Height}(\mathcal{T}_{\tau'_{n}}(\XXb^n))=\max \left(\Height( \mathcal{T}_{\lfloor (1-\delta) t _{\lambda} n \rfloor}(\XXb^ {n} )), \max_{0 \leq i \leq N_{\eta}} \mathsf{H}^{n}_{\delta}(i)\right).
\]
By \eqref{eq:inegl} we have $\kappa(\lambda) \geq \gamma e^{z_{\lambda}}$, so by Lemma~\ref{lem:limHbefore}, the convergence \eqref{eq:toprove} will follow if we establish that
{for every $\varepsilon>0$, {if $\delta\in \intervalleoo{0}{1}$ is chosen small enough, then}
\begin{equation}
\label{eq:cvvers0}
\mathbb{P}\left(
\mathbb{1}_{E_{n}}
\left|
\frac{1}{\log n}\max_{0 \leq i \leq N_{\eta}} \mathsf{H}^{n}_{\delta}(i) -\kappa(\lambda)
\right| \geq \varepsilon
\right)  \quad \mathop{\longrightarrow}_{n \rightarrow \infty} \quad 0.
\end{equation}}

 Fix $\varepsilon>0$ {and $\delta\in \intervalleoo{0}{1}$}.
For every $0 \leq i \leq N_{\eta}$, recalling the notation $h_{\lambda}(x)= { f_{\lambda}(x)}/(-\log m_{\lambda})$,
\begin{eqnarray*}
&& \Pp{
\mathbb{1}_{E_{n}}(
\gamma e^{\eta i z_{\lambda}}+ h_{\lambda}(\eta i z_{\lambda})  -\varepsilon) \leq  \mathbb{1}_{E_{n}}
 \frac{\mathsf{H}^{n}_{\delta}(i)}{\log n} \leq   \mathbb{1}_{E_{n}}
 (\gamma e^{(\eta (i+1) z_{\lambda}) \wedge z_{\lambda}}+ h_{\lambda}(\eta i z_{\lambda})+\varepsilon)
} \\
&& = \mathbb{E} \left[
\Ppsq{\mathbb{1}_{E_{n}}(
\gamma e^{\eta i z_{\lambda}}+ h_{\lambda}(\eta i z_{\lambda})  -\varepsilon) \leq  \mathbb{1}_{E_{n}}
 \frac{\mathsf{H}^{n}_{\delta}(i)}{\log n} \leq   \mathbb{1}_{E_{n}}
 (\gamma e^{(\eta (i+1) z_{\lambda}) \wedge z_{\lambda}}+ h_{\lambda}(\eta i z_{\lambda})+\varepsilon)}{\mathcal{T}^{n,\delta}} \right],
\end{eqnarray*}
which converges to $1$ as $n \rightarrow \infty$ by Lemma~\ref{lem:dangling} {for every  $\delta\in \intervalleoo{0}{1}$ small enough}; for $i=N_{\eta}$ we also use Lemma~\ref{lem:limHbefore} combined with the inequality
$$  \mathbb{1}_{E_{n}} {\mathsf{H}^{n}_{\delta}(N_{\eta})} \leq   \mathbb{1}_{E_{n}} \left( \Height( \mathcal{T}_{\lfloor (1-\delta) t _{\lambda} n \rfloor}(\XXb^ {n} ))+    H^{n}_{\delta}(\eta N_{\eta} z_{\lambda},\infty) \right).$$

But by continuity, observe that
\[
\max_{0 \leq i  \leq   \lfloor 1/\eta \rfloor} \left(\gamma e^{\eta i z_{\lambda}}+ h_{\lambda}(\eta i z_{\lambda})\right)  \quad \mathop{\longrightarrow}_{\eta \rightarrow 0} \quad \sup_{0 \leq s \leq z_{\lambda}}  \left(\gamma e^{s}+ h_{\lambda}(s) \right) 
 \]
and
\[
\max_{0 \leq i \leq   \lfloor 1/\eta \rfloor} \left(\gamma e^{(\eta (i+1) z_{\lambda}) \wedge z_{\lambda}}+ h_{\lambda}(\eta i z_{\lambda})\right)  \quad \mathop{\longrightarrow}_{\eta \rightarrow 0} \quad \sup_{0 \leq s \leq z_{\lambda}}  \left(\gamma e^{s}+ h_{\lambda}(s) \right). 
 \]
Thus, for a fixed $\e>0$, by taking first $\eta>0$ small enough and then  $\delta\in \intervalleoo{0}{1}$  small enough, we get 
 $$
\mathbb{P}\left(
\mathbb{1}_{E_{n}}
\left|
\frac{1}{\log n}\max_{0 \leq i \leq N_{\eta}} \mathsf{H}^{n}_{\delta}(i) - \sup_{0 \leq s \leq z_{\lambda}}  \left(\gamma e^{s}+ h_{\lambda}(s) \right)
\right| \geq \varepsilon
\right)  \quad \mathop{\longrightarrow}_{n \rightarrow \infty} \quad 0,
$$
and \eqref{eq:cvvers0} follows from Proposition~\ref{prop:lc}(ii).

\end{proof}

\section{Profile of the tree via Laplace transforms and martingales}
\label{sec:profile}
In this section we establish our main result concerning 
{the active profile}
of the infection tree,
{i.e.\ the function recording the number of active vertices at each height in the tree.}
Recall that $\mathbb{A}_{k}^{n}(h)$ denotes the number of active vertices at height $h$ at time $k$ of $\mathcal{T}^{n}_{k}$ and that 
 $\bL_k^n(h)={\mathbb{A}_{k}^{n}(h)}/{I_k^n}$ is its normalized version.

Our goal is to  establish the following result, which in particular implies Proposition~\ref{prop: profil faible} as we will later see.
\begin{theorem}
	\label{thm:profil}
	Let $\lambda >1$ and set $\gamma=\lambda/(\lambda-1)$.
	For $n\ge 1$, 
	we consider $(\T_k^n)_{k\geq 0}$ the evolution of the epidemic tree constructed with parameter
	$\lambda_n \coloneqq \lambda/n$. 
	 Fix $t \in (0,t_{\lambda})$.  As $ n\rightarrow \infty$, on the event $\{\tau'_{n} \geq  \lfloor n t\rfloor\}$, we have
	\begin{equation*}
		\mathbb{L}^n_ {\lfloor n t \rfloor}(\gamma e^{x} \log n ) = e^{(f_{\lambda}(x)-1) \log n - \frac{1}{2}\log \log n + O_\P(1)}
	\end{equation*}
	uniformly for $x$ in a compact set of $(0,z_{\lambda})$ when $ \gamma e^{x} \log n \in \N$.
\end{theorem}
More precisely, this result actually holds with some form of uniformity in $\lambda$: the term $O_\P(1)$ in Theorem~\ref{thm:profil} denotes a random function $A_n(x, \lambda)$ with values in $\intervalleff{-\infty}{\infty}$ 
that is such that 
if we fix $t\in \intervalleoo{0}{\infty}$, a compact set $K$ and a compact interval $I\subset \intervalleoo{1}{\infty}$ so that $t\in \intervalleoo{0}{t_\lambda}$ and $K\subset \intervalleoo{0}{z_\lambda}$ for all $\lambda\in I$, then 
the family $ \left(\sup_{x\in K} |A_n(x, \lambda)| \ : \ n\geq 1, \lambda \in I\right)$
{satisfies a form of ``asymptotic tightness'', a rigorous definition of which can be found in Section~\ref{ssec:Oo} below.}

Let us describe our strategy to establish Theorem~\ref{thm:profil}.  
We first introduce, {for every $z\in \mathbb{C}$,
	\begin{align*}
		\cL(z,\T^n_k)\coloneqq\sum_{h=0}^{\infty} e^{hz}\bL_k^n(h)=\frac{1}{I^n_k} \sum_{\substack{u \in \T^n_k\\ \text{active}}} e^{z\haut(u)},
\end{align*}
the Laplace transform of the normalized active profile of the tree $\T^n_k$. 
Using the identification \eqref{eq couplage freezing} between the epidemic tree $\T^n_k$ and the uniform attachment tree with freezing $\T_k(\XXb^n)$, we show in Section~\ref{sous-section martingales} that, conditionally on the sequence $\XXb^n$ that tracks the order in which infections and recoveries take place during the epidemic, the conditional expectation $\Ecsq{\cL(z,\T_k(\XXb^n))}{\XXb^n}$ has a very tractable product form (see \eqref{eq:productform}). 
Moreover, for any fixed $z\in \mathbb{C}$, the
{quantity}
$\cL(z,\T_k(\XXb^n))$ divided by its expectation {(assuming it does not vanish)}
forms a martingale as $k$ grows.

Understanding the behaviour of $\cL(z,\T_k(\XXb^n))$ can hence be split into two parts: first, understanding the behaviour of its expectation $\Ecsq{\cL(z,\T_k(\XXb^n))}{\XXb^n}$, which is done in Section~\ref{sous-section contrôle des martingales}; second, showing that the ratio between $\cL(z,\T_k(\XXb^n))$ and its expectation, which we said above was a martingale, concentrates around some random function $M_\infty(z)$ when $k=\lfloor nt \rfloor$ and $n \rightarrow \infty$.
To this effect, we rely on the study of analogous quantities defined for the sequence $\XXb$ obtained as the limit of $\XXb^n$ as $n\rightarrow\infty$, and a coupling {between $\XXb$ and $\XXb^n$.}
This is done in Section~\ref{sous-section convergence des martingales}.
Some properties of the limiting function $z\mapsto M_\infty(z)$ are then studied in Section~\ref{sous-section limite non nulle}.
Last, in Section~\ref{ssec:endprofile}, we establish Theorem~\ref{thm:profil}
{by applying tools coming from}
Fourier analysis 
{to the function
 $z\mapsto \cL(z,\T_{\lfloor nt \rfloor}(\XXb^n))$}.
We also explain how to then obtain Proposition~\ref{prop: profil faible} from there.

Before tackling the study of these martingales, we need to lay down some background. 
We first introduce some {probabilistic} big-O ad little-o notation in Section~\ref{ssec:Oo}.
{Then, in Section~\ref{sous-section preliminaire couplage}, we provide} a coupled construction of the infection process {for different values of $n\geq 1$ and of the parameter $\lambda>1$}.
At some point, we will also need some technical results about the infection process, which we state and prove in Section~\ref{sous-section resultats techniques sur S}.

\uwave{\hspace{0.9\linewidth}}

\textbf{From now on, except in the proof Proposition~\ref{prop: profil faible}, we assume that for all $n\ge 1$, we have $\lambda_n = \lambda/n$ with $\lambda \in \intervalleoo{1}{\infty}$.}

\raisebox{0.8em}{\uwave{\hspace{0.9\linewidth}}}

\subsection{Probabilistic big-O and little-o notation}
\label{ssec:Oo}

Suppose that we have a family of random variables $(R(n;a_1,a_2,\dots ,a_k))$ with values in $\intervalleff{-\infty}{+\infty}$ indexed by $n$ and a finite number of parameters $a_1,a_2,\dots ,a_k$ (that can be integers, real numbers or complex numbers). 
We say that
\begin{align*}
	R(n;a_1,a_2,\dots ,a_k) = O_\P(1) \qquad \text{as $n\rightarrow\infty$,}
\end{align*} 
uniformly in $a_1 \in K_n^1$, \dots, $a_\ell \in K_n^\ell$,
weakly uniformly in $a_{\ell+1} \in K_n^{\ell+1}$, \dots, $a_k \in K_n^k$
if 
\begin{align}\label{eq:def OP(1)}
	\lim_{M\rightarrow \infty} \limsup_{n\rightarrow \infty} \sup_{a_{\ell+1} \in K_n^{\ell+1}, \dots ,a_k \in K_n^k} \Pp{\sup_{a_1 \in K_n^1, \dots, a_\ell \in K_n^\ell} |R(n;a_1,a_2,\dots ,a_k)| \geq M} = 0.
\end{align}
We similarly write $o_\P(1)$ instead of $O_\P(1)$ if for all $\e>0$ we have  
\begin{align}\label{eq:def oP(1)}
	\limsup_{n\rightarrow \infty} \sup_{a_{\ell+1} \in K_n^{\ell+1}, \dots ,a_k \in K_n^k} \Pp{\sup_{a_1 \in K_n^1, \dots, a_\ell \in K_n^\ell} |R(n;a_1,a_2,\dots ,a_k)| \geq \e} = 0.
\end{align}
Note that these definitions distinguish two types of uniformity: a (strong) uniformity \emph{in space} for the variables $a_1,\dots ,a_\ell$, and a weak uniformity \emph{in probability} for the variables $a_{\ell+1}, \dots, a_k$.

We adopt the following conventions:
\begin{itemize}
	\item If \eqref{eq:def oP(1)} holds and if the convergence $\sup_{a_1 \in K_n^1, \dots, a_\ell \in K_n^\ell} |R(n;a_1,a_2,\dots ,a_k)| \rightarrow 0$ as $n\rightarrow \infty$ also takes place almost surely, then we say that the $o_\P(1)$ is almost sure. 
	\item When dealing with deterministic quantities, we keep the same definition but we instead write $O(1)$ and $o(1)$ to emphasize that the quantity {at hand} is not random. 
	\item If $(B_n)_{n\geq 1}$ is a sequence of random variables, we will write $O_\P(B_n)$ and $o_\P(B_n)$ to mean $B_n \cdot O_\P(1)$ and $B_n \cdot o_\P(1)$, respectively. 
	\item 
	By writing ``on the event $E_n$, we have $R(n;a_1,\dots, a_k)=O_\P(1)$''  we will mean that we have $ \mathbbm{1}_{E_n} \cdot R(n;a_1,\dots, a_k)=O_\P(1)$.
\end{itemize}

\subsection{The infection process: a coupled construction}
\label{sous-section preliminaire couplage}
Recall from Section~\ref{ssec:SIR} the definition of the process $((H^n_k,I^n_k), \ k\geq 0)$.
Observe from dynamics of the system that for all $k\geq 0$ we have $n-k\leq H^n_k\leq n$, and $H^n_k=n - \sum_{i=1}^{k} \ind{I^n_i-I^n_{i-1}=1}$.
Also from the transition probabilities \eqref{eq:transitions}, we have 
\begin{align*}
	\frac{\lambda \cdot \left(1-\frac{k}{n}\right)}{1+\lambda \cdot \left(1-\frac{k}{n}\right)} \leq \Ppsq{I^n_{k+1}=I^n_k+1}{I^n_0,\dots , I^n_k}= \frac{\lambda \cdot \frac{H^n_k}{n}}{1+\lambda \cdot \frac{H^n_k}{n} } \leq \frac{\lambda}{1+\lambda}.
\end{align*} 
Observe that the left-  and right-hand side don't depend on the past. 
This motivates the following coupled construction:
starting from a sequence $(U_k)_{k\geq1}$ of i.i.d.\ uniform random variables on $[0,1]$ we define the sequences $\XXb=(\XX_k)_{k\geq 1}, \underline{\XXb}^n=(\underline{\XX}^n_k)_{k\geq 1}$ and $\XXb^n=(\XX^n_k)_{k\geq 1}$ and their associated walks $S, \underline{S}^n$ and $S^n$. 
We let $S_0=\underline{S}^n_0=S^n_0=1$ and then inductively for $k\geq 0$,
\begin{equation}\label{eq couplage}
	\left\lbrace 
	\begin{aligned}
		\XX_{k+1} &= S_{k+1} - S_k = 2 \cdot \mathbbm{1} \left\lbrace U_{k+1} \leq \frac{\lambda }{1+\lambda}\right\rbrace - 1\\
		\XX^n_{k+1} &= S^n_{k+1} - S^n_k = 2 \cdot \mathbbm{1}\left\lbrace U_{k+1} \leq \frac{\lambda \cdot \left(1-\frac{1}{n}\sum_{i=1}^{k} \ind{\XX_i^n=1}\right)}{1+\lambda \cdot \left(1-\frac{1}{n}\sum_{i=1}^{k} \ind{\XX_i^n=1}\right)}\right\rbrace -1\\
		\underline{\XX}^n_{k+1} &= \underline{S}^n_{k+1} - \underline{S}^n_k = 2 \cdot \mathbbm{1}\left\lbrace U_{k+1} \leq \frac{\lambda \cdot \left(1-\frac{k}{n}\right)}{1+\lambda \cdot \left(1-\frac{k}{n}\right)} \right\rbrace \cdot \ind{k\leq n} -1.\\
	\end{aligned}
	\right.
\end{equation}
Now by construction, we have $\underline{\XX}^n_{k} \leq \XX^n_{k} \leq \XX_{k}$ and hence $\underline{S}^n_{k} \leq S^n_{k} \leq S_{k}$ for all $k\geq 0$ and all $n\geq 1$.   

Also, it follows from the transitions \eqref{eq:transitions}   that  {$(S^n_{k\wedge \inf \{j\ge 0, \ S^n_j = 0\}})_{k\geq 0}$ has the same distribution as $(I^n_k)_{k\geq 0}$, and that in the coupling between these two processes, $\inf \{j\ge 0, \ S^n_j = 0\}$ corresponds to $\tau'_n$ the absorption time of $(H^n_k,I^n_k)_{k\ge 0}$. In the coupling \eqref{eq couplage freezing}, $\tau'_n$ corresponds to $\tau_n$ which is the number of steps made when the epidemic ceases. In what follows, to simplify notation using these couplings we shall identify $\tau_n$ and $\tau'_n$ with $\inf \{j\ge 0, \ S^n_j = 0\}$.}

{Most of the time, we keep the dependence in $\lambda$ implicit in the objects defined above. 
Whenever we need to make the dependence explicit we will write $\XX_k^n(\lambda)$, $\XX_k(\lambda), S(\lambda)$, etc. 
Note that all those objects are defined jointly for all $\lambda\in \intervalleoo{1}{\infty}$ in our coupled construction.
In particular, it is easy to check from \eqref{eq couplage} that the function $\lambda \mapsto S^n_{k}(\lambda)$ is non-decreasing for any $k$ and $n$: this can be shown by induction on $k$, just noting that if $S^n_k(\lambda_1) = S^n_k(\lambda_2)$, with $\lambda_1\leq \lambda_2$, then necessarily $\XX^n_{k+1}(\lambda_1) \leq \XX^n_{k+1}(\lambda_2)$.
}

\paragraph{A coupled construction of the trees.}
 For every $n \geq 1$ and $k \geq 0$ we set $\T^n_k \coloneqq \T_{k}(\XXb^n)$ and $\T_{k}\coloneqq \T_{k}(\XXb)$, as defined in Section~\ref{subsec:uniform attachment with freezing}. 
 It will be useful to couple the two sequences of trees $(\T^n_k)_{k\geq 0}$ and $(\T_k)_{k\geq 0}$, in such a way that the trees are the same until the two walks $S^n$ and $S$ start disagreeing. 
 This can be e.g.~achieved as follows. 
 Fix a sequence $(\widetilde{U}_k)_{k\geq 1}$ of i.i.d.\ uniform random variables on $[0,1]$, independent of the random variables $(U_k)_{k\geq 1}$. When building the trees 
 $(\T^n_k)_{k\geq 0}$ and $(\T_k)_{k\geq 0}$ using Algorithm \ref{algo1}, when needed, choose the  random active vertex $V_k$ sampled uniformly at random in $\mathcal{A}(T_{{k-1}})$ with $T_{k-1} \in \{\T^n_{k-1},\T_{k-1}\}$ as follows: 
 let  $v_1,v_2, \dots, v_{ \# \mathcal{A}(T_{k-1})}$ be the enumeration of the active vertices of $\mathcal{A}(T_{k-1})$ in their order of appearance, and choose $V_k$ by setting
\begin{align}\label{eq:coupling of the freezing and attachment vertices}
	V_k \coloneqq v_{I_k} \qquad \text{where}  \qquad I_k = \lceil    \widetilde{U}_k \cdot \#\mathcal{A}(T_{k-1}) \rceil.
\end{align}
Conditionally given $T_{k-1}$, {the random variable} $I_k$ is indeed uniform in $\intervalleentier{1}{\# \mathcal{A}(T_{k-1})}$.

In the sequel we assume that the two sequences $(\T^n_k)_{k\geq 0}$ and $(\T_k)_{k\geq 0}$ are built in this way, so that their evolution is the same until the two walks $S^n$ and $S$ start disagreeing.

\paragraph{Improved convergence results.}

Now that the processes $S, \underline{S}^n$ and $S^n$ are defined on the same probability space for all $n\geq 1$, we can improve some convergence results. 
Indeed, by \eqref{eq:limflu} (see also \cite[Eq. (33)]{BBKK23+})
we have the following fluid limit, where the convergence holds in distribution: for any $t\in \intervalleoo{0}{t_\lambda}$,
\begin{align}\label{eq:limite fluide en loi}
\left(\left(\frac{S^n_{\lfloor ns \rfloor}}{n}\right)_{s \geq 0} , \ind{\tau_n \geq tn}\right)  \quad \mathop{\longrightarrow}^{(d)}_{n \rightarrow \infty} \quad
\left((2- 2g_\lambda(s) -s)_{s \geq 0}, \mathcal{B}\right),
\end{align}
where $\mathcal{B}$ is a Bernoulli r.v.\ with parameter
$p=1-{1}/{\lambda}$, and where the first convergence holds for the topology of uniform convergence on compact sets.
Under the coupled construction, the above convergence will be improved as follows.
\begin{lemma}
For any fixed $t>0$, for any compact interval $I\subset \intervalleoo{1}{\infty}$ such that $t\in \intervalleoo{0}{t_\lambda}$ for all $ \lambda \in I$, for any $[t_1,t_2]\subset \R_+$, we have 
\begin{equation}\label{eq: limite fluide en proba}
	\ind{\tau_n > tn} = \ind{\forall i \ge 0, \ S_i >0}+ o_\P(1)\qquad \text{and} \qquad \frac{S^n_{\lfloor ns \rfloor}}{n}  =  (2- 2g_\lambda(s) -s)+ o_\P(1),
\end{equation}
where the $o_\P(1)$ are understood as $n\rightarrow \infty$, uniformly in $s\in \intervalleff{t_1}{t_2}$, weakly uniformly in $\lambda \in I$. 
\end{lemma}
Note that $\ind{\forall i\geq 0, \ S_i>0}$ is a Bernoulli r.v.\ with parameter $1-\frac{1}{\lambda}$: it corresponds to $\mathcal{B}$ in the previous statement.
\begin{proof}
First note that, deterministically, $
	\ind{\tau_n > nt} = \ind{\forall i \in \intervalleentier{0}{\lfloor nt\rfloor }, \ S_i^n>0} \leq \ind{\forall i \in \intervalleentier{0}{\lfloor nt\rfloor }, \ S_i>0}$
and that 
$
\ind{\forall i \in \intervalleentier{0}{\lfloor nt\rfloor }, \ S_i>0} =  \ind{\forall i\geq 0, \ S_i>0} + o_\P(1)
$
 as $n\rightarrow \infty$, weakly uniformly in $\lambda \in I$.
This entails that,
\begin{align*}
	\Ec{\left\vert\ind{\forall i \in \intervalleentier{0}{\lfloor nt\rfloor }, \ S_i>0} - \ind{\forall i \in \intervalleentier{0}{\lfloor nt\rfloor }, \ S_i^n>0} \right\vert} 
	&= \Ec{\ind{\forall i \in \intervalleentier{0}{\lfloor nt\rfloor }, \ S_i>0} - \ind{\forall i \in \intervalleentier{0}{\lfloor nt\rfloor }, \ S_i^n>0} }\\
	&= \Pp{\forall i \in \intervalleentier{0}{\lfloor nt\rfloor }, \ S_i>0} - \Pp{\forall i \in \intervalleentier{0}{\lfloor nt\rfloor }, \ S_i^n>0} \\
	&=\Pp{\forall i \geq 0, \ S_i>0} + o(1) - \Pp{\mathcal{B}=1} + o(1)\\
	&= o(1),
\end{align*}
where the two $o(1)$ appearing on the penultimate line should be understood as $n\rightarrow \infty$, weakly uniformly in $\lambda\in I$. 
The fact that the second $o(1)$ indeed holds weakly uniformly in $\lambda \in I$ can be checked using the proof of \cite[Theorem~24]{BBKK23+}.  
We can then write 
\begin{align*}
	&\ind{\forall i\geq 0, \ S_i>0} -\ind{\forall i \in \intervalleentier{0}{\lfloor nt\rfloor }, \ S_i^n>0}  \\
	&= \left( \ind{\forall i\geq 0, \ S_i>0} - \ind{\forall i \in \intervalleentier{0}{\lfloor nt\rfloor }, \ S_i>0}\right)
	+\left(\ind{\forall i \in \intervalleentier{0}{\lfloor nt\rfloor }, \ S_i>0} - \ind{\forall i \in \intervalleentier{0}{\lfloor nt\rfloor }, \ S_i^n>0} \right)\\
	&= o_\P(1) + o_\P(1),
\end{align*}
thanks to the considerations above. This proves the first part of \eqref{eq: limite fluide en proba}.

Let us check the second part of \eqref{eq: limite fluide en proba}.  
Fix $\varepsilon>0$. 
Since $\lambda \mapsto g_\lambda$ is continuous from $(1,\infty)$ to the space of continuous functions on $\R_+$ equipped with the topology of uniform convergence on compact sets, we may find a subdivision $\lambda_{1} <\cdots < \lambda_{k}$  of $I$ such that for all $j \in \intervalleentier{1}{k-1}$, for all $\lambda \in [\lambda_{j}, \lambda_{j+1}]$, for all $s \in [t_1,t_2]$,
\[
\vert g_{\lambda_{j}}(s)- g_{\lambda}(s) \vert \le \e \qquad \text{and} \qquad \vert g_{\lambda_{j+1}}(s)-g_\lambda(s) \vert \le \e.
\]
In the coupled construction, the convergence \eqref{eq:limite fluide en loi} holds in probability for all $\lambda \in \{\lambda_{1}, \ldots, \lambda_{k}\}$. 
As a result, using the monotonicity of $\lambda \mapsto S^n_{\lfloor ns \rfloor}(\lambda)$, 
\begin{multline*}
\P
\Bigg(\forall j \in \intervalleentier{1}{k-1}, \ \forall\lambda \in [\lambda_{j}, \lambda_{j+1}] , \ \forall s \in [t_1, t_2],\\
 \ 2-2g_{\lambda_{j}}(s) -s -\e \le \frac{S^n_{\lfloor ns \rfloor}(\lambda)}{n} \le 2-2g_{\lambda_{j+1}}(s) - s + \e \Bigg)
\mathop{\longrightarrow}\limits_{n \to \infty} 1.
\end{multline*}
Thus
\begin{multline*}
	\P
	\Bigg( \forall\lambda \in I , \ \forall s \in [t_1, t_2]
	\ 2-2g_{\lambda}(s) -s -3\e \le \frac{S^n_{\lfloor ns \rfloor}(\lambda)}{n} \le 2-2g_{\lambda}(s) - s + 3\e \Bigg)
	\mathop{\longrightarrow}\limits_{n \to \infty} 1,
\end{multline*}
hence the second part of \eqref{eq: limite fluide en proba}.
\end{proof}

\subsection{Martingales associated with the profile of uniform attachment trees with freezing}
\label{sous-section martingales}
Fix  $\Xb=(\X_i)_{i \geq 1}\in \{-1,+1\}^\N$
with associated walk $\Sb=(s_i)_{i \geq 0}$. 
Let $k\ge 0$ be such that $s_0, \ldots,s_{k}>0$. 
Conditionally given $\T_k{(\Xb)}$, let 
$W_k$
be taken uniformly at random in the set of active vertices of $\T_k{(\Xb)}$ (independently from all the other random variables) and define for all $z \in \mathbb{C}$,
\begin{align*}
	\cL(z,\T_k{(\Xb)})\coloneqq\Ecsq{e^{z \haut({{W}}_k)}}{\T_k{(\Xb)}} = \frac{1}{s_k} \sum_{\substack{u \in \T_k{(\Xb)}\\ \text{active}}} e^{z\haut(u)}.
\end{align*}
A very useful property of this object is stated in the following lemma.
\begin{lemma}\label{lem:Mk is a martingale}
	For every $i\in \intervalleentier{0}{k-1}$, we have 
		\begin{align*}
			\Ecsq{\cL(z,\T_{i+1}{(\Xb)})}{\T_1{(\Xb)}, \dots, \T_i{(\Xb)}}= \cL(z,\T_i{(\Xb)})  \cdot \left(1 + \frac{1}{s_{i+1}} (e^{z}-1) \cdot \ind{\X_{i+1}=1}\right).
		\end{align*}
\end{lemma}
In particular, since $\cL(z,\T_0(\Xb))=1$, this entails that the expectation of $\cL(z,\T_k{(\Xb)})$ has the product form
\begin{align*}
\Ec{\cL(z,\T_k{(\Xb)})} = \prod_{i=1}^{k}\left(1 + \frac{1}{s_{i}} (e^{z}-1) \cdot \ind{\X_{i}=1}\right).
\end{align*}
The content of the previous lemma hints at the fact that we can construct a martingale by dividing $\cL(z,\T_k{(\Xb)})$ by its expectation; 
we still need to be careful here because for some values of $z\in\mathbb{C}$ it is possible that some of the terms of this product vanish. 
To circumvent this problem we add a parameter $j\geq 1$: for any fixed $j\in \N$ set 
\begin{align*}
	C_k(z,\Xb,{j})\coloneqq \prod_{i={j+1}}^{k}\left(1 + \frac{1}{s_{i}} (e^{z}-1) \cdot \ind{\X_{i}=1}\right),
\end{align*}
and for every $z\in\mathbb{C}$ {and $\ell\in \intervalleentier{j}{k}$} for which $C_{{\ell}}(z,\Xb,{j})\neq 0$ we set 
	\begin{align*}
		M_{{\ell}}(z,\Xb,j)\coloneqq \frac{1}{C_{{\ell}}(z,\Xb,j)} \cL(z,\T_{{\ell}}(\Xb)).
\end{align*}
The considerations above
{and the product form of $C_{{\ell}}(z,\Xb,j)$}
 entail that if $C_{k}(z,\Xb,{j})\neq 0$ for some choice of $j,k$ and $z$ then the sequence $(M_{{\ell}}(z,\Xb,{j}))_{{j} \leq {\ell} \leq k}$ is a martingale for its canonical filtration. 

\begin{proof}[Proof of Lemma~\ref{lem:Mk is a martingale}]
	Set $\mathcal{F}_{i}\coloneqq \sigma(\T_{1}{(\Xb)}, \ldots, \T_{i}{(\Xb)})$ for $0 \leq i \leq k$. 
	Fix $i \in \{0,1, \ldots,k-1\}$. 
	Observe that when $\X_{i+1}=1$, the tree $\T_{i+1}{(\Xb)}$ is obtained from the tree $\T_i{(\Xb)}$ by adding a new vertex attached to a uniform random active vertex $V_i$ of $\T_i{(\Xb)}$ so we have
	\begin{align*}
		\cL(z,\T_{i+1}{(\Xb)})= \frac{s_i}{s_{i+1}} \cL(z,\T_i{(\Xb)}) + \frac{1}{s_{i+1}} e^{z (1+\haut(V_{i}))}.
	\end{align*}
	When $\X_{i+1}=-1$, the tree $\T_{i+1}{(\Xb)}$ is obtained from the tree $\T_i{(\Xb)}$ by freezing a uniform random active vertex $V_i$ of the tree so
	\begin{align*}
		\cL(z,\T_{i+1}{(\Xb)})= \frac{s_i}{s_{i+1}} \cL(z,\T_i{(\Xb)}) - \frac{1}{s_{i+1}} e^{z \haut(V_{i})}.
	\end{align*}
	All in all,  we have
	\begin{equation}
		\label{eq:L}\cL(z,\T_{i+1}{(\Xb)})= \frac{s_i}{s_{i+1}} \cL(z,\T_i{(\Xb)}) +  \frac{\X_{i+1}}{s_{i+1}} e^{z \haut(V_{i})} e^{z \ind{\X_{i+1}=1}}.
	\end{equation}Taking conditional expectations yields
	\begin{align*}
		\Ecsq{\cL(z,\T_{i+1}{(\Xb)})}{\mathcal{F}_i}&= \cL(z,\T_i{(\Xb)})  \cdot \left(\frac{s_i}{s_{i+1}} + \frac{\X_{i+1}}{s_{i+1}} e^{z \ind{\X_{i+1}=1}}\right)\\
		&= \cL(z,\T_i{(\Xb)})  \cdot \left(1 + \frac{1}{s_{i+1}} (e^{z}-1) \cdot \ind{\X_{i+1}=1}\right),
	\end{align*}
	which is the first statement of the lemma.
	The rest follows immediately.
\end{proof}

\paragraph{{The case of the epidemic tree.}} Recall from \eqref{eq couplage} the definitions of $S^{n}$ and $\XXb^{n}$, and the fact that $\T_k^n= \mathcal{T}_{k}(\XXb^{n})$. To simplify notation,  for every $n\ge 0$ we use  $\mathbb{E}_{n}$ for $\Ecsq{\, \cdot}{S^{n}}=\Ecsq{\, \cdot}{\XXb^{n}}$, 
where the randomness comes from the choice of the active vertices which are either frozen or to which is attached a new vertex {at each step}. 
Observe that {for a fixed $t>0$} the event $\{\forall i \in \intervalleentier{0}{\lfloor nt \rfloor}, \ S^n_i > 0\}$ is clearly  $\XXb^{n}$-measurable. 

{We introduce the sets 
	\begin{align}\label{eq:def domain E}
		\mathscr{E}=\mathscr{E}(\lambda) \coloneqq \enstq{z\in \mathbb C}{\Re(z)<z_\lambda} \qquad \text{and} \qquad \mathscr{E}'=\mathscr{E}'(\lambda) \coloneqq \enstq{z\in \mathbb C}{\Re(z)<2z_\lambda}.
	\end{align}}
For $\lambda>1$ and $t\in \intervalleoo{0}{t_\lambda}$ we also introduce 
	\begin{align*}
		J^n=J^n(\lambda,t) 
\coloneqq {\sup\enstq{j\in \intervalleentier{0}{\lfloor nt \rfloor}}{\frac{1}{S^n_i} \left( e^{2z_\lambda}+1 \right)\geq \frac{1}{2}}}
	\end{align*}
{with $J^{n}= \lfloor nt \rfloor +1$ by convention if the set that we consider is empty.}
Observe that $J^{n}$ is $\XXb^{n}$-measurable and that $J^{n} \leq \lfloor nt \rfloor +1$ by definition.
In addition, on the event $\{\forall i \in \intervalleentier{0}{\lfloor nt \rfloor}, \ S^n_i > 0\}$, for every $n \geq 0$, for every $k\in \intervalleentier{J^n}{\lfloor nt \rfloor}$ and $z \in {\mathscr{E}'}$,  we set
\begin{align*}
	C^n_k(z)\coloneqq C_k(z,\XXb^n,{J^n})
	= \prod_{i={J^n+1}}^{k} \left( 1+ \frac{1}{S^n_i} (e^z-1) \ind{\XX^n_i = 1} \right),
\end{align*}
so that $C^n_k(z)\neq 0$. 
Indeed, by the triangle inequality, for all $i \in \intervalleentier{J^n+1}{k}$ we have 
\begin{align*}
	\left\vert 1+ \frac{1}{S^n_i} (e^z-1) \ind{\XX^n_i = 1} \right\vert \geq 1 -   \frac{1}{S^n_i} (e^{\Re(z)}+1)>  1 -   \frac{1}{S^n_i} (e^{2z_\lambda}+1)>\frac{1}{2} >0 .
\end{align*}
	We can then define
\begin{align}\label{eq:definition M k n de z}
	M^n_k(z)\coloneqq M_k(z,\XXb^n,{J^n})=\frac{1}{C^n_k(z)} \mathcal{L}(z,\T^n_k).
\end{align}
Note that for any $z\in {\mathscr{E}'}$ we have
	\begin{equation}
		\Ecp{n}{\mathcal{L}(z,\T^n_k)} = \Ecp{n}{\mathcal{L}(z,\T^n_{J^n})} \cdot C_k^n(z). \label{eq:productform}
	\end{equation}
More generally, by Lemma~\ref{lem:Mk is a martingale}, under $\mathbb{E}_{n}$, the process $(M^{n}_{k}(z))_{J^{n} \leq k \leq \lfloor nt \rfloor}$ is a martingale for its canonical filtration.

\paragraph{The case of the local limit.} Similarly to the case of $S^n$, we introduce the analogous objects for the walk $S$. 
Recall from \eqref{eq couplage} the definitions of $S$ and $\XXb$, and the fact that $\T_k= \mathcal{T}_{k}(\XXb)$. 
We work on the event $\{\forall k \geq 0, \ S_k>0\}$. 
For $\lambda>1$, we introduce 
\begin{equation}
\label{eq:defJ}
	J=J(\lambda) 
	{\coloneqq \sup\enstq{j\geq 0}{\frac{1}{S_j} \left( e^{2z_\lambda}+1 \right)\geq {\frac{1}{2}}}.}
\end{equation}
Observe that $J<\infty$ almost surely by the strong law of large numbers.
For every $k\geq J$ and $z \in \mathscr{E}'$ we set
\begin{align}\label{eq:definition Ckz}
C_k(z)= C_k(z,\XXb,J)\coloneqq \prod_{i=J+1}^{k} \left( 1+ \frac{1}{S_i} (e^z-1) \ind{\XX_i = 1} \right),
\end{align}
so that $C_k(z)\neq 0$ by definition of $J$.  
We can then define for $k \geq J$
\begin{align}\label{eq:definition Mkz}
M_k(z)\coloneqq M_k(z,\XXb,J)=\frac{1}{C_k(z)} \mathcal{L}(z,\T_k),
\end{align}
so that, conditionally given $S$, the process $(M_{k}(z))_{k\ge J}$ is a martingale for its canonical filtration.

\subsection{Some technical estimates for $S^n$, $S$, $J^{n}$ and $J$}\label{sous-section resultats techniques sur S}
In order to estimate the quantities $C_k(z)$, $C^n_k(z)$, $J$ and $J^n$ that have been introduced in the previous section, we will rely on a few technical lemmas. We gather their statements here and prove them below, each in their own separate subsection.

In all the following lemmas, we fix $t>0$ a real number and  $I\subset \intervalleoo{1}{\infty}$ a compact interval such that $t\in \intervalleoo{0}{t_\lambda}$ for all $\lambda \in I$.

We start with a technical result.

\begin{lemma}\label{lem:tension de M lambda}
Let $M(\lambda)$ be defined as
\begin{align}\label{eq:definition of M lambda}
	M(\lambda) \coloneqq \sup_{i\geq 0} \left(\frac{i}{S_i(\lambda)}\ind{S_i(\lambda)>0}\right).
\end{align}
The family $(M(\lambda): \lambda \in I)$ is tight. 
\end{lemma}
This enables us to prove the next lemma, which will be useful for controlling $C_k(z)$ as $k\rightarrow\infty$.
\begin{lemma}\label{lem:convergence sum 1/Sk - c log k towards rv}
	We have 
	\begin{align*}
		\ind{\forall i\in \intervalleentier{0}{k},\ S_i>0} \cdot \left(\sum_{i=1}^k\frac{1}{S_i} \ind{\XX_i=1}- \frac{\lambda}{\lambda-1} \log k \right) = Z(\lambda) + o_\P(1),
	\end{align*}
	as $k\rightarrow \infty$ weakly uniformly in $\lambda\in I$, where the family of random variables $(Z(\lambda) : \lambda\in I)$ is tight and the $o_\P(1)$ is almost sure.  
\end{lemma}

We state in Lemma~\ref{lem:sum 1/Sni is gamma log k} below a somewhat similar statement for $S^n$, which instead will help us control the term $C^n_k(z)$, as $k,n\rightarrow \infty$.
The proof of Lemma~\ref{lem:sum 1/Sni is gamma log k} relies {on a technical result, Lemma~\ref{lem:control on the Si}, which we state first.}
\begin{lemma}\label{lem:control on the Si}
	We have 
	\begin{align}\label{eq:control 1/S^n_i}
		\ind{\forall i \in \intervalleentier{0}{\lfloor nt \rfloor}, \ S^n_i > 0} \cdot \frac{k}{S_k^n} =   O_\P(1)
		\qquad \textrm{and} \qquad 
		S_k - S^n_k = \left(1+\frac{k^2}{n}\right) \cdot O_\P(1),
	\end{align}
	as $n\rightarrow \infty$, uniformly in $k \in \intervalleentier{0}{\lfloor nt \rfloor}$, weakly uniformly in $\lambda \in I$. 
\end{lemma}

\begin{lemma}
	\label{lem:sum 1/Sni is gamma log k}
	We have 
	\begin{equation}\label{eq approximation somme un sur S}
		\ind{\forall i\in\intervalleentier{0}{\lfloor n t \rfloor},\ S^n_i>0} \cdot \left(\sum^{k}_{i=1} \frac{1}{S^n_i} \ind{\XX^n_i=1} - \frac{\lambda}{\lambda-1} \log k\right) = O_\P(1)
	\end{equation}
	as $n\rightarrow \infty$, uniformly in $k \in \intervalleentier{0}{\lfloor n t \rfloor }$, weakly uniformly in $\lambda \in I$. 
\end{lemma}

Finally, we state a result that involves $J$ and $J^n$.
\begin{lemma}\label{lem:JnJ}
		The following assertions hold.
		\begin{enumerate}[label=(\roman*)]
		\item The family $(J(\lambda) {\ind{\forall k \geq 0 , \ S_k(\lambda)>0}} : \lambda \in I)$ is tight.
		\item We have 
		\begin{align*}
			J^n \cdot  \ind{\forall i\in \intervalleentier{0}{\lfloor nt \rfloor}, \ S^n_i>0} = J \cdot \ind{\forall i\geq 0, \  \ S_i>0} + o_\P(1),
		\end{align*}
		where the $o_\P(1)$ holds as $n\rightarrow \infty$, weakly uniformly in $\lambda\in I$. 
		\end{enumerate}
\end{lemma}

\subsubsection{Proof of Lemma~\ref{lem:tension de M lambda}}
\begin{proof}[Proof of Lemma~\ref{lem:tension de M lambda}]
Let $\lambda_0\coloneqq \min(I)$ be the minimum of the interval $I$. 
{
First, note that by properties of random walks, we have the almost sure convergence $\frac{i}{S_i(\lambda_0)} \rightarrow \frac{\lambda_0 +1}{\lambda_0-1}>0$ as $i\rightarrow \infty$, and also we know that almost surely the inequality $\frac{i}{S_i(\lambda_0)}  \leq \frac{\lambda_0 +1}{\lambda_0-1}$ does not hold simultaneously for all $i\geq 0$.
This ensures that $\sup_{i \geq 0} \left(\frac{i}{S_i(\lambda_0)}\ind{S_i(\lambda_0)>0}\right)$ is not attained for $i\rightarrow \infty$, so it has to be attained at a finite time. 
}
{Consequently, we may consider $K$ the smallest such time, so that
	\begin{align*}
		\sup_{i\geq 0} \left(\frac{i}{S_i(\lambda_0)}\ind{S_i(\lambda_0)>0}\right)= \frac{K}{S_K(\lambda_0)}\ind{S_K(\lambda_0)>0}.
	\end{align*}}
Note that necessarily $S_K(\lambda_0)>0$. 
Now, for any $i\geq 1$ and any $\lambda \in I$, thanks to the monotonicity in $\lambda$ we have $S_i(\lambda_0) \leq S_i(\lambda)$, so that
\begin{itemize}
	\item If $S_i(\lambda_0)\geq 1$ then it is immediate that 
	\begin{align*}
		\frac{i}{S_i(\lambda)}\ind{S_i(\lambda)>0} \leq \frac{i}{S_i(\lambda_0)}\ind{S_i(\lambda_0)>0} \leq \frac{K}{S_K(\lambda_0)}\ind{S_K(\lambda_0)>0}.
	\end{align*} 
	\item If $S_i(\lambda_0) \leq 0$, then since $\frac{S_i(\lambda_0)}{i} \rightarrow \frac{\lambda_0-1}{\lambda_0+1}>0$ as $i\rightarrow \infty$ and the random walk only moves by steps of $+1$ and $-1$, there exists a time $T\geq i$ such that $S_T(\lambda_0)=1$. 
	Then
	\begin{align*}
		\frac{i}{S_i(\lambda)}\ind{S_i(\lambda)>0} \leq i \leq T = \frac{T}{S_T(\lambda_0)}\ind{S_T(\lambda_0)>0} \leq \frac{K}{S_K(\lambda_0)}\ind{S_K(\lambda_0)>0}.
	\end{align*}  
\end{itemize}
The reasoning above ensures that for any $i\geq 0$ and $\lambda \in I$ we have 
\begin{align*}
	\frac{i}{S_i(\lambda)}\ind{S_i(\lambda)>0} \leq \frac{K}{S_K(\lambda_0)}\ind{S_K(\lambda_0)>0}=M(\lambda_0),
\end{align*}
so that $M(\lambda_0) = \sup_{\lambda \in I} M(\lambda)$,
which is in fact stronger than what is claimed by the lemma. 
\end{proof}

\subsubsection{Proof of Lemma~\ref{lem:convergence sum 1/Sk - c log k towards rv}}
\begin{proof}[Proof of Lemma~\ref{lem:convergence sum 1/Sk - c log k towards rv}]
	We write, {on the event $\{\forall i\in \intervalleentier{0}{k},\ S_i>0\}$, } with $\gamma^{\mathrm{EM}}$ the Euler-Mascheroni constant, 
	\begin{align*}
		&	\sum_{i=1}^k\frac{1}{S_i} \ind{\XX_i=1} - \frac{\lambda}{\lambda-1} \log k \\
		&=\sum_{i=1}^k  \left(\frac{1}{S_i} \ind{\XX_i=1} - \frac{\lambda}{\lambda-1}\cdot \frac{1}{i} \right) + \frac{\lambda}{\lambda-1}\cdot\gamma^{\mathrm{EM}} + o(1)\\
		&= \sum_{i=1}^k\left(\frac{1}{S_i} - \frac{1}{i}\cdot \frac{\lambda+1}{\lambda-1} \right)\cdot\ind{\XX_i=1} 
		+\sum_{i=1}^k\left(\frac{1}{i}\cdot \frac{\lambda+1}{\lambda-1}\ind{\XX_i=1} 
		- \frac{\lambda}{\lambda-1}\cdot \frac{1}{i} \right) + \frac{\lambda}{\lambda-1}\cdot\gamma^{\mathrm{EM}} + o(1),
	\end{align*}
where the $o(1)$ is uniform in $\lambda \in I$. 
We handle the two random terms on the RHS of the last display separately. 

First, note that the term $\sum_{i=1}^k\left(\frac{1}{i}\cdot \frac{\lambda+1}{\lambda-1}\ind{\XX_i=1} 
	- \frac{\lambda}{\lambda-1}\cdot \frac{1}{i} \right)$ is a sum of independent centered bounded random variables with second moments given by 
	\begin{align*}
		\Ec{ \left(\frac{1}{i} \frac{\lambda+1}{\lambda-1} \ind{\XX_i=1} - \frac{\lambda}{\lambda-1} \frac{1}{i}\right)^2} =\frac{1}{i^2} \left(\frac{\lambda+1}{\lambda-1}\right)^2 \left(\frac{\lambda}{\lambda+1}-\left(\frac{\lambda}{\lambda+1}\right)^2 \right) = O(i^{-2}), 
	\end{align*}
	as $i\rightarrow \infty$, uniformly in $\lambda \in I$. 
	This ensures that by defining
	\begin{align*}
		Z'(\lambda) \coloneqq \sum_{i=1}^\infty \left(\frac{1}{i} \frac{\lambda+1}{\lambda-1} \ind{\XX_i=1} - \frac{\lambda}{\lambda-1} \frac{1}{i}\right),
	\end{align*}
then $Z'(\lambda)$ is well-defined as an $L^2$ random variable, with $\Ec{(Z'(\lambda))^2}= O(\sum_i i^{-2})  = O(1)$ uniformly in $\lambda\in I$, so that $(Z'(\lambda))_{\lambda\in I}$ is tight. 
By considering the $L^2$ norm of the remainders we write 
\begin{align}
\sum_{i=1}^k\left(\frac{1}{i}\cdot \frac{\lambda+1}{\lambda-1}\ind{\XX_i=1} 
	- \frac{\lambda}{\lambda-1}\cdot \frac{1}{i} \right) 
	&= Z'(\lambda) -\sum_{i=k+1}^\infty \left(\frac{1}{i}\cdot \frac{\lambda+1}{\lambda-1}\ind{\XX_i=1} 
	- \frac{\lambda}{\lambda-1}\cdot \frac{1}{i} \right) \notag\\
	&= Z'(\lambda) + o_\P(1), \label{eq:convergence Z'}
\end{align}
as $k\rightarrow \infty$, weakly uniformly in $\lambda\in I$. 
Finally, by Kolmogorov's two series theorem, we get that the $o_\P(1)$ appearing in the last display holds almost surely.

Now, let us turn to the term $\sum_{i=1}^k\left(\frac{1}{S_i} - \frac{1}{i}\cdot \frac{\lambda+1}{\lambda-1} \right)\cdot\ind{\XX_i=1} $. 
First, consider the random variable
\begin{align}
	Y= Y(\lambda) \coloneqq \sup_{i\geq 1}\left( i^{-\frac{3}{4}}\cdot \Big|S_i - i \frac{\lambda -1}{\lambda +1}\Big|\right).
\end{align}
From a union-bound and Hoeffding's inequality, it is easy to check that 
\begin{align*}
	\Pp{Y\geq A} 
	=\Pp{\forall i\geq 1, \ \left\vert S_i - i\frac{\lambda-1}{\lambda +1}\right \vert  \leq A \cdot i^{3/4}} 
	&\geq 1 - \sum_{i=1}^{\infty} \Pp{\left\vert S_i - i \frac{\lambda-1}{\lambda +1}\right \vert  \geq A \cdot i^{3/4}}\\ 
	&\geq 1 -  \sum_{i=1}^{\infty}2\exp(-\frac{A^2}{2} \sqrt{i})
\end{align*}
which tends to $1$ as $A\rightarrow \infty$, uniformly in $\lambda \in I$, so that $(Y(\lambda))_{\lambda \in I}$ is tight. 
Recalling the definition of $M=M(\lambda)$ from Lemma~\ref{lem:tension de M lambda}, on the event $\{\forall i\geq 0, \ S_i >0\}$ we have $M=\sup_{i\geq 0}\left(\frac{i}{S_i}\right)$ so that 
\begin{align*}
	\left\vert \frac{1}{S_i} - \frac{1}{i}\cdot \frac{\lambda+1}{\lambda-1} \right\vert \leq  
	\frac{\lambda+1}{\lambda -1}
	\frac{i}{S_i}
	\frac{\left\vert S_i -i \frac{\lambda -1}{\lambda +1} \right\vert }{i^2} 
	\leq \frac{\lambda+1}{\lambda -1} \cdot M \cdot \frac{Y \cdot i^{3/4}}{i^2}
	\leq \frac{\lambda+1}{\lambda -1}\cdot M(\lambda) \cdot Y(\lambda) \cdot i^{-5/4},
\end{align*}
which is the general term of a convergent series.
Hence we have
\begin{align}\label{eq:convergence Z''}
	\ind{\forall i\geq 0, \ S_i >0}\cdot \sum_{i=1}^k\left(\frac{1}{S_i} - \frac{1}{i}\cdot \frac{\lambda+1}{\lambda-1} \right)\cdot\ind{\XX_i=1}  = Z''(\lambda) +o_\P(1) 
\end{align}
almost surely as $k\rightarrow \infty$, weakly uniformly in $\lambda\in I$, where
\begin{align*}
	Z''(\lambda)\coloneqq 	\ind{\forall i\geq 0, \ S_i >0}\cdot \sum_{i=1}^\infty \left(\frac{1}{S_i} - \frac{1}{i}\cdot \frac{\lambda+1}{\lambda-1} \right)\cdot\ind{\XX_i=1},
\end{align*}
and $|Z''(\lambda)| \leq \frac{\lambda +1}{\lambda - 1} \cdot M(\lambda) \cdot Y(\lambda)\cdot \zeta(5/4)$. 
It is immediate from the tightness of $(Y(\lambda))_{\lambda \in I}$ and $(M(\lambda))_{\lambda \in I}$ that $(Z''(\lambda))_{\lambda \in I}$ is tight. 

Setting $Z(\lambda) \coloneqq Z'(\lambda) + Z''(\lambda)$ and using together \eqref{eq:convergence Z'} and \eqref{eq:convergence Z''} and the fact that almost surely $\ind{\forall i\in \intervalleentier{0}{k},\ S_i>0}= \ind{\forall i\ge 0,\ S_i>0} + o_\P(1)$ yields the convergence result.
The tightness result follows from the tightness of $(Z'(\lambda))_{\lambda \in I}$ and  $(Z''(\lambda))_{\lambda \in I}$.
\end{proof}

\subsubsection{Proof of Lemma~\ref{lem:control on the Si}}
\begin{proof}[Proof of Lemma~\ref{lem:control on the Si}]
	From the convergence result \eqref{eq: limite fluide en proba} we know that for any small enough fixed $\e>0$, we have 
	\begin{align*}
	\ind{\forall i \in \intervalleentier{0}{\lfloor nt \rfloor}, \ S^n_i>0} \cdot  \min_{\e n \leq i \leq tn} \left(\frac{S^n_i}{n}\right)
	 = 
	 \ind{\forall i \geq 0, \ S_i>0} \cdot \inf_{\e \leq s \leq t}(2-2g_\lambda(s)- s) + o_{\P}(1)
	\end{align*}
where $\inf_{\e\leq s \leq t}(2-2g_\lambda(s)- s)>0$ so that  
	\begin{align*}
		\ind{\forall i \in \intervalleentier{0}{\lfloor nt \rfloor}, \ S^n_i>0} \cdot \max_{\e n \leq i \leq tn}\left( \frac{i}{S^n_i} \right) = O_\P(1).
	\end{align*}
	In the end we just need to control what happens when $0\leq i \leq \e n$.
	For that, we define on the same probability space yet another random walk $\underline{S}^{(\e)}$ as $\underline{S}^{(\e)}_k=1+ \sum_{i=1}^{k} \underline{\XX}^{(\e)}_i$ for all $k\ge 0$ where for $i\geq 1$
	 \[\underline{\XX}^{(\e)}_i= 2 \cdot \mathbbm{1}\left\lbrace U_i \leq \frac{\lambda \cdot \left(1-\e\right)}{1+\lambda \cdot \left(1-\e\right)}\right\rbrace -1.\]
	 From their construction, it is clear that for all $i\in \intervalleentier{0}{\lfloor \e n\rfloor}$ we have $\underline{S}^{(\e)}_i\leq \underline{S}^n_i \leq S^n_i$.
	 Also note that, provided that $\e>0$ is chosen small enough, by the law of large numbers,
	 \begin{align*}
	  \ind{\forall i \in \intervalleentier{0}{\lfloor nt \rfloor}, \ \underline{S}^{(\e)}_i >0} \cdot \max_{i \in \intervalleentier{0}{\lfloor n \e \rfloor}} \left(\frac{i}{\underline{S}^{(\e)}_i}\right)= O_\P(1).
	 \end{align*}
 In the end 
 \begin{multline}\label{eq:majoration ind fois max i/S}
 	\ind{\forall i \in \intervalleentier{0}{\lfloor nt \rfloor}, \ S^n_i > 0}\cdot \max_{i \in \intervalleentier{0}{\lfloor nt \rfloor}} \left(\frac{i}{S_i^n}\right)\\
 	\leq \ind{\forall i \in \intervalleentier{0}{\lfloor nt \rfloor}, \ S^n_i>0} \cdot \max_{i \in \intervalleentier{\lfloor n \e  \rfloor}{\lfloor nt \rfloor}} \left( \frac{i}{S^n_i}\right) 
 	+  \ind{\forall i \in \intervalleentier{0}{\lfloor nt \rfloor}, \ \underline{S}^{(\e)}_i >0}\cdot  \max_{i \in \intervalleentier{0}{\lfloor n \e \rfloor}} \left(\frac{i}{\underline{S}^{(\e)}_i}\right)\\
 	+   \mathbbm{1} \left\lbrace \forall i \in \intervalleentier{0}{\lfloor nt \rfloor}, \ S^n_i > 0
 	\right\rbrace \setminus 
 	\lbrace \forall i \in \intervalleentier{0}{\lfloor nt \rfloor}, \ \underline{S}^{(\e)}_i >0
 	\rbrace
 \cdot \max_{i \in \intervalleentier{0}{\lfloor nt \rfloor}} \left(\frac{i}{S_i^n}\right).
 \end{multline}
For a fixed and small enough $\e>0$ the two first terms are $O_\P(1)$. 
Let us control the last term. Since $S^n_i \le S_i$ and $\underline{S}_i^{(\e)}\le S_i$,
\begin{align*}
	\P&\left(\forall i \in \intervalleentier{0}{\lfloor nt\rfloor}, \ S^n_i >0 \text{ and } \exists i \in  \intervalleentier{0}{\lfloor nt\rfloor}, \ \underline{S}^{(\e)}_i \le0\right)\\
	&\le
	 \P\left(\forall i \in \intervalleentier{0}{\lfloor nt\rfloor}, \ S_i >0 \text{ and } \exists i \in  \intervalleentier{0}{\lfloor nt\rfloor}, \ \underline{S}^{(\e)}_i \le0\right)\\
	 &=\P\left(\forall i \in \intervalleentier{0}{\lfloor nt\rfloor}, \ S_i >0\right) - \P\left(\forall i \in \intervalleentier{0}{\lfloor nt\rfloor}, \ \underline{S}^{(\e)}_i >0\right)\\
	 &= 1- \frac{1}{\lambda} +o(1) - \left(1- \frac{1}{\lambda-\e}\right) + o(1)
\end{align*}
uniformly in $\lambda \in I$.
Thus,
\begin{align*}
 \limsup_{n\to \infty} \sup_{\lambda \in I} \P\left(\forall i \in \intervalleentier{0}{\lfloor nt\rfloor}, \ S^n_i >0 \text{ and } \exists i \in  \intervalleentier{0}{\lfloor nt\rfloor}, \ \underline{S}^{(\e)}_i \le0\right) \mathop{\longrightarrow}\limits_{\e \to 0} 0.
\end{align*}
From \eqref{eq:majoration ind fois max i/S} and what immediately follows, for any $A>0$, the quantity 
\begin{align*}
	\sup_{\lambda \in I}\Pp{\ind{\forall i \in \intervalleentier{0}{\lfloor nt \rfloor}, \ S^n_i > 0}\cdot \max_{i \in \intervalleentier{0}{\lfloor nt \rfloor}} \left(\frac{i}{S_i^n}\right) \geq A} 
\end{align*}
can be shown to be arbitrarily close to $0$ as $n\rightarrow \infty$ by first taking $\e$ small enough and then $n$ large enough.
This proves the first point of \eqref{eq:control 1/S^n_i}.
	
For the second point, 
note that for $i\in \intervalleentier{\lceil \frac{n}{4}\rceil }{\lfloor nt \rfloor }$, the result holds trivially using that $S_i\leq i$ a.s..
Now remark that $ (S_i-\underline{S}^n_i)/2$ is a sum of independent Bernoulli random variables with parameters $(p^n_i)_{i\geq 1}$ given by
\begin{align*}
	p^n_i=\frac{\lambda}{1+\lambda} - \frac{\lambda \cdot \left(1-\frac{i}{n}\right)}{1+\lambda \cdot \left(1-\frac{i}{n}\right)} = \frac{\lambda}{1+\lambda} \cdot \left(1 - \frac{1- \frac{i}{n}}{1-\frac{\lambda}{1+\lambda} \frac{i}{n}}\right)
	= \frac{\lambda}{(1+\lambda)^2} \frac{\frac{i}{n}}{1-\frac{\lambda}{1+\lambda} \frac{i}{n}} \leq \frac{2\lambda}{(1+\lambda)^2}\frac{i}{n},
\end{align*}
where for the last inequality, we assume that $i\leq \frac{n}{4}$. 
Hence for $z>0$ we have
\begin{align*}
	\Ec{\exp \left(z \cdot \frac{1}{2}\cdot (S_k-\underline{S}^n_k)\right)} = \prod_{i=1}^{k} \left(1+ (e^z-1) \cdot p^n_i\right) &\leq \exp\left((e^z-1) \cdot \sum_{i=1}^{k} p^n_i\right)\\
	&\leq \exp\left(\frac{2\lambda}{(1+\lambda)^2}\cdot (e^z-1) \cdot \frac{k^2}{n}\right).
\end{align*}

{
We now fix an integer $\ell\geq 1$ and consider values of $k$ so that we have $\ell-1 < \frac{k^2}{n} \leq \ell$, or equivalently $\sqrt{n(\ell-1)} < k \leq \sqrt{n \ell}$.
Since $k\mapsto (S_k-\underline{S}^n_k)$ is non-increasing, its maximum over an interval is always attained at its rightmost point.}
Using first this monotonicity argument and then a Chernoff bound with $z=1$, we get
\begin{align*}
	\Pp{\sup_{\sqrt{n(\ell-1)} < k \leq \sqrt{n \ell} }\frac{1}{2}\cdot (S_k-\underline{S}^n_k) \geq x\ell} 
	&=\Pp{\frac{1}{2}\cdot (S_{\lfloor\sqrt{n \ell}\rfloor}-\underline{S}^n_{\lfloor\sqrt{n \ell}\rfloor}) \geq x\ell}\\ 
	&\leq \exp\left(-x \ell + \frac{2\lambda}{1+\lambda}\cdot (e-1) \cdot \ell \right)\\
	&\leq \exp\left(\left(\frac{2\lambda}{1+\lambda}\cdot (e-1)-x\right) \cdot \ell \right).
\end{align*}
If $x>0$ is large enough, then the last expression is summable in $\ell$ so we can use a union bound over all $\ell \leq \frac{n}{16}$ so that 
\begin{align*}
	 \Pp{\sup_{0\leq i \leq nt}\left(\frac{S_i - \underline{S}^n_i}{2\left(1+\frac{i^2}{n}\right)}\right) \geq x} 
	 &\leq \sum_{\ell=1}^{\lfloor\frac{n}{16}\rfloor}
	 \Pp{\sup_{\sqrt{n(\ell-1)} < k \leq \sqrt{n \ell} }\frac{1}{2}\cdot (S_k-\underline{S}^n_k) \geq x\ell}\\
	 &\leq \sum_{\ell=1}^{\lfloor\frac{n}{16}\rfloor}
	 \exp\left(\left(\frac{2\lambda}{1+\lambda}\cdot (e-1)-x\right) \cdot \ell \right)\\
	 &\leq \frac{\exp\left(\frac{2\lambda}{1+\lambda}\cdot (e-1)-x \right)}{1 - \exp\left(\frac{2\lambda}{1+\lambda}\cdot (e-1)-x\right)},
\end{align*}
where 
the first inequality comes from a union-bound, the second comes from the previous display and the last is obtained by summing the obtained geometric series. 
Now, this last quantity goes to $0$ as $x\rightarrow \infty$, uniformly in $\lambda\in I$.
\end{proof}

\subsubsection{Proof of Lemma~\ref{lem:sum 1/Sni is gamma log k}}
\begin{proof}[Proof of Lemma~\ref{lem:sum 1/Sni is gamma log k}]
	We write
	\begin{align*}
		&\sum^{k}_{i=1} \frac{1}{S^n_i} \ind{\XX^n_i=1} - \frac{\lambda}{\lambda-1} \log k\\ 
		&= \sum^{k}_{i=1} \left(\frac{1}{S^n_i}-\frac{1}{S_i} \right)\ind{\XX^n_i=1} 
		+ \sum^{k}_{i=1} \frac{1}{S_i} \left(\ind{\XX^n_i=1} - \ind{\XX_i=1}\right)
		+ \left(\sum_{i=1}^k\frac{1}{S_k} \ind{\XX_k=1}- \frac{\lambda}{\lambda-1} \log k\right).
	\end{align*}
	Note that the last term is already taken care of by the previous lemma. Since $S_i \ge S^n_i$ for all $i\ge 0$ by the coupling we get
	\begin{align*}
		\ind{\forall i\in \intervalleentier{0}{\lfloor nt \rfloor}, \ S^n_i>0}\sum^{k}_{i=1} \frac{1}{S_i} \left\vert \ind{\XX^n_i=1} - \ind{\XX_i=1}\right\vert 
		&\le 
		\ind{\forall i\in \intervalleentier{0}{\lfloor nt \rfloor}, \ S_i>0}\sum^{k}_{i=1} \frac{1}{S_i} \left\vert \ind{\XX^n_i=1} - \ind{\XX_i=1}\right\vert \\
		&\leq M \cdot \sum^{k}_{i=1} \frac{1}{i} \left\vert \ind{\XX^n_i=1} - \ind{\XX_i=1}\right\vert,
	\end{align*}
where $M=M(\lambda)$ is defined in Lemma~\ref{lem:tension de M lambda}, and is $O_\P(1)$ thanks to that lemma. 
	Now just note that for any $i$ we have $\left\vert \ind{\XX^n_i=1} - \ind{\XX_i=1}\right\vert \leq \left\vert \ind{\underline{\XX}^n_i=1} - \ind{\XX_i=1}\right\vert$ and the latter is a Bernoulli random variable with parameter 
	\begin{align*}
		p^n_i=\frac{\lambda }{1+\lambda} - \frac{\lambda \cdot \left(1-\frac{i}{n}\right)}{1+\lambda \cdot \left(1-\frac{i}{n}\right)} =i \cdot O\left(\frac{1}{n}\right)
	\end{align*}
	uniformly in $i \in \intervalleentier{1}{\lfloor nt \rfloor} $. 
	This ensures that the expectation of the sum $\sum^{k}_{i=1} \frac{1}{i} \left\vert \ind{\XX^n_i=1} - \ind{\XX_i=1}\right\vert$ is $O(\frac{k}{n})=O(1)$ for $k\in \intervalleentier{1}{\lfloor nt \rfloor}$, so that the sum itself is indeed $O_{\P}(1)$.
	
	Last, we need to take care of the first term. 
	On the event $\{\forall i \in \intervalleentier{0}{\lfloor nt \rfloor}, \ S^n_i>0\}$, we have 
	\begin{align*}
		0\leq \left(\frac{1}{S^n_i}-\frac{1}{S_i} \right)\ind{\XX^n_i=1}  \leq  \frac{1}{S^n_i}-\frac{1}{S_i} = \frac{S_i - S^n_i}{S_i \cdot S^n_i}. 
	\end{align*}

From \eqref{eq:control 1/S^n_i} in Lemma~\ref{lem:control on the Si}, 
noting that $S_k\geq S^n_k$ by construction we have 
	\begin{align*}
		\frac{\ind{\forall i \in \intervalleentier{0}{\lfloor nt \rfloor}, \ S^n_i > 0}}{S_k}\leq \frac{\ind{\forall i \in \intervalleentier{0}{\lfloor nt \rfloor}, \ S^n_i > 0}}{S^n_k}= \frac{O_{\P}(1)}{k}
	\end{align*}
	so that using Lemma~\ref{lem:control on the Si} again we get
	\begin{equation}\label{eq majoration 1 sur S moins 1 sur S n}
	\ind{\forall i \in \intervalleentier{0}{\lfloor nt \rfloor}, \ S^n_i > 0} \cdot 	\frac{1}{S_k} \cdot \frac{1}{S^n_k} \cdot \left(S_k - S^n_k\right) \leq \frac{O_{\P}(1)^2}{k^2} \cdot \left(1+\frac{k^2}{n}\right) \cdot O_\P(1) = \left(\frac{1}{k^2} + \frac{1}{n} \right) \cdot O_\P(1)
	\end{equation}
uniformly in $k \in \intervalleentier{1}{\lfloor nt \rfloor}$, weakly uniformly in $\lambda \in I$, so that in the end $\ind{\forall i \in \intervalleentier{0}{\lfloor nt \rfloor}, \ S^n_i > 0} \cdot \sum_{i=1}^{\lfloor nt \rfloor} \left(\frac{1}{S^n_i}-\frac{1}{S_i} \right)\ind{\XX^n_i=1} = O_\P(1)$, which finishes the proof of Lemma~\ref{lem:sum 1/Sni is gamma log k}.
\end{proof}

\subsubsection{Proof of Lemma~\ref{lem:JnJ}}
\begin{proof}[Proof of Lemma~\ref{lem:JnJ}]
For the first assertion, by Lemma~\ref{lem:tension de M lambda} we have for every $i \geq 1$, 
$$
 \ind{S_i(\lambda)>0} \frac{1}{S_i(\lambda)}  \leq \frac{M(\lambda)}{i}.
$$
It follows that for $i> 2 M(\lambda)(e^{2 z_{\lambda}}+1)$ we have $\ind{S_i(\lambda)>0} {S_i(\lambda)^{-1}} (e^{2 z_{\lambda}}+1) <{{1}/{2}}$ so that 
$J(\lambda) {\ind{\forall k \geq 0 , \ S_k(\lambda)>0}} \leq 2 M(\lambda)(e^{2 z_{\lambda}}+1) +1$, which entails the desired tightness.

For (ii), we first check that  ${J^n \cdot  \ind{\forall i\in \intervalleentier{0}{\lfloor nt \rfloor}, \ S^n_i>0}} = O_\P(1)$, weakly uniformly in $\lambda \in I$.
Recall from the coupled construction \eqref{eq couplage} that for all $k\ge 0$, we have $S^n_k \le S_k$, 
and from \eqref{eq: limite fluide en proba} that $\ind{\forall i \in \intervalleentier{0}{\lfloor nt \rfloor}, \ S^n_i >0} = \ind{\forall i\ge 0, \ S_i >0} + o_\P(1)$. 
Using \eqref{eq majoration 1 sur S moins 1 sur S n}, we have 
\begin{align}
\frac{1}{S^n_k} \left(e^{2z_\lambda} +1 \right) \ind{\forall i \in \intervalleentier{0}{\lfloor nt \rfloor}, \ S^n_i >0}&\le \frac{1}{S_k} \left(e^{2z_\lambda} +1 \right) \ind{\forall i \in \intervalleentier{0}{\lfloor nt \rfloor}, \ S^n_i >0} + \left( \frac{1}{k^2} + \frac{1}{n}\right) O_\P(1)\notag \\
&=  \frac{1}{S_k} \left(e^{2z_\lambda} +1 \right)\left( \ind{\forall i\ge 0, \ S_i >0} + o_\P(1) \right)+ \left( \frac{1}{k^2} + \frac{1}{n}\right) O_\P(1), \label{eq:4.8a}
\end{align}
{ as $n\rightarrow \infty$, uniformly in $k \in \intervalleentier{1}{\lfloor nt \rfloor}$, weakly uniformly in $\lambda \in I$.} 
As a result, by definition of $J^n$ and by the strong law of large numbers, using the coupling of Section~\ref{sous-section preliminaire couplage}, we see that $J^n = O_\P(1)$, weakly uniformly in $\lambda \in I$.

Besides, using the coupling of Section~\ref{sous-section preliminaire couplage}, 
for any fixed $k\geq 0$, we have $S^n_k=S_k+o_\P(1)$ almost surely as $n\rightarrow \infty$, weakly uniformly in $\lambda \in I$. 
As a consequence, 
for any fixed $K\geq 1$ we have 
$$
\frac{1}{S^n_k} \left(e^{2z_\lambda} +1 \right) \ind{\forall i \in \intervalleentier{0}{\lfloor nt \rfloor}, \ S^n_i >0} = \left( 1 + o_\P(1) \right)\frac{1}{S_k}\left(e^{2z_\lambda} +1 \right) \ind{\forall i \in \intervalleentier{0}{\lfloor nt \rfloor}, \ S^n_i >0},
$$
as $n\rightarrow \infty$, uniformly in $k\in \intervalleentier{0}{K}$, weakly uniformly in $\lambda \in I$. 
Finally, let $K\ge 1$. 
The above equality implies that
\begin{equation}
\label{eq:4.8b}J^n \ind{\forall i \in \intervalleentier{0}{\lfloor nt \rfloor}, \ S^n_i >0} \ind{J^n \le K \text{ and } J \le K} = J \ind{\forall i \in \intervalleentier{0}{\lfloor nt \rfloor}, \ S^n_i >0} \ind{J^n \le K \text{ and } J \le K} + o_\P(1).
\end{equation}
{Using \eqref{eq: limite fluide en proba} again together with} 
the fact that $(J(\lambda) {\ind{\forall i \ge 0, \ S_i(\lambda) >0}})_{\lambda \in I}$ is tight , we get
\begin{equation}\label{eq:J n proche de J avec K}
J^n \ind{\forall i \in \intervalleentier{0}{\lfloor nt \rfloor}, \ S^n_i >0} \ind{J^n \le K \text{ and } J \le K} = J \ind{\forall i \ge 0, \ S_i >0} \ind{J^n \le K \text{ and } J \le K} + o_\P(1).
\end{equation}
Thus, for all $K \ge 1$,
\begin{align*}
	&\sup_{\lambda \in I}\P\left( \left\vert J^n \ind{\forall i \in \intervalleentier{0}{\lfloor nt \rfloor}, \ S^n_i >0} -J \ind{\forall i \ge 0, \ S_i >0}  \right\vert \ge \e \right)\\
	&\le
	\sup_{\lambda \in I} \P\left(J^n > K \text{ or } J >K\right)+ \sup_{\lambda \in I}\P\left( \left\vert J^n \ind{\forall i \in \intervalleentier{0}{\lfloor nt \rfloor}, \ S^n_i >0} -J \ind{\forall i \ge 0, \ S_i >0}  \right\vert \ge \e , \ J^n \le K \text{ and } J \le K \right).
\end{align*}
This entails the desired result, since the second term goes to zero as $n\to \infty$ by \eqref{eq:J n proche de J avec K} and the first term can be made arbitrarily small by taking $K$ large enough and using the fact that $J^n = O_\P(1)$  and that $(J(\lambda))_{\lambda \in I}$ is tight.
\end{proof}

\subsection{Control of the martingales $(M^n_k(z))$}\label{sous-section contrôle des martingales}
We give here some estimates on the martingales $M^n_k(z)$, which will be useful for the convergence result given in the next subsection.
{In all this subsection,
	we fix $I\subset \intervalleoo{1}{\infty}$ a compact interval and $t\in \intervalleoo{0}{\infty}$ such that $t\in \intervalleoo{0}{t_\lambda}$ for all $\lambda \in I$. Recalling the definition of $J^{n}$,  on the event $\{\forall i\in \intervalleentier{0}{\lfloor nt \rfloor},\ S^n_i>0\}$,  for every $ k \in \llbracket J^n, \lfloor nt \rfloor \rrbracket$ and for every $z\in \mathscr{E}'$ we have $C^n_k(z) \neq 0$, so that $M^n_k(z)$ is well-defined.}
\begin{lemma}\label{lemme majoration moment martingale}
	Fix a compact $K\subset \mathbb C$ such that $K\subset \mathscr{E}(\lambda)$ for all $\lambda \in I$. For {$p\in \intervalleof{1}{2}$}, for all $n \ge 0$, on the event $\{\forall i\in \intervalleentier{0}{\lfloor nt \rfloor},\ S^n_i>0\}$
	we have 
	\begin{align*} \ind{k\geq J^n} \E_n \left[{\left\vert M^n_k(z)\right\vert^p}\right]
	\le 
	 {\ind{k\geq J^n}} \frac{C^n_k(p \Re z)}{\vert C^n_k(z) \vert^p} {\cdot O_\P(1)},
	\end{align*} 
as $n\rightarrow\infty$, {uniformly in $z\in K$, and in $k\in \intervalleentier{0}{\lfloor nt \rfloor}$,} weakly uniformly in $\lambda\in I$.
\end{lemma}
\begin{proof}
		Fix $k\in \intervalleentier{0}{\lfloor nt \rfloor}$.  We work on the event $\{\forall i \in \intervalleentier{0}{\lfloor nt \rfloor}, \ S^n_i > 0\}\cap \{ k \geq J_{n}\}$.
		By definition we have 
\(
	\mathcal{L}(z,\T_k^n) = \Ecp{n}{e^{z \haut(U^n_k)}},
\)	
where conditionally on $\T_k^n$, the vertex ${V^n_k}$ is sampled uniformly at random among the active vertices of $\T_k^n$.
	By Jensen's inequality, we have
	\begin{align*}
		\E_n \left[{\left\vert M^n_k(z)\right\vert^p}\right] =
		\frac{\E_n \left[{\left\vert \mathcal{L}(z,\T_k^n)\right\vert^p}\right]}{\vert C^n_k(z) \vert^p}
		= 	\frac{\E_n \left[{\left\vert \Ecsq{e^{z \haut({V^n_k})}}{\T_k^n} \right\vert^p}\right]}{\vert C^n_k(z) \vert^p}
		&\leq \frac{\E_n\left[{\Ecsq{\left\vert e^{z \haut({V^n_k})} \right\vert^p}{\T_k^n} }\right]}{\vert C^n_k(z) \vert^p} \\
		&=\frac{\E_n\left[{\Ecsq{e^{ p \Re (z) \haut({V^n_k})}}{\T_k^n} }\right]}{\vert C^n_k(z) \vert^p}\\
		&= \frac{C^n_k(p \Re z)}{\vert C^n_k(z) \vert^p} \cdot {\Ecp{n}{\mathcal{L}(p \Re z,\T^n_{J^n})}},
	\end{align*}
where in the last equality, we have used \eqref{eq:productform}.
Now, since by Lemma~\ref{lem:JnJ}(ii) on the event $\{\forall i \in \intervalleentier{0}{\lfloor nt \rfloor}, \ S^n_i > 0\}$ we have $J^n=O_\P(1)$ as $n\rightarrow\infty$ weakly uniformly in $\lambda\in I$, then the same is true for  $\Ecp{n}{\mathcal{L}(p \Re z,\T^n_{J^n})}$. 
	This completes the proof.
\end{proof}

\begin{lemma}
	\label{lem:asymptotics Cnkz}
Fix a compact $K\subset \{z \in \mathbb{C}, \ \Re z>0\}$ such that $K\subset \mathscr{E}(\lambda)$ for all $\lambda \in I$.
On the event $\{\forall i \in \intervalleentier{0}{\lfloor nt \rfloor}, \ S^n_i > 0\}$, 
 we have
	\begin{align*} 
	{\ind{k\geq J^n}}C^n_k(z) = {\ind{k\geq J^n}}\exp\left( \gamma (e^z -1) \log k + O_{\P}(1)\right),
	\end{align*} 
as $n\rightarrow \infty$, uniformly in $z \in K$ and in $k\in \intervalleentier{0}{\lfloor nt \rfloor}$, weakly uniformly in $\lambda \in I$. 

Similarly, for compact $K\subset \{z \in \mathbb{C}, \ \Re z>0\}$ such that $K\subset \mathscr{E}'(\lambda)$ for all $\lambda \in I$,
on the event $\{\forall i \in \intervalleentier{0}{\lfloor nt \rfloor}, \ S^n_i > 0\}$,
for every {$p\in \intervalleof{1}{2}$}, 
	\begin{align*} 
	{\ind{k\geq J^n}}\frac{C^n_k(p \Re z)}{\vert C^n_k(z) \vert^p} ={\ind{k\geq J^n}}\exp \left(
	\gamma\left( e^{p \Re z} -1 -p (\Re(e^z) -1)\right) \log k + O_{\P}(1)
	\right),
	\end{align*}
as $n\rightarrow \infty$, 	uniformly in $z \in K$ and in $k\in \intervalleentier{0}{\lfloor nt \rfloor}$, {weakly uniformly in $\lambda \in I$}. 
\end{lemma}
\begin{proof}
We work on the event $\{\forall i \in \intervalleentier{0}{\lfloor nt \rfloor}, \ S^n_i > 0\}$. 
Fix $k$ such that $J^n \leq  k \leq nt$ and write
	\begin{align*}
		C^n_k(z) = \prod_{i={J^n{+1}}}^k \left(1+ \frac{1}{S^n_i} (e^z-1) \ind{\XX^n_i= 1}\right)
		& = \exp\left(\sum_{i=J^n{+1}}^k \log\left(1+\frac{1}{S^n_i} \ind{\XX^n_i=1} (e^z-1)\right)\right)\\
		&=\exp\left(\sum_{i=J^n{+1}}^k \frac{1}{S^n_i} \ind{\XX^n_i=1} (e^z-1) + O_{\P}(1)\right),
	\end{align*}
{where $\log(1+w)\coloneqq\sum_{\ell \ge 1} (-1)^{\ell+1} w^\ell/\ell$ for any $w \in \mathbb{C}$ with $|w|<1$ (note that here by definition of $J^n$ we have $\left\vert \frac{1}{S^n_i} \ind{\XX^n_i=1} (e^z-1) \right\vert <\frac{1}{2}$ for all $i\geq J^n$).}
The second equality above is obtained from the fact that $|\log(1+w)-w|\leq |w|^2$ for all $w$ with $|w|<{\frac{1}{2}}$ together with the fact that,  by \eqref{eq:control 1/S^n_i}, we have 
	\begin{align}\label{eq:sum inverse square Si is O of 1}
		\sum_{i=J^n{+1}}^{ \lfloor t n\rfloor} \frac{1}{(S^n_i)^2} = O_{\P}(1)
	\end{align}
	as $n \to \infty$, {weakly uniformly in $\lambda \in I$}.
	Moreover
	\begin{align*}
		\frac{C^n_k(p \Re z)}{\vert C^n_k(z) \vert^p} &=\prod_{i=J^n{+1}}^{k} \left(\left(1 + \frac{1}{S^n_{i}} (e^{p \Re z}-1) \cdot \ind{\XX^n_i=1}\right)\cdot
		\left\vert 1 + \frac{1}{S^n_{i}} (e^{z}-1) \cdot \ind{\XX^n_i=1}\right\vert^{-p}\right) \\
		&=
		\exp \left(
		\sum_{i=J^n{+1}}^{k} \left( \log \left(1 + \frac{1}{S^n_{i}} (e^{p \Re z}-1) \cdot \ind{\XX^n_i=1}\right)
		-p\log \left\vert 1 + \frac{1}{S^n_{i}} (e^{z}-1) \cdot \ind{\XX^n_i=1}\right\vert
		\right)
		\right) \\
		&=\exp \left(
		\sum_{i=J^n{+1}}^{k} \left( \frac{1}{S^n_{i}} (e^{p \Re z}-1) \cdot \ind{\XX^n_i=1}
		-p\Re\left( \frac{1}{S^n_{i}} (e^{z}-1) \cdot \ind{\XX^n_i=1}\right)
		\right)
		+O_{\P}(1)
		\right), \\
		&=
		\exp \left(
		\left(  e^{p \Re z}-1
		-p\Re\left(  e^{z}-1 \right)
		\right)\cdot 
		\sum_{i=J^n{+1}}^{k} \frac{1}{S^n_{i}} \ind{\XX^n_i=1} 
		+O_{\P}(1)
		\right)
	\end{align*}
	using \eqref{eq:sum inverse square Si is O of 1}  again to get the third equality.
	We conclude using Lemma~\ref{lem:sum 1/Sni is gamma log k}.
	\end{proof}

\begin{proposition}\label{proposition majoration increments Mn}
	Let $K \subset \{ z \in \mathbb{C}, \ \Re z>0\}$ be a compact set
	{such that $K\subset \mathscr{E}(\lambda)$ for all $\lambda \in I$.} 
	For every {$p\in \intervalleof{1}{2}$}, on the event $\{\forall i \in \intervalleentier{0}{\lfloor nt \rfloor}, \ S^n_i > 0\}$,  we have
	\begin{align*} 
	{\ind{k\geq J^n}}\E_n \left[ \lvert M^n_{k+1 }(z)- M^n_k(z)\rvert^p \right] \le \exp \left(
	\left( -p+ \gamma\left( e^{p \Re z} -1 -p (\Re(e^z) -1)\right)
	\right) \log k + O_\P(1)
	\right)
	\end{align*} 
where the $O_\P(1)$ holds as $n\rightarrow \infty$, uniformly in $z \in K$ and $k\in \intervalleentier{0}{\lfloor nt \rfloor}$, weakly uniformly in $\lambda \in I$. 
\end{proposition}
\begin{proof}
	We work on the event $\{\forall i \in \intervalleentier{0}{\lfloor nt \rfloor}, \ S^n_i > 0\}$. For $k\in \intervalleentier{0}{\lfloor nt \rfloor}$, using \eqref{eq:L} 
we get that 
\begin{align*}
\cL(z,\T^n_{k+1})= \frac{S^n_k}{S^n_{k+1}} \cL(z,\T^n_k) +  \frac{\XX^n_{k+1}}{S^n_{k+1}} e^{z \haut(V^n_{k})} e^{z \ind{\XX^n_{k+1}=1}}.
\end{align*}
where $V^n_k$ denotes the independent uniform active vertex of $ \T^{n}_{k}$ {chosen at step $k+1$}.
Now, if $k\in \intervalleentier{J^n}{\lfloor nt \rfloor}$, from the definition \eqref{eq:definition M k n de z} of $M^n_k(z)$ we then get 
	\begin{align} 
	M^n_{k+1}&(z) - M^n_k(z)  
	= { \frac{1}{C^n_{k+1}(z)} \mathcal{L}(z,\T^n_{k+1}) - \frac{1}{C^n_{k}(z)} \mathcal{L}(z,\T^n_{k})} \notag \\
	&= \left(1- \frac{\XX^n_{k+1}}{S^n_{k+1}} \right)
	\left( \frac{C^n_{k}(z)}{C^n_{k+1}(z)} -1 \right)M^n_k(z)-\frac{\XX^n_{k+1}}{S^n_{k+1}} M^n_k(z)
	+ \frac{\XX^n_{k+1}}{S^n_{k+1}} \left( (e^z-1) \ind{\XX^n_{k+1} = 1} +1\right) \frac{e^{z \haut(V^n_k)}}{C^n_{k+1}(z)}.\label{eq M k plus 1 moins M k}
	\end{align} 
	Using that $C^n_{k+1}(z) = (1+(1/S^n_{k+1})(e^z-1)\ind{\XX^n_{k+1}=1}) C^n_k(z)$ and Lemma~\ref{lemme majoration moment martingale}, we see that uniformly in $z\in K$, uniformly in $k\in \intervalleentier{0}{\lfloor nt \rfloor}$,
	\begin{equation}\label{eq premier terme increment M k}
		{\ind{k\geq J^n}}\E_n\left[ \left\vert  \left(1- \frac{\XX^n_{k+1}}{S^n_{k+1}} \right)
		\left( \frac{C^n_{k}(z)}{C^n_{k+1}(z)} -1 \right)M^n_k(z)\right\vert^p\right]= O_\P\left(\frac{1}{(S^n_{k+1})^p}\frac{C^n_k(p \Re (z))}{\left\vert C^n_{k}(z) \right\vert^p} \right).
	\end{equation}
	Again by Lemma~\ref{lemme majoration moment martingale}, uniformly in $z\in K$, uniformly in $k\in \intervalleentier{0}{\lfloor nt \rfloor}$, we have
	\begin{equation}\label{eq deuxieme terme increment M k}
		{\ind{k\geq J^n}}\E_n\left[ \left\vert \frac{\XX^n_{k+1}}{S^n_{k+1}} M^n_k(z) \right\vert^p\right]= O_\P\left(\frac{1}{(S^n_{k+1})^p}\frac{C^n_k(p \Re (z))}{\left\vert C^n_{k}(z) \right\vert^p} \right).
	\end{equation}
	For $k \geq J^{n}$, the identity  $C^n_{k+1}(z) = (1+(1/S^n_{k+1})(e^z-1)\ind{\XX^n_{k+1}=1}) C^n_k(z)$ combined with the fact that $\left\vert \frac{1}{S^n_{k+1}} \ind{\XX^n_{k+1}=1} (e^z-1) \right\vert <\frac{1}{2}$  implies that $|C^n_{k+1}(z)| \geq |C^n_{k}(z)|/2$. Besides, using the identity 
	\[\E_{n}\left[\left\vert e^{z \haut(V^n_k)} \right\vert^p\right]=\E_{n}[ e^{p\mathrm{Re}(z) \haut(V^n_k)} ]= C^{n}_{k}(p\mathrm{Re}(z)) \cdot {\Ecp{n}{\cL(p\mathrm{Re}(z),\T^n_{J^n})}}
	\]
and the fact that $\Ecp{n}{\mathcal{L}(p \Re z,\T^n_{J^n})}=O_{\P}(1)$ which is obtained in the end  of the proof of Lemma \ref{lemme majoration moment martingale},
	 we get uniformly in $z\in K$, uniformly in $k\in \intervalleentier{0}{\lfloor nt \rfloor}$,
	\begin{equation}\label{eq troisieme terme increment M k}
		{\ind{k\geq J^n}}\E_n\left[ \left\vert\frac{\XX^n_{k+1}}{S^n_{k+1}} \left( (e^z-1) \ind{\XX^n_{k+1} = 1} +1\right) \frac{e^{z \haut(V^n_k)}}{C^n_{k+1}(z)} \right\vert^p\right]= O_\P\left(\frac{1}{(S^n_{k+1})^p}\frac{C^n_k(p \Re (z))}{\left\vert C^n_{k}(z) \right\vert^p} \right).
	\end{equation}
Combining \eqref{eq M k plus 1 moins M k}, \eqref{eq premier terme increment M k}, \eqref{eq deuxieme terme increment M k} and \eqref{eq troisieme terme increment M k}, we deduce that uniformly in $z\in K$, uniformly in $k\in \intervalleentier{0}{\lfloor nt \rfloor}$,
	\begin{align*}
	{\ind{k\geq J^n}}\E_n\left[\left\vert M^n_{k+1}(z) - M^n_k(z) \right\vert^p\right]
	= O_\P\left(\frac{1}{(S^n_{k+1})^p}\frac{C^n_k(p \Re (z))}{\left\vert C^n_{k}(z) \right\vert^p} \right).
	\end{align*}
	Hence, by \eqref{eq:control 1/S^n_i},
	\begin{align*} 
		{\ind{k\geq J^n}}\Epc{n}{\lvert M^n_{k+1}(z)- M^n_k(z) \rvert^p}
		= O_\P\left(\frac{1}{k^p} \frac{C^n_k(p \Re (z))}{\left\vert C^n_{k}(z) \right\vert^p}\right).
	\end{align*} 
	The conclusion follows from Lemma~\ref{lem:asymptotics Cnkz}.
\end{proof}

Recall from \eqref{eq:param} the definitions of $z_{\lambda}$ and {$f_\lambda$}.

\begin{corollary}\label{corollaire M nt proche de M An}
Fix $z\in \intervalleoo{0}{\infty}$ such that $z\in \intervalleoo{0}{z_\lambda}$ for all $\lambda \in I$.
Let $(K_n)$  be a sequence of compact subsets of $\mathbb{C}$ such that 
$\mathrm{diam}(K_n) \to 0$ as $n\to \infty$ 
and such that {for every $n \geq 1$ we have $K_n \subset \mathscr{E}(\lambda)$ for all $\lambda \in I$} and $z \in K_n$. 
Let $(A_n)$ be a sequence of integers such that $A_n \to \infty$ and $A_n \le nt$ {for all $n\geq 1$}. 
Then there exists {$p\in \intervalleof{1}{2}$} such that on the event $\{\forall i \in \intervalleentier{0}{\lfloor nt \rfloor}, \ S^n_i > 0\} \cap \{J^{n} \leq A_{n}\} $,
	\begin{align*} 
\Epc{n}{\vert M^n_{\lfloor nt \rfloor } (z_n) -M^n_{A_n}(z_n) \vert^p}	=o_\P(1),
	\end{align*} 
where the $o_\P(1)$ holds as $n\rightarrow \infty$, uniformly in $z_n \in K_n$ and weakly uniformly in $\lambda \in I$. 
\end{corollary}
\begin{proof}
	By Lemma~1 of \cite{Big92} (see also Lemma~A.2 of \cite{Sen21}), on the event $\{\forall i \in \intervalleentier{0}{\lfloor nt \rfloor}, \ S^n_i > 0\} {\cap \{J^n \leq A_n\}}$ we have
	\begin{align*} 
	\Ecp{n}{\vert M^n_{\lfloor nt \rfloor } (z_n) -M^n_{A_n}(z_n)\vert^p}
	\le 2^p \sum_{k=A_n}^{\lfloor nt \rfloor-1} 
	\Ecp{n}{\vert M^n_{k+1}(z_n) - M^n_{k}(z_n) \vert^p},
	\end{align*} 
	We then apply the above Proposition~\ref{proposition majoration increments Mn}
to bound all the terms of the sum appearing in the last display. 
Now, note that for $z \in \R$ the expression appearing in the display of Proposition~\ref{proposition majoration increments Mn} can be written as
		\begin{align}\label{eq asymptotique p tend vers un}
			-p+ \gamma\left( e^{p \Re z} -1 -p (\Re(e^z) -1)\right)
			&= -1 - \left(p-1+ \gamma (p(e^z-1)-e^{pz}+1)\right)  \notag \\
		   &\underset{p \rightarrow 1}{=} -1 -(p-1)(1+\gamma(e^z -1 -ze^z) + o(1)) \notag \\
		   &\underset{p \rightarrow 1}{=}-1 - (p-1) (f_\lambda(z)+o(1)).
	\end{align}
	Observe that $f_{\lambda}(z)>0$ since $z \in \intervalleoo{0}{z_\lambda}$. Thus if we fix {$p\in \intervalleof{1}{2}$} close to $1$ the above expression is uniformly bounded above by $-1-\eta$, for some $\eta>0$, for $z_n\in K_n$ with $n$ sufficiently large, and $\lambda \in I$.
	This entails that we have $$\sum_{k=A_n}^{\lfloor nt \rfloor-1}\Epc{n}{\vert M^n_{k+1}(z_n) - M^n_{k}(z_n) \vert^p}={O_\P}\left(\sum_{k=A_n}^{\lfloor nt \rfloor-1} k^{-1-\eta}\right),$$ {as $n\rightarrow \infty$,} uniformly in $z_{n} \in K_{n}$, {weakly uniformly in $\lambda \in I$} and the desired result follows.
\end{proof}
The last lemma of this subsection shows that $M^n_k(z)$ is close to $M^n_k(z')$ when $z$ and $z'$ are close.
\begin{lemma}\label{lemme z n proche de z}
Let $z \in (0,\infty )$ and $I\subset \intervalleoo{1}{\infty}$ a compact interval such that
	$z\in \intervalleoo{0}{z_\lambda}$ for all $\lambda \in I$.
	Let $(A_n)$ be a sequence of integers such that $A_n \to \infty$. Let also $(K_n)_{n \geq 1}$  be a sequence of compact subsets of $\mathbb{C}$ such that $\mathrm{diam}(K_n) =  o(1/\log A_n)$ and such that $z \in K_n$ for every $n \geq1$. 
	There exists {$p\in \intervalleof{1}{2}$} such that, as $n\rightarrow \infty$, uniformly in $z_n \in K_n$, weakly uniformly in $\lambda \in I$, on the event $\{\forall i \in \intervalleentier{0}{\lfloor nt \rfloor}, \ S^n_i > 0\} {\cap \{J^n \leq A_n\}}$,
		\begin{align*}
{\ind{z_n \in \mathscr{E}}}\Epc{n}{\left\vert M^n_{A_n}(z_n)-M^n_{A_n}(z) \right\vert^p}= o_\P(1).
	\end{align*}
\end{lemma}
\begin{proof}
We work on the event $\{\forall i \in \intervalleentier{0}{\lfloor nt \rfloor}, \ S^n_i > 0\}{\cap \{J^n \leq A_n\}}$.
In this proof, any term $o_\P(1)$ and $O_\P(1)$ should be understood as $n\rightarrow \infty$, uniformly in $z_n \in K_n$, weakly uniformly in $\lambda \in I$.
Recall from \eqref{eq:definition M k n de z} the fact that by definition, 
we have $M^n_{A_n}(z) = \frac{1}{C^n_{A_n}(z)} \cdot \mathcal{L}(z, \T^n_{A_n})$, {defined for any $z\in \mathscr{E}$, so in particular for $z\in\intervalleoo{0}{z_\lambda}$}. 
{If $z_n\in \mathscr{E}$, which happens for $n$ large enough,} one can then write
\begin{align*}
	M^n_{A_n}(z_n)-M^n_{A_n}(z)
	&= \frac{1}{C^n_{A_n}(z_n)} \cdot \mathcal{L}(z_n, \T^n_{A_n}) - \frac{1}{C^n_{A_n}(z)} \cdot \mathcal{L}(z, \T^n_{A_n})\\
	&= \frac{1}{C^n_{A_n}(z)} \cdot \left(\mathcal{L}(z_n, \T^n_{A_n}) - \mathcal{L}(z, \T^n_{A_n})\right) + \left(\frac{1}{C^n_{A_n}(z_n)} - \frac{1}{C^n_{A_n}(z)}\right) \cdot \mathcal{L}(z_n, \T^n_{A_n})\\
	&= \frac{1}{C^n_{A_n}(z)} \cdot \left(\mathcal{L}(z, \T^n_{A_n}) - \mathcal{L}(z_n, \T^n_{A_n})\right) + \left(\frac{C^n_{A_n}(z)-C^n_{A_n}(z_n)}{C^n_{A_n}(z)}\right) \cdot M^n_{A_n}(z_n).
\end{align*}
Fix some value {$p\in \intervalleof{1}{2}$}.
Using the last display, the inequality $\vert x+y \vert^p \le 2^p\cdot (\vert x\vert^p+\vert y \vert^p)$, and taking the expectation under $\P_n$ we get that 
\begin{multline}\label{eq: En split into 2 terms}
	\Epc{n}{\left\vert M^n_{A_n}(z_n)-M^n_{A_n}(z) \right\vert^p} 
	\leq 2^p \cdot \bigg( \frac{1}{\vert C^n_{A_n}(z) \vert^p} \cdot \Ecp{n}{\vert \mathcal{L}(z, \T^n_{A_n}) - \mathcal{L}(z_n, \T^n_{A_n})\vert^p}\\
	 + \frac{\vert C^n_{A_n}(z_n)-C^n_{A_n}(z) \vert^p}{\vert C^n_{A_n}(z)\vert^p} \cdot \Ecp{n}{\vert M^n_{A_n}(z_n)\vert^p} \bigg).
\end{multline}
We now handle the two terms appearing on the RHS of the last display separately.

\textbf{First term. } For the first term, recall that $\mathcal{L}(z, \T^n_{A_n})=\Epcsq{n}{e^{z\haut(V^n_{A_n})}}{\T^n_{A_n}}$ where the vertex $V^n_{A_n}$ is chosen uniformly among the active vertices of $\T^{n}_{A_{n}}$, conditionally on $\T^{n}_{A_{n}}$ (and similarly for $\mathcal{L}(z_n, \T^n_{A_n})$). 
Hence we can write
\begin{align*}
		\frac{1}{\vert C^n_{A_n}(z)\vert^p}
		\E_n \left[\vert \mathcal{L}(z_n, \T^n_{A_n})-\mathcal{L}(z , \T^n_{A_n})\vert^p \right]&=\frac{1}{\vert C^n_{A_n}(z) \vert^p}\E_n\left[\left\vert\Epcsq{n}{ 
			e^{z_n \haut (V^n_{A_n})}- e^{z \haut (V^n_{A_n})}
		}{\T^n_{A_n}} \right\vert^p \right]\\
		& \leq \frac{1}{\vert C^n_{A_n}(z) \vert^p}\E_n\left[\left\vert 
		e^{z_n \haut (V^n_{A_n})}- e^{z \haut (V^n_{A_n})}
		\right\vert^p\right]
\end{align*}
where we use Jensen's inequality to get the second line.
We then rewrite the RHS as
\begin{align*} 
	\frac{1}{\vert C^n_{A_n}(z) \vert^p} \E_n\left[\left\vert 
	e^{z_n \haut (V^n_{A_n})}- e^{z \haut (V^n_{A_n})}
	\right\vert^p\right]
	=
	\E_n\left[\frac{e^{pz\haut(V^n_{A_n})}}{\vert C^n_{A_n}(z) \vert^p} \cdot \left\vert e^{(z_n-z)\haut(V^n_{A_n})}-1
	\right\vert^{p}\right].
\end{align*} 
By Hölder's inequality, we have for all $q>1$,
\begin{align}\label{eq:holder ineq} 
	\E_n \left[
	\frac{e^{pz \haut(V^n_{A_n})}}{\vert C^n_{A_n}(z)\vert^p}\left\vert e^{(z_n-z)\haut(V^n_{A_n})}-1
	\right\vert^{p}
	\right]
	\le 
	\left(\E_n \left[\frac{e^{pqz \haut(V^n_{A_n})}}{\vert C^n_{A_n}(z)\vert^{pq}}
	\right]\right)^{\frac{1}{q}}
	\left(
	\E_n\left[\left\vert
	e^{(z_n-z) \haut(V^n_{A_n})} -1
	\right\vert^{pq'}
	\right]
	\right)^{\frac{1}{q'}}
\end{align} 
where $q'$ is chosen so that $\frac{1}{q}+ \frac{1}{q'}=1$. 
By taking $p,q$ close enough to $1$, one obtains that 
\begin{align}\label{eq:first factor is O1} 
	\E_n \left[\frac{e^{pqz \haut(V^n_{A_n})}}{\vert C^n_{A_n}(z)\vert^{pq}}
	\right] =\frac{C^n_{A_n}(pqz)}{\vert C^n_{A_n}(z)\vert^{pq}}= O_\P(1)
\end{align} 
thanks to Lemma~\ref{lem:asymptotics Cnkz} and to \eqref{eq asymptotique p tend vers un}, so we just need to study the second factor appearing on the RHS of \eqref{eq:holder ineq}. 

For that, first note that $\vert e^w - 1\vert^{q'} \le 2^{q'} e^{q' \vert \Re (w)\vert}$ for all $w \in \mathbb{C}$ 
and that for all $w \in \mathbb{C}$ such that $\vert w \vert \le 1$, we have $\vert e^w -1 \vert \le e \vert w \vert$, so that using the latter inequality on the event where it is possible and the former otherwise we get
\begin{align} 
	\E_n\left[\left\vert
	e^{(z_n-z) \haut(V^n_{A_n})} -1
	\right\vert^{pq'}
	\right]
	\le &2^{pq'} \E_n\left[e^{pq'\vert \Re(z_n)-z \vert \haut(V^n_{A_n})}
	\ind{\vert z_n -z \vert \haut(V^n_{A_n}) \ge 1 }\right] \label{eq premiere somme z n proche de z} \\
	&+ e^{pq'} \E_n \left[  \vert z_n -z \vert^{pq'} \haut(V^n_{A_n})^{pq'} \ind{\vert z_n -z \vert \haut(V^n_{A_n}) \le 1}
	\right].\label{eq deuxieme somme z n proche de z}
\end{align} 
Moreover, by  Theorem~3 (1) of \cite{BBKK23+}, we know that
\begin{equation*}
	\E_n \left[ \haut(V^n_{A_n})^{pq'}\right] = (1+ o_\P(1)) \left(  \sum_{i=1}^{A_n} \frac{1}{S_i^n} \ind{\XX^n_i=1}\right)^{pq'}  = (1+ o_\P(1)) \left(\frac{\lambda}{\lambda-1}\log A_n\right)^{p q'}
\end{equation*}
as $n\rightarrow \infty$, weakly uniformly in $\lambda \in I$,
where the second equality comes from Lemma~\ref{lem:sum 1/Sni is gamma log k} (the uniformity in $\lambda$ is a consequence of the bounds in the proof of Theorem~3 (1) of \cite{BBKK23+}). 
Therefore, taking the fact that $\mathrm{diam}(K_n) = o(1/ \log A_n)$ into account, we deduce that the second term
\eqref{eq deuxieme somme z n proche de z} is $o_\P(1)$.
But using Markov's inequality, the above identities also imply that
\begin{align*}
\Ppp{n}{\vert z_n -z \vert \haut(V^n_{A_n} )\ge 1} \leq \vert z_n -z \vert^{pq'} \Ecp{n}{\haut(V^n_{A_n} )^{pq'}} = o_\P(1).
\end{align*}
Now using the Cauchy-Schwarz inequality on the term \eqref{eq premiere somme z n proche de z} and then the above display, we get
\begin{align*}
\Ecp{n}{e^{pq'\vert \Re(z_n)-z \vert \haut(V^n_{A_n})}
\ind{\vert z_n -z \vert \haut(V^n_{A_n}) \ge 1 }}
&\leq \Ecp{n}{\left(e^{pq'\vert \Re(z_n)-z \vert \haut(V^n_{A_n})}\right)^2}^\frac{1}{2} \cdot 
	\Ppp{n}{\vert z_n -z \vert \haut(V^n_{A_n}) \ge 1 }^{\frac{1}{2}}\\
& =C^n_{A_n}\left(2pq' \vert \Re(z_n)-z \vert \right)^{\frac{1}{2}} \cdot o_\P(1).
\end{align*}
Using \eqref{eq approximation somme un sur S} one obtains that
\begin{align*}
C^n_{A_n}\left(2pq' \vert \Re(z_n)-z \vert \right)
	&=\prod_{i=J^n{+1}}^{A_n} \left(1+ \frac{1}{S^n_i} \left(e^{2pq' \vert \Re(z_n)-z \vert }-1\right)\ind{\XX^n_i=1} \right) \\
	&\le \exp\left(\sum_{i=1}^{A_n} \frac{1}{S^n_i}\left(e^{2pq' \vert \Re(z_n)-z \vert }-1\right)\ind{\XX^n_i=1} 
	\right) \\
	&= \exp \left( (\gamma \log (A_n)+O_\P(1)) \cdot \left(e^{2pq' \vert \Re(z_n)-z \vert }-1\right) \right) \\
	&= 1+ o_\P(1),
\end{align*}
where the last line follows from the fact that $\mathrm{diam}(K_n) =  o(1/\log A_n)$. 
This entails that the term \eqref{eq premiere somme z n proche de z} is $o_\P(1)$. 
Putting together \eqref{eq:holder ineq}, \eqref{eq:first factor is O1}, and then the fact that the two terms  \eqref{eq premiere somme z n proche de z} and \eqref{eq deuxieme somme z n proche de z} are $o_\P(1)$ we have proved that if $p$ and $q$ are chosen sufficiently close to $1$ then 
\begin{align}\label{eq:first term of En}
	\frac{1}{\vert C^n_{A_n}(z)\vert^p}
	\E_n \left[\vert \mathcal{L}(z_n, \T^n_{A_n})-\mathcal{L}(z , \T^n_{A_n})\vert^p \right]&= o_\P(1).
\end{align}

\textbf{Second term. } Now we focus on the second term appearing in \eqref{eq: En split into 2 terms}. 
We first use the fact that $C^n_{A_n}(z)=\Ecp{n}{\mathcal{L}(z , \T^n_{A_n})}$ and similarly for $z_n$,  and then use Jensen's inequality to get 
\begin{align}\label{eq laplace en zn proche de z 2}
	\frac{1}{\vert C^n_{A_n}(z)\vert^p} \vert C^n_{A_n}(z_n)-C^n_{A_n}(z) \vert^p 
	&= \frac{1}{\vert C^n_{A_n}(z)\vert^p}\cdot \left\vert \Ecp{n}{\mathcal{L}(z_n, \T^n_{A_n})-\mathcal{L}(z , \T^n_{A_n})} \right\vert^p \notag \\
	&\leq
	\frac{1}{\vert C^n_{A_n}(z)\vert^p}
	\E_n \left[\vert \mathcal{L}(z_n, \T^n_{A_n})-\mathcal{L}(z , \T^n_{A_n})\vert^p \right]
	=  o_\P(1),
\end{align}
where the last equality comes from \eqref{eq:first term of En}.
Last, we can use Lemma~\ref{lemme majoration moment martingale}, Lemma~\ref{lem:asymptotics Cnkz} and \eqref{eq asymptotique p tend vers un} to show that for {$p\in \intervalleof{1}{2}$} small enough we have $\E_n[\vert M^n_{A_n}(z_n)\vert^p]=O_\P(1)$. 

\textbf{Conclusion. } Plugging the results proved above back into \eqref{eq: En split into 2 terms}, we get that for {$p\in \intervalleof{1}{2}$} sufficiently small, on the event $\{\forall i \in \intervalleentier{0}{\lfloor nt \rfloor}, \ S^n_i > 0\}$, we have   
\begin{align*}
	\Epc{n}{\left\vert M^n_{A_n}(z_n)-M^n_{A_n}(z) \right\vert^p} \leq 2^p\cdot \left( o_\P(1) + o_\P(1) O_\P(1)\right) = o_\P(1),
\end{align*}
as $n\rightarrow \infty$, uniformly in $z_n \in K_n$, weakly uniformly in $\lambda\in I$. 
This is what we wanted to prove.
\end{proof}
\subsection{Convergence of $(M^n_{\lfloor nt \rfloor}(z))_n$ via the martingales $(M_k(z))_k$}
\label{sous-section convergence des martingales}
The goal of this section is to prove (a quantitative version of) the convergence of $(M^n_{\lfloor nt \rfloor}(z))$ as $n\rightarrow \infty$ for suitable complex parameters $z$. 
Roughly {speaking}, this is double limit problem: we want to take the limit of $M^n_{k}(z)$ as both $n$ and $k$ go to infinity together. 
We split the problem into two parts. 
First we study the process $(M_k(z))_k$ defined in \eqref{eq:definition Mkz}, which is in some sense the limit of $(M^n_k(z))_k$ as $n\rightarrow \infty$. Relying on the fact that this process is a martingale, we prove that $M_k(z) \rightarrow M_\infty(z)$ as $k\rightarrow \infty$, see Proposition~\ref{proposition cv unif martingale} and Proposition~\ref{prop cv lp martingale} below. 
Second, we then argue that when $n$ is large, for some range of values of $k$, the quantities $M^n_k(z)$ and $M_k(z)$ are close together, thus proving that $M^n_{\lfloor nt \rfloor}(z) \rightarrow M_\infty(z)$ as $n\rightarrow \infty$. This is the content of Theorem~\ref{th: convergence vers M infini}.

In this subsection, we fix  a compact interval $I\subset \intervalleoo{1}{\infty}$ and $t\in \intervalleoo{0}{\infty}$ such that $t\in \intervalleoo{0}{t_\lambda}$ for all $\lambda \in I$.
Recall from \eqref{eq:defJ} the definition of $J$ and  from \eqref{eq:definition Ckz} and \eqref{eq:definition Mkz}, the fact that on the event $\{\forall k\geq 0,\ S_k>0\}$, for all $k\geq J$ and $z\in \mathscr{E}'$ we have  $C_k(z) \neq 0$ so that $M_k(z)$ is well-defined. 
\begin{lemma}\label{lemme majoration C}
	There exists a random analytic function $c_\lambda(z)$ such that,  
	{for any compact complex domain $K$ that satisfies $K\subset \mathscr{E}(\lambda)$ for all $\lambda\in I$,}
	on the event $\{\forall i\geq 0, \ S_i>0 \}{\cap \{k\geq J\}}$, 
	we have 
	\begin{align*} 
		C_k(z)= \exp\left( \gamma (e^z-1) \log k + c_\lambda(z) + o_{\P}(1)\right)
	\end{align*} 
where the $o_\P(1)$ holds almost surely as $k\to \infty$, uniformly in $z\in K$ weakly uniformly in $\lambda \in I$.
\end{lemma}
\begin{proof}
	Recall {from \eqref{eq:definition Ckz} that by definition, on the event $\{\forall i\geq 0, \ S_i>0 \}$ we have}
	\begin{align*} 
		C_k(z) = \prod_{i=J{+1}}^k \left(1+\frac{1}{S_i} (e^z-1) \ind{\XX_i=1} \right).
	\end{align*} 
{As in the proof of Lemma~\ref{lem:asymptotics Cnkz}, we can write 
\begin{align*}
	C_k(z)&=\exp\left(\sum_{i=J{+1}}^k \log \left(1+\frac{1}{S_i} (e^z-1) \ind{\XX_i=1} \right) \right) \\
	&= \exp\left(
	(e^z-1) \sum_{i=J{+1}}^k
	\frac{1}{S_i}  \ind{\XX_i=1}
	+
	\sum_{i=J{+1}}^k
	\left(
	\log\left(1+ \frac{1}{S_i} (e^z -1) \ind{\XX_i=1} \right)
	-\frac{1}{S_i} (e^z-1) \ind{\XX_i=1}
	\right)
	\right).
\end{align*}
By Lemma~\ref{lem:convergence sum 1/Sk - c log k towards rv}, we can handle the first term in the exponential as 
\begin{align*}
	(e^z-1) \sum_{i=J{+1}}^k
	\frac{1}{S_i}  \ind{\XX_i=1} = (e^z-1) \cdot \frac{\lambda}{\lambda -1} \log k + \left( Z{(\lambda)} - (e^z-1) \sum_{i=1}^{{J}}
	\frac{1}{S_i}  \ind{\XX_i=1} \right) + o_\P(1),
\end{align*} 
where the $o_\P(1)$ is almost sure, uniformly in $z\in K$, and weakly uniformly in $\lambda \in I$.
The second term can be written
\begin{multline*}
\sum_{i=J{+1}}^\infty
	\left(
	\log\left(1+ \frac{1}{S_i} (e^z -1) \ind{\XX_i=1} \right)
	-\frac{1}{S_i} (e^z-1) \ind{\XX_i=1}
	\right)\\ - \sum_{i=k+1}^\infty
	\left(
	\log\left(1+ \frac{1}{S_i} (e^z -1) \ind{\XX_i=1} \right)
	-\frac{1}{S_i} (e^z-1) \ind{\XX_i=1}
	\right).
\end{multline*}
We can then use the inequality $\vert \log(1+w)-w\vert \leq \vert w \vert^2$ 
valid for all $w$ such that $\vert w \vert \leq \frac{1}{2}$, and so for all large enough $k$, to get 
	\begin{align*} 
	\left\vert\sum_{i=k+1}^\infty \left(
		\log\left(1+ \frac{1}{S_i} (e^z -1) \ind{\XX_i=1} \right)
		-\frac{1}{S_i} (e^z-1) \ind{\XX_i=1}\right)
		\right\vert 
		&\leq C\vert e^z -1\vert^2 \cdot \sum_{i=k+1}^\infty \frac{1}{(S_i)^2} \\ 
		&\underset{\eqref{eq:definition of M lambda}}{\leq}   C \vert e^z -1\vert^2 \cdot M \cdot \sum_{i=k+1}^\infty \frac{1}{i^2} = o_\P(1)
	\end{align*} 
	where the $o_\P(1)$ is almost sure, uniformly in $z\in K$, and weakly uniformly in $\lambda \in I$}. 

In the end, this ensures that the statement of {the} lemma holds with 
\begin{align*}
	c_\lambda(z)= \left( Z - (e^z-1) \sum_{i=1}^{{J}}
	\frac{1}{S_i}  \ind{\XX_i=1} \right) + \sum_{i=J{+1}}^\infty
	\left(
	\log\left(1+ \frac{1}{S_i} (e^z -1) \ind{\XX_i=1} \right)
	-\frac{1}{S_i} (e^z-1) \ind{\XX_i=1}
	\right).
\end{align*}
	Note that this function is analytic as it is a uniform limit of analytic functions.  
\end{proof}

Recall the notation $S=(S_{n})_{n \geq 0}$.

\begin{lemma}\label{lemme majoration M}
	Let {$p\in \intervalleof{1}{2}$}. We have, for every compact complex domain $K$ such that $K\subset \mathscr{E}(\lambda)$ for all $\lambda\in I$, 
	\begin{align*} 
		\ind{k\geq J}\ind{\forall i\geq 0, \ S_i>0}\cdot \Ecsq{\left\vert
			M_{2k}(z) - M_k(z) 
			\right\vert^p
		}{S} = O_\P\left(
		k^{\gamma(e^{p\Re z}-1-p(\Re(e^z)-1))-p+1+ o_\P(1)}
		\right)
	\end{align*} 
	and
	\begin{align*} 
		\ind{k\geq J} \ind{\forall i\geq 0, \ S_i>0}\cdot \Ecsq{\vert M_k(z) \vert^p}{S}
		= O_\P\left(
		k^{(\gamma(e^{p\Re z}-1-p(\Re(e^z)-1))-p+1)\vee 0 + o_\P(1)}
		\right)
	\end{align*} 
	where the $O_\P$ and $o_\P$ appearing above hold almost surely as $k\rightarrow \infty$, 
	uniformly in $z\in K$, weakly uniformly in $\lambda \in I$.
\end{lemma}
\begin{proof}
The same proof as the one of Proposition~\ref{proposition majoration increments Mn} goes through, using this time Lemma~\ref{lemme majoration C}: using exactly the same computations as in the proof of Proposition~\ref{proposition majoration increments Mn}, we see that
\[
\ind{k\geq J} \ind{\forall i\geq 0, \ S_i>0} \E\left[\left.\left\vert M_{k+1}(z)- M_k(z) \right\vert^p \right\vert S\right] = O_\P\left(\frac{1}{k^p} \frac{C_k(p\mathrm{Re}(z))}{\vert C_k(z)\vert^p}\right)
\]
almost surely as $k\to \infty$, uniformly in $z \in K$, weakly uniformly in $\lambda \in I$. By Lemma~\ref{lemme majoration C}, we deduce that
\begin{align*}
\ind{k\geq J} \ind{\forall i\geq 0, \ S_i>0} &\E\left[\left.\left\vert M_{k+1}(z)- M_k(z) \right\vert^p \right\vert S\right] \\
&= O_\P\left( \frac{1}{k^p} \exp(\gamma(e^{p\mathrm{Re}(z)} -1 -p(\mathrm{Re}(e^z)-1)) \log k + O_\P(1))\right)\\
&=O_\P\left(k^{\gamma(e^{p\mathrm{Re}(z)} -1 -p(\mathrm{Re}(e^z)-1))-p+o_\P(1)}\right)
\end{align*}
almost surely as $k\to \infty$, uniformly in $z \in K$, weakly uniformly in $\lambda \in I$. So, by Lemma~1 of \cite{Big92} (see also Lemma~A.2 of \cite{Sen21}),
\begin{align*}
	\ind{k\geq J}\ind{\forall i\geq 0, \ S_i>0}\cdot \Ecsq{\left\vert
		M_{2k}(z) - M_k(z) 
		\right\vert^p
	}{S}& \le \ind{k\geq J}\ind{\forall i\geq 0, \ S_i>0}2^p \sum_{j=k}^{2k-1} \Ecsq{\left\vert M_{j+1}(z)- M_j(z) \right\vert^p}{S}\\
	&=O_\P\left(k^{\gamma(e^{p\mathrm{Re}(z)} -1 -p(\mathrm{Re}(e^z)-1))-p+1+o_\P(1)}\right),
\end{align*}
and similarly
\begin{align*}
		\ind{k\geq J}\ind{\forall i\geq 0, \ S_i>0}\cdot \Ecsq{\left\vert
		M_{k}(z)
		\right\vert^p
	}{S}\le 	&\ind{k\geq J}\ind{\forall i\geq 0, \ S_i>0}\cdot \Ecsq{\left\vert
	M_J(z) 
	\right\vert^p
	}{S}\\
	&+ \ind{k\geq J}\ind{\forall i\geq 0, \ S_i>0}2^p \sum_{j=J}^{k-1} \Ecsq{\left\vert M_{j+1}(z)- M_j(z) \right\vert^p}{S}\\
	=&O_\P\left(
	k^{(\gamma(e^{p\Re z}-1-p(\Re(e^z)-1))-p+1)\vee 0 + o_\P(1)}
	\right),
\end{align*}
where the last line comes from the fact that $\sum_{j=J}^k j^\alpha = O_\P(k^{(\alpha+1)\vee 0})$ almost surely as $k \to \infty$ for all $\alpha \in \R$. This ends the proof.
\end{proof}
Let 
\begin{align*} 
	\mathcal{V}=\mathcal{V}(\lambda) \coloneqq	 \bigcup_{{p\in\intervalleof{1}{2}}} \enstq{ z\in \mathbb{C}}{  
	\Re z<z_\lambda \text{ and }
	\gamma(e^{p\Re z}-1-p(\Re(e^z)-1))-p+1<0
	}.
\end{align*} 
The proposition below will be useful for the next section.
\begin{proposition}\label{proposition cv unif martingale}
Fix $\lambda \in I$.
	On the event $\{\forall k \geq 0, \ S_k>0\}$,
	for every compact $K$ such that $K \subset \mathcal{V}(\lambda)$ we have
	\begin{align*} 
		\ind{k\geq J}  \ind{\forall i\geq 0, \ S_i>0}\cdot \left\vert
		M_k(z) - M_\infty(z)
		\right\vert
		= o_\P(1)
	\end{align*}
where the $o_\P$ holds almost surely as $k\rightarrow \infty$, uniformly in $z\in K$.
\end{proposition}
\begin{proof}
The same proof as the one of Proposition~3.6 of \cite{Sen21} carries through, as a consequence of Lemma~A.3 of \cite{Sen21}  and Lemma~\ref{lemme majoration M}.
\end{proof}
Note that a similar result to Proposition~\ref{proposition cv unif martingale} with the $o_\P(1)$ holding weakly uniformly in $\lambda \in I$ can be obtained via a straightforward adaptation of Lemma~A.3 of \cite{Sen21}. We won't need the stronger result in this paper. 
Next, we state an $L^{p}$ martingale convergence result.
\begin{proposition}\label{prop cv lp martingale}
	For any compact subset {$K$ that satisfies $K\subset\mathcal{V}(\lambda)$ for all $\lambda\in I$}, there exists {$p\in \intervalleof{1}{2}$} such that 
	\begin{equation*}
		{\ind{k\geq J}\cdot }\ind{\forall i \ge 0, \ S_i >0}\E \left[ \left.\left\vert  M_k(z)- M_\infty(z)\right\vert ^p\right\vert S\right] = o_\P(1),
	\end{equation*}
{where the $o_\P(1)$ holds almost surely as $k \to \infty$, uniformly in $z \in K$, weakly uniformly in $\lambda \in I$.}
\end{proposition}
\begin{proof}
	By Lemma~\ref{lemme majoration M}, taking {$p\in \intervalleof{1}{2}$} small enough so that for all $z \in K$ {and for all $\lambda \in I$} we have
	$\gamma(e^{p\Re z}-1-p(\Re(e^z)-1))-p+1<0$, one gets that uniformly in $z \in K$, as $k \to \infty$,
	\begin{align} \label{eq:sum of increments over powers of 2}
		\ind{k\geq J}\cdot \sum_{j\ge 0} \ind{\forall i \ge 0,\ S_i >0} \Ecsq{ \vert M_{2^jk}(z)-M_{2^{j+1}k}(z)\vert^p}{ S} = o_\P(1),
	\end{align}
	where the $o_\P(1)$ holds almost surely.
	But, by Lemma~1 of \cite{Big92} (see also Lemma~A.2 of \cite{Sen21}), for all $\ell\ge k$,
	\begin{align*}
	{\ind{k\geq J}\cdot }\ind{\forall i \ge 0,\ S_i >0} \E \left[\left. \left\vert M_k(z) - M_{2^\ell k}(z)\right\vert ^p\right\vert S\right]\le 	{\ind{k\geq J}}\cdot 2^p\sum_{j =  0}^{\ell-1} \ind{\forall i \ge 0,\ S_i >0} \E\left[\left. \vert M_{2^jk}(z)-M_{2^{j+1}k}(z)\vert^p\right\vert S\right], 
	\end{align*}
	so that by Fatou's lemma,
	\begin{align*}
	{\ind{k\geq J}\cdot } \ind{\forall i \ge 0,\ S_i >0} \E \left[\left.\left\vert M_k(z) - M_{\infty}(z)\right\vert ^p\right\vert S\right]\le 	{\ind{k\geq J}}\cdot 2^p\sum_{j\ge 0} \ind{\forall i \ge 0,\ S_i >0} \E\left[\left. \vert M_{2^jk}(z)-M_{2^{j+1}k}(z)\vert^p\right\vert S\right], 
	\end{align*}
	This completes the proof thanks to \eqref{eq:sum of increments over powers of 2}. 
\end{proof}

Finally, we use the results of this section and the previous one to make a connection between $M^n_k(z)$ and $M_\infty(z)$. 
\begin{theorem}\label{th: convergence vers M infini}
	Let $t,z$ be positive real numbers such that  $t \in (0,t_\lambda)$ and $z \in (0,z_\lambda)$ for all $\lambda\in I$. 
	Let $(K_n)$  be a sequence of compact subsets of {$\mathscr{E}$} such that $\mathrm{diam}(K_n) \to 0$ as $n\to \infty$ and such that $z \in K_n$ for every $n \geq 1$. 
	Then there exists {$p\in \intervalleof{1}{2}$} such that 
	\begin{align*} 
	\ind{\lfloor nt \rfloor \geq J}\cdot  {\ind{\tau_n \geq \lfloor nt \rfloor} }\ind{\forall k\ge 0, \ S_k>0}\cdot \Ecsq{\vert M^n_{\lfloor nt \rfloor}(z_n)-M_\infty(z) \vert^p }{S^n,S}
	=o_\P(1),
	\end{align*}
{as $n\rightarrow \infty$, uniformly in $z_n \in K_n$, weakly uniformly in $\lambda \in I$.} 
\end{theorem}
\begin{proof}
		Let $A_n$ be a sequence of integers with $A_n \rightarrow \infty$ 
		and $A_n\leq \lfloor nt \rfloor$ for all $n\geq 1$, 
		and $\log(A_n)= o(1/\mathrm{diam}(K_n))$. 
	On the event $E_n \coloneqq 
{\{J^n \leq A_n\}\cap \{J \leq A_n\}}
	\cap \{\tau_n \geq \lfloor nt \rfloor\} \cap \{\forall k\ge 0, \ S_k>0\}$, 
	for $z_n\in K_n$ we can write
	\begin{multline*}
		\Ecsq{\lvert M^n_{\lfloor nt \rfloor}(z_n)-M_\infty(z) \rvert^p}{S^n,S}\\ \leq 4^p \cdot 
		\bigg(
		\Ecsq{\vert M^n_{\lfloor nt \rfloor}(z_n)-M^n_{A_n}(z_n) \vert^p}{S^n,S}
		+
		\Ecsq{\vert M^n_{A_n}(z_n) - M^n_{A_n}(z) \vert^p}{S^n,S}\\
		+ 
		\Ecsq{\vert M^n_{A_n}(z) - M_{A_n}(z)  \vert^p}{S^n,S}
		+ 
		\Ecsq{\vert M_{A_n}(z) - M_\infty(z)  \vert^p}{S^n,S}
		\bigg).
	\end{multline*}
The LHS is the quantity that we want to show is $o_\P(1)$ on $E_n$ and the RHS is a sum of 4 terms. 
The first term is $o_\P(1)$ on $E_n$  by Corollary~\ref{corollaire M nt proche de M An}, 
the second term is $o_\P(1)$  on $E_n$  by Lemma~\ref{lemme z n proche de z}
and the fourth term is $o_\P(1)$ on $E_n$  thanks to Proposition~\ref{prop cv lp martingale}.
It remains to show that the third term in $o_\P(1)$ on $E_n$ as well.

First, on the event $E_n\cap \{J^n=J\}\cap \{(S^n_k)_{1\leq k \leq A_n} = (S_k)_{1\leq k \leq A_n}\}$, the terms $C^n_{A_n}(z)$ and $C_{A_n}(z)$ are equal and by the coupled construction \eqref{eq:coupling of the freezing and attachment vertices}, 
the trees $\T_{A_n}$ and $\T^n_{A_n}$ are identical, so $M^n_{A_n}(z)=M_{A_n}(z)$.
This entails that 
\begin{align}\label{eq:martingale difference is 0 on some event}
\mathbbm{1}_{E_n\cap \{J^n=J\}\cap \{(S^n_k)_{1\leq k \leq A_n} = (S_k)_{1\leq k \leq A_n}\}}\cdot \Ecsq{\vert M^n_{A_n}(z) - M_{A_n}(z)  \vert^p}{S^n,S}=0.
\end{align}

{Using the coupling of Section~\ref{sous-section preliminaire couplage},
since the convergence of $(S^n_k)_{k\ge 0}$ towards $(S_k)_{k\ge 0}$ for the product topology holds almost surely and uniformly in $\lambda \in I$ for any compact interval $I\subset (1,\infty)$, we may choose an integer sequence $(A_n)$ that grows slowly enough so that $(S^n_k)_{0\le k \le A_n}=(S_k)_{0\le k \le A_n}$ with probability $1-o(1)$, as $n\rightarrow \infty$, uniformly in $\lambda\in I$, so that $\mathbbm{1}_{\{(S^n_k)_{1\leq k \leq A_n} = (S_k)_{1\leq k \leq A_n}\}}={1-}o_\P(1)$. By Lemma~\ref{lem:JnJ}, we also have $\mathbbm{1}_{E_n \cap \{J^n = J\}}={1-}o_\P(1)$. This proves that
\begin{align*}
	\mathbbm{1}_{E_n\cap \{J^n=J\}\cap \{(S^n_k)_{1\leq k \leq A_n} = (S_k)_{1\leq k \leq A_n}\}} = \mathbbm{1}_{E_n} + o_\P(1).
\end{align*}
Combining this with \eqref{eq:martingale difference is 0 on some event} ensures that $\Ecsq{\vert M^n_{A_n}(z) - M_{A_n}(z)  \vert^p}{S^n,S}=o_\P(1)$ on the event $E_n$, which finishes the proof.
}\end{proof}

\subsection{The limiting function $z\mapsto M_{\infty}(z)$ does not vanish on the interval $\intervalleoo{0}{z_\lambda}$}
\label{sous-section limite non nulle}
In this subsection, 
we fix $\lambda\in \intervalleoo{1}{\infty}$ and
we prove that almost surely, $M_\infty(z)>0$ for all $ z \in (0,z_\lambda)$.
We first check that the number of zeros is countable.
\begin{lemma}\label{lemme limite martingale non nulle z fixe}
	On the event $\{\forall k \geq 0, \ S_k>0\}$, a.s.\@ for all $z \in (0,z_\lambda)$ we have $\Ppsq{M_\infty(z) \neq 0}{S}=1$. 
	In particular, we have $\Ppsq{\forall z\in (0,z_\lambda),\ M_\infty(z) =0}{S}=0$ so by analyticity of 
	the function $z\mapsto M_\infty(z)$,
	it almost surely has only a countable number of zeros on $(0,z_\lambda)$.
\end{lemma}
\begin{proof}
	It follows from exactly the same proof as in the proof of Lemma~3.10 of \cite{Sen21} using Proposition~\ref{proposition cv unif martingale}, considering for all $N\ge {J}$ the martingale $(M^{(N)}_k(z))_{k\ge N}$ defined by
	\begin{equation}\label{eq martingale N premiers pas identifies}
		\forall k\ge N, \qquad M^{(N)}_k(z) = \frac{1}{C_k(z)}\frac{1}{S_k} \sum_{\substack{u \in \T_k \\ \text{active}}} e^{zd(u,\T_N)},
	\end{equation}
	where $\T_N$ is viewed as a subtree of $\T_k$ and applying Kolmogorov's 0-1 law conditionally on $S$.
\end{proof}
\begin{proposition}\label{prop: M infini ne s'annule pas}
	On the event $\{\forall k \geq 0, \ S_k>0\}$, the function $z\mapsto M_\infty(z)$ almost surely has no zero in $(0,z_\lambda)$.
\end{proposition}
\begin{proof}
	 Let $\mathcal{A}$ be the event 
	 \[\mathcal{A}  \coloneqq\{z\mapsto M_\infty(z) \text{ has no zero in }(0,z_\lambda)\},\] 
	 which is measurable since $z\mapsto M_\infty(z)$ is continuous. 
	The goal of the proof is hence to prove that on the event $\{\forall k\geq 0, \ S_k>0\}$ we have $\Ppsq{\mathcal{A}}{S} = 1$ almost surely.
	
First, we prove that $\Ppsq{\mathcal{A}}{S} \in \{0,1\}$ almost surely on the event $\{\forall k\geq 0, \ S_k>0\}$.
	Using the martingale introduced in \eqref{eq martingale N premiers pas identifies}, as in the proof of Lemma~3.10 in \cite{Sen21}, for all $N \ge {J}$, we have the inequality for all $k \ge 0$,
	\begin{align*}
		(1 \wedge e^z)^N M^{(N)}_k(z) \le M_k(z) \le (1 \vee e^z)^N M^{(N)}_k(z),
	\end{align*}
	so that by taking the limit when $k\to \infty$, one gets that the function $z\mapsto M_\infty(z)$ has a zero in $(0,z_\lambda)$ if and only if the function $z\mapsto \limsup_{k \to \infty} M^{(N)}_k(z)$ has a zero in $(0,z_\lambda)$. 
	But by construction of the martingale in \eqref{eq martingale N premiers pas identifies}, the function $z \mapsto \limsup_{k \to \infty}M_k^{(N)}(z)$ does not depend on the first $N$ steps of the construction of the tree. 
	So the event $\mathcal{A}$ belongs to the tail $\sigma$-algebra generated by the uniform random variables
	$\widetilde{U}_1, \widetilde{U}_2, \dots$ that we use in \eqref{eq:coupling of the freezing and attachment vertices} 
	 to determine which vertex is frozen or to which active vertex we attach a new one.
By the Kolomogorov $0$-$1$ law, this ensures that $\Ppsq{\mathcal{A}}{S} \in \{0,1\}$.

The rest of the proof is dedicated to proving that $\Ppsq{\mathcal{A}}{ S}>0$ almost surely on the event $\{\forall k\geq 0, \ S_k>0\}$.
Let
		\[\sigma_1\coloneqq \inf\enstq{k\geq 0}{S_k=2 \ \text{and} \ \forall i \in \intervalleentier{0}{k-1}, S_i>0}.\]
On the event $\{\sigma_1 <\infty\}$, we consider, for all $k\ge \sigma_1$, the subtrees $\T^{v(1)}_k$ and $\T^{w(1)}_k$ made of the descendants of the two active vertices $v(1),w(1)$ at time $\sigma_1$. 
	Let $S^{v(1)}_k$ and $S^{w(1)}_k$ be the number of active vertices in $\T^{v(1)}_k$ and $\T^{w(1)}_k$. 
From the dynamics of the construction, conditionally on $S$, the sequence $(S^{v(1)}_k, S^{w(1)}_k)_{k\geq \sigma_1}$ evolves as the number of blue and red balls in \emph{time-dependent P\'olya urn with removals} with starting composition $(1,1)$ and replacement sequence $(\XX_k)_{k\geq \sigma_1 +1}$, as defined in Section~\ref{sec:polya}. 
Note that thanks to Lemma~\ref{lem:convergence proportion urn} below, on the event $\{\forall k\geq 0, \ S_k>0\}$, we have the almost sure convergences $S^{v(1)}_k /(S^{v(1)}_k+S^{w(1)}_k) \underset{k \rightarrow \infty }{\rightarrow} Z_1 $ as well as 
	\begin{align}\label{eq:convergence towards limiting proportion urn}
		\frac{1}{k} \sum_{i=\sigma_1}^{k-1} \ind{S^{v(1)}_{i+1} - S^{v(1)}_i \neq 0}  \quad \mathop{\longrightarrow}^{\textrm{a.s.}}_{k \rightarrow \infty} \quad  Z_1 
		\qquad \text{and} \qquad 
		\frac{1}{k} \sum_{i=\sigma_1}^{k-1} \ind{S^{w(1)}_{i+1} - S^{w(1)}_i \neq 0}  \quad \mathop{\longrightarrow}^{\textrm{a.s.}}_{k \rightarrow \infty} \quad 1-Z_1.
	\end{align}
for some random variable $Z_1$. 

Now, on the event  
$\{\sigma_1 <\infty\}\cap \{\forall k> \sigma_1, \ S_k \geq 2\} $
we almost surely have 
\[\forall k\geq \sigma_1, \ S_k \geq 2 \qquad \text{and} \qquad \sum_{k=\sigma_1}^{\infty} \frac{\XX_k^2}{S_k^2}<\infty\] so that by Proposition~\ref{prop urnes}, 
we have
$\Ppsq{Z_1\in \intervalleoo{0}{1}}{S}>0$.
	For all $k \ge \sigma_1$, for all $z \in \mathbb{C}$, let
	\begin{align*} 
	&C^{v(1)}_k(z) \coloneqq\prod_{i=\sigma_1+1}^{k} \left(1+\frac{1}{S^{v(1)}_i}\ind{S^{v(1)}_i-S^{v(1)}_{i-1}=1}(e^z-1)\right),\\
	&C_k^{w(1)}(z)\coloneqq\prod_{i=\sigma_1+1}^{k} \left(1+\frac{1}{S^{w(1)}_i}\ind{S^{w(1)}_i-S^{w(1)}_{i-1}=1}(e^z-1)\right).
	\end{align*} 
	Set 
	\[\mathcal{W}\coloneqq \enstq{ z \in \mathcal{V}}{ \mathrm{Im}(z) \in \intervalleff{-\frac{\pi}{4}}{ \frac{\pi}{4}}}.
	\]
	 Then, for all $z \in \mathcal{W}$, we have $\mathrm{Re}(e^z)> 0$ so that for all $\ell\ge 1$ we have
	$\mathrm{Re}\left(  1 + \frac{1}{\ell} (e^z -1)\right) >0$.
	In particular, for all $z \in \mathcal{W}$, we have $C^{v(1)}_k(z)\neq 0$ (resp.\@ $C_k^{w(1)}(z) \neq 0$) for all $k\ge \sigma_1$ such that $S^{v(1)}_k>0$ (resp.\@ $S^{w(1)}_k>0$) and we set
	\begin{align*} 
	M^{v(1)}_k(z) \coloneqq\frac{1}{C^{v(1)}_k(z)} \frac{1}{S^{v(1)}_k} \sum_{\substack{u \in \T^{v(1)}_k\\ \text{active}}} e^{z (\haut(u)-\haut(v(1)))}
		\text{ and }
	M^{w(1)}_k(z) \coloneqq\frac{1}{C^{w(1)}_k(z)} \frac{1}{S^{w(1)}_k} \sum_{\substack{u \in \T^{w(1)}_k\\ \text{active}}} e^{z (\haut(u)-\haut(w(1)))},
	\end{align*} 
	where the height $\haut$ is measured in $\T_k$. When $S^{v(1)}_k=0$ (resp.\@ $S^{w(1)}_k=0$), we set $M^{v(1)}_k(z)= 0$ (resp.\@ $M^{w(1)}_k(z)=0$). Then one can write for all $k\ge \sigma_1{\vee J}$,
	\begin{equation}\label{eq decomposition martingale}
		M_k(z)= \frac{1}{C_k(z)} \left(  e^{\haut(v(1))} C^{v(1)}_k(z) M^{v(1)}_k(z)+e^{\haut(w(1))} C^{w(1)}_k(z) M^{w(1)}_k(z)\right).
	\end{equation}
	Let  $\tau_{v(1)}(0)=\sigma_1$ and  $\tau_{w(1)}(0)=\sigma_1$, 
	and for $n\geq 1$, we define 
		\[\tau_{v(1)}(n)\coloneqq \inf\enstq{k>\tau_{v(1)}(n-1)}{S^{v(1)}_k \neq S^{v(1)}_{k-1}}
		\text{ and }
		\tau_{w(1)}(n)\coloneqq \inf\enstq{k>\tau_{w(1)}(n-1)}{S^{w(1)}_k \neq S^{w(1)}_{k-1}}
		\]
		to be the $n$-th time that the number of active vertices above $v(1)$ changes (resp. above $w(1)$), where by convention we set $\tau_{v(1)}(n)=\tau_{v(1)}(n-1)$ if the number of active vertices changes less than $n-1$ times.
	
Now let us reason under $\P$, that is, we do not condition on $S$. 
We can check that conditionally on the event $\{\sigma_1<\infty\}$, the {time-changed} sequences 
${\widetilde{S}^{v(1)}=(\widetilde{S}^{v(1)}_{n} : n\ge 0)} = (S^{v(1)}_{\tau_{v(1)}(n)} : n\ge 0)$ and ${\widetilde{S}^{w(1)}=(\widetilde{S}^{w(1)}_{n} : n\ge 0)} =(S^{w(1)}_{\tau_{w(1)}(n)} : n \ge 0)$ are independent random walks (stopped when they reach $0$) whose increments have the same law as the increments of $S$,

 On the event 
		\begin{align*}
		E_1 &\coloneqq \{\sigma_1 <\infty\} \cap \{\forall k> \sigma_1, \ S_k \geq 2\} \cap  \{Z_1 \in (0,1)\}\\
		 &{\subset \{\sigma_1 <\infty\}\cap \{\forall n\geq 0, \ \widetilde{S}^{v(1)}_{n}>0\} \cap \{\forall n\geq 0, \ \widetilde{S}^{w(1)}_{n}>0\}},
		 \end{align*}
the sequences of functions $(z\mapsto C^{v(1)}_{\tau_{v(1)}(n)}(z))_{n\geq 0}$ and $(z\mapsto C^{w(1)}_{\tau_{w(1)}(n)}(z))_{n\geq 0}$ almost surely satisfy the asymptotics of Lemma~\ref{lemme majoration C}. Using \eqref{eq:convergence towards limiting proportion urn}, which ensures that $\tau_{v(1)}(n)\sim n/Z_1$ and $\tau_{w(1)}(n)\sim n/(1-Z_1)$ a.s.~as $n \rightarrow \infty$,
we get that 
 the sequences of functions $(z\mapsto C^{v(1)}_k(z)/C_k(z))_{k\geq \sigma_1 }$ and $(z \mapsto C^{w(1)}_k(z)/C_k(z))_{k\geq \sigma_1 }$ converge a.s.\@ as $k \to \infty$, uniformly in $z\in K$ for any compact $K\subset \mathcal{W}$, to limiting functions that do not vanish on $\intervalleoo{0}{z_\lambda}$. 
	Moreover, by Proposition~\ref{proposition cv unif martingale}, the sequences of functions $(z\mapsto M^{v(1)}_k(z))_{k\geq \sigma_1}$ and $(z\mapsto M^{w(1)}_k(z))_{k\geq \sigma_1}$ also converge almost surely uniformly in $z$ on any compact subset of $\mathcal{W}$ to some random analytic functions $z\mapsto M^{v(1)}_\infty(z)$ and $z \mapsto M^{w(1)}_\infty(z)$.
	Thus, letting $k \to \infty$ in \eqref{eq decomposition martingale}, one can write on the event $E_1$
	\begin{align*} 
	\forall z \in (0,z_\lambda), \qquad 
	M_\infty(z) = A_1(z) M^{v(1)}_\infty(z) + B_1(z)M^{w(1)}_\infty(z),
	\end{align*} 
with some functions $A_1(z),B_1(z)$ that are measurable with respect to $S,S^{v(1)},S^{w(1)},\haut(v(1)),\haut(w(1))$ and so that for all $z \in (0,z_\lambda)$ we have $A_1(z),B_1(z)>0$.
%
Now, by Lemma~\ref{lemme limite martingale non nulle z fixe},  on the event $E_1$,  we have
\[\Ppsq{\forall z\in (0,z_\lambda),\ M^{v(1)}_\infty(z) =0}{\widetilde{S}^{v(1)}}=0\]
so the analytic function $z\mapsto M^{v(1)}_\infty(z)$ is non-zero almost surely so we can enumerate its zeros $(\zeta_n)_{n \geq 1}$  in $(0,z_\lambda)$ in a measurable way.
Now, by construction, on the event $E_1$, conditionally on $S$, $\widetilde{S}^{v(1)}$ and $\widetilde{S}^{w(1)}$, the function $z\mapsto M^{w(1)}_\infty(z)$ is independent of $z\mapsto M^{v(1)}_\infty(z)$. 
Moreover, the function $z\mapsto M^{w(1)}_\infty$ is independent of $S$ and $\widetilde{S}^{v(1)}$ conditionally on $\widetilde{S}^{w(1)}$.
Thus, on the event $E_1$, for any $n\geq 1$,
\begin{align*}\Ppsq{M^{w(1)}_\infty(\zeta_n) =0}{S, \widetilde{S}^{v(1)},\widetilde{S}^{w(1)},(z\mapsto M^{v(1)}_\infty(z))}
	&= \Ppsq{M^{w(1)}_\infty(\zeta_n) =0}{S,\widetilde{S}^{v(1)},\widetilde{S}^{w(1)} } \\
	&=\Ppsq{M^{w(1)}_\infty(\zeta_n) =0}{\widetilde{S}^{w(1)}}\\
	&=0. 
\end{align*}
where the last equality stems from Lemma~\ref{lemme limite martingale non nulle z fixe} again.
	Thus, on the event $E_1$, the function $M_\infty$ has almost surely no zero in $(0,z_\lambda)$.
	
	Similarly to the case $j=1$, for all $j\ge 2$, we set
	\begin{align*} 
	\sigma_j \coloneqq \inf \enstq{k > \sigma_{j-1}}{ \ S_k = 2 \ \text{and}  \ \forall i \in \intervalleentier{0}{k-1}, S_i >0},
	\end{align*} 
	where by convention $\inf \emptyset = \infty$. 
On the event $\{\sigma_j<\infty\}$,
	one can consider the subtrees made of the descendants of the two active vertices 
$v(j)$ and $w(j)$
	of $\T_{\sigma_j}$ and the numbers $S^{v(j)}_k$ and $S^{w(j)}_k$ of active vertices in these subtrees at time $k \ge \sigma_j$.
As in the case $j=1$, conditionally on $S$, on the event $\{\forall k \geq 0, \ S_k >0 \}\cap \{\sigma_j <\infty \}$, the process $(S^{v(j)}_k/(S^{v(j)}_k+S^{w(j)}_k))_{k\geq \sigma_j}$ is a bounded martingale so we let 
	 $Z_j$ be its almost sure limit
	 as $k \to \infty$. 
	By Proposition~\ref{prop urnes} again, on the event
$ \{\sigma_j <\infty \}\cap \{\forall k\geq \sigma_j, \ S_k\geq 2\}$
	we have a.s.
	\begin{equation}\label{eq Z j non trivial}
		\Ppsq{Z_j \in (0,1)}{ S}>0.
	\end{equation}
	Besides, since $S_k \to \infty$ almost surely, 
	\begin{equation}\label{eq somme des probas des sigma j}
		\Pp{\bigcup_{j=1}^\infty\left( \{ \sigma_j <\infty \}\cap \{\forall k\geq \sigma_j, \ S_k\geq 2\}\right)}= \Pp{\forall k \geq 0, \ S_k >0}.
	\end{equation}
	Moreover, by the same reasoning as for $j=1$, on the event 
\[E_j\coloneqq\{\sigma_j<\infty\} \cap \{\forall k\geq \sigma_j, \ S_k\geq 2\} \cap \{Z_j \in (0,1)\} = \{\forall k \geq 0, \ S_k >0 \} \cap \{\sigma_j<\infty\} \cap \{Z_j \in (0,1)\}\]
the function $z\mapsto M_\infty(z)$ has almost surely no zero in $(0,z_\lambda)$,
so the event $\mathcal{A}$ is realized. 
	This amounts to saying that $E_j \subset \mathcal A$. 
	
	We deduce, using properties of conditional expectation and the above remark, that for all $j\ge 1$,
	\begin{align*}
		\E\left[ \Ppsq{Z_j \in (0,1)}{S} \ind{\forall k \geq 0, \ S_k >0 } \ind{\sigma_j <\infty}\right] 
		&= \Pp{E_j}\\
		&= \Pp{\mathcal{A} \cap E_j}\\
		&= \E\left[ \Ppsq{\mathcal{A} \cap\{Z_j \in (0,1)\}}{ S} \ind{\forall k \geq 0, \ S_k >0 }  \ind{\sigma_j <\infty}\right].
	\end{align*}
Combined with the obvious relation $ \Ppsq{\mathcal{A} \cap\{Z_j \in (0,1)\}}{ S} \leq \Ppsq{Z_j \in (0,1)}{S}$, which holds almost surely on the event 
$\{\forall k\geq 0,\ S_k>0\} \cap \{\sigma_j <\infty\}$
for all $j \ge 1$, we get that almost surely
	\begin{align*} 
	\Ppsq{\mathcal{A} \cap \{Z_j \in (0,1)\}}{S}  \ind{\forall k \geq 0, \ S_k >0 } \ind{\sigma_j <\infty}= \Ppsq{Z_j\in (0,1)}{S} \ind{\forall k \geq 0, \ S_k >0 } \ind{\sigma_j <\infty}.
	\end{align*} 
This ensures that on the event $\{\forall k\geq 0,\ S_k>0\}$, for any $j\geq 1$ we have
	\begin{align*} 
	\Ppsq{\mathcal{A}}{S} 
	&\geq \mathbbm{1}_{\{ \sigma_j <\infty \}\cap \{\forall k\geq \sigma_j, \ S_k\geq 2\}} \Ppsq{\mathcal{A} \cap \{Z_j \in (0,1)\}}{S}\\
	&=
	\mathbbm{1}_{\{ \sigma_j <\infty \}\cap \{\forall k\geq \sigma_j, \ S_k\geq 2\}}
	\Ppsq{Z_j \in (0, 1)}{S}.
	\end{align*} 
	Using \eqref{eq Z j non trivial} and \eqref{eq somme des probas des sigma j} we conclude that 
the RHS of the last display is non-zero for at least one value of $j\geq 1$ so that
	$\Ppsq{\mathcal{A}}{S} >0$ almost surely. 
Since we already knew that $\Ppsq{\mathcal{A}}{S} \in \{0,1\}$, this ensures that $\Ppsq{\mathcal{A}}{S} =1$ almost surely, which is what we wanted to prove.
\end{proof}

\subsection{From the Laplace transform to the profile: proofs of Theorem~\ref{thm:profil} and Proposition~\ref{prop: profil faible}}	
\label{ssec:endprofile}
{In this subsection we finally prove Theorem~\ref{thm:profil}, the main result of Section \ref{sec:profile}, and then explain how it implies Proposition~\ref{prop: profil faible}.
}
The type of objects and arguments that we use in this section is very close to the theory of mod-$\phi$ convergence, exposed for example in the book \cite{FMN16}.
We shall borrow some notation from the latter reference.
Specifically:
\begin{enumerate}
\item {let} $\phi$ is the Poisson distribution with parameter $\gamma$;
\item {define} $\eta(z)\coloneqq\gamma (e^z -1)$ so that $\exp(\eta(z))=\int e^{zx} \phi(\dd x)$;
\item {denote} $F$ the Legendre transform of $\eta$ i.e.
\begin{align*}
	\forall \theta \in \R, \qquad \qquad
	F(\theta)= \sup_{h\in \mathbb R} (h\theta - \eta(h)).
\end{align*}
\end{enumerate}
Following  the convention of  \cite{FMN16} (see Section~2.2), for a fixed $\theta\in \mathbb R$ we denote by $h=h(\theta)$  the unique value that maximizes $h \mapsto h\theta - \eta(h)$; it is defined by the equation $\eta'(h)=\theta$.
This implies the identities
\begin{align*}
	F(\theta)=\theta h-\eta(h), \qquad F'(\theta)=h, \qquad F''(\theta)=h'(\theta)=\frac{1}{\eta''(h)}.
\end{align*}
In our case we have $\eta(z)=\gamma (e^z -1)$ and $h$ and $\theta$ are so that $\gamma e^h=\theta$, i.e.\ $h=\log(\theta/\gamma)$. 
Hence 
\begin{align*}
	F(\theta) = \log (\theta/\gamma)\theta - \gamma(\theta/\gamma -1) \quad \text{and} \quad F'(\theta)=\log(\theta/\gamma) \quad \text{and} \quad F''(\theta)=\frac{1}{\theta}.
\end{align*}
In particular, note that $F$ and the function $f_\lambda$ defined in \eqref{eq:param} are related by the identity
\begin{equation}\label{eq lien F f lambda}
\forall x >0, \qquad \qquad -F(\gamma e^x) = f_\lambda(x)-1.
\end{equation}
\begin{proof}[Proof of Theorem~\ref{thm:profil}]
{Let $I\subset (1,\infty)$ be a compact interval. All the $O_\P(1), o_\P(1),O(1),o(1)$ in this proof hold weakly uniformly in $\lambda \in I$.} Recall from Section~\ref{sous-section preliminaire couplage} our coupled construction and in particular the fact that for $n \geq 1$ and $k \geq 0$ we have $\T^n_k= \T_{k}(\XXb^n)$. 
We work on the event 
		\begin{align*}
			\{\forall i \in \intervalleentier{0}{\lfloor nt \rfloor}, \ S^n_i > 0\} \cap \{\forall i \geq 0, \ S_i > 0\} \cap \{J^n{\leq} \lfloor nt \rfloor\} \cap \{J{\leq} \lfloor nt \rfloor\},
		\end{align*}
	so that the quantities $M^n_{\lfloor nt \rfloor}(z)$ and $M_{\lfloor nt \rfloor}(z)$ are well-defined for all $z\in  \mathscr{E}$. 
	There is no loss of generality in doing this, since
	\begin{align*}
		\mathbbm{1}_{\{\forall i \in \intervalleentier{0}{{\lfloor nt \rfloor}}, \ S^n_i > 0\} \cap \{\forall i \geq 0, \ S_i > 0\} \cap \{J^n{\leq} {\lfloor nt \rfloor}\} \cap \{J{\leq} {\lfloor nt \rfloor}\}} = \ind{\forall i \in \intervalleentier{0}{{\lfloor nt \rfloor}}, \ S^n_i > 0} + o_\P(1).
	\end{align*}
by Lemma~\ref{lem:JnJ}. 

	Recall that
\begin{align*}
	\bL_{\lfloor nt \rfloor}^n(k)=\frac{\#\{\text{active vertices at height $k$ at time ${\lfloor nt \rfloor}$}\}}{S_{\lfloor nt \rfloor}^n}
\end{align*}
denotes the normalized active profile of the tree $\T_{\lfloor nt \rfloor}^n$.
Now we keep the notation introduced in Section~\ref{sous-section martingales} and write {for all $h,u \in \R$,}
\begin{align*}
	\mathcal{L}(h+iu, \T_{\lfloor nt \rfloor}^n) = \sum_{k=0}^{\infty} \bL_{\lfloor nt \rfloor}^n(k) \cdot  e^{k(h+iu)}.
\end{align*}
Since $\bL_{\lfloor nt \rfloor}^n(k) \cdot  e^{kh}$ is the $k$-th Fourier coefficient of the expansion of $\mathcal{L}(h+iu, \T_{\lfloor nt \rfloor}^n)$ we have 
\begin{align*}
	\bL_{\lfloor nt \rfloor}^n(k)  = \frac{1}{2\pi} \int_{-\pi}^{\pi} \mathcal{L}(h+iu, \T_{\lfloor nt \rfloor}^n) e^{-k(h+iu)} \dd u.
\end{align*}
Now, following Theorem~3.2.2 in \cite{FMN16}, let {$\theta>0$ such that $\theta\in (\eta'({0}), \eta'(z_\lambda)) = (\gamma, \gamma e^{z_\lambda})$ for all $\lambda \in I$} and $h$ defined by the equation $\eta'(h)=\gamma e^h=\theta$. 
Assume that $\theta \log n \in \mathbb{N}$. 
Then 
\begin{align*}
	\bL_{\lfloor nt \rfloor}^n(\theta \log n)  &= \frac{1}{2\pi} \int_{-\pi}^{\pi} \mathcal{L}(h+iu, \T_{\lfloor nt \rfloor}^n) e^{-(\theta \log n) (h+iu)} \dd u\\
	&= \frac{1}{2\pi} \int_{-\pi}^{\pi} M_{\lfloor nt \rfloor}^n(h+iu) C_{{\lfloor nt \rfloor}}^n(h+iu)  e^{-(\theta \log n) (h+iu)} \dd u\\
	&\! \! \! \! \! \underset{\text{Lem.~\ref{lem:asymptotics Cnkz}}}{=} \frac{1}{2\pi} \int_{-\pi}^{\pi}  M_{\lfloor nt \rfloor}^n(h+iu) e^{-(\theta \log n)\cdot (h+iu) + \eta(h+iu) \log n + O_\P(1)} \dd u,
\end{align*}
{where in the last equality we use the fact that $0<\mathrm{Re}(h+iu)=h<z_\lambda$ for all $\lambda \in I$ so that $h+iu \in \mathcal{E}(\lambda)\cap \{z \in \mathbb{C}, \ \mathrm{Re}(z)>0\}$ for all $\lambda \in I$ and $u \in [-\pi,\pi]$.}

Now focusing only on the term in the exponential, and using that $\theta h=F(\theta)+ \eta(h)$ and $\theta=\eta'(h)$ we get
\begin{align*}
-(\theta \log n)\cdot  (h+iu) + \eta(h+iu) \log n&= \log n\cdot (- \theta h -i\theta u + \eta(h+iu))\\
&= \log n\cdot (- (F(\theta) + \eta(h)) -i \eta'(h)u + \eta(h+iu))\\
&= -\log n\cdot F(\theta)+ \log n \cdot (\eta(h+iu)-  \eta(h) -i \eta'(h)u ).
\end{align*}
Putting things together we get
\begin{align}\label{eq:normalized profile expressed using integral}
	\bL_{\lfloor nt \rfloor}^n(\theta \log n)  &=  \frac{e^{-F(\theta) \log n + O_\P(1)} }{2\pi} \int_{-\pi}^{\pi}  M_{\lfloor nt \rfloor}^n(h+iu) e^{ \log n \cdot (\eta(h+iu)-  \eta(h) -i \eta'(h)u )} \dd u.
\end{align}
From there, we are going to split the integral $\int_{-\pi}^{\pi}$ in the last display into a main term $\int_{-\delta_n}^{\delta_n}$ and some error terms $\int_{\delta_n}^{u_0} + \int_{-u_0}^{-\delta_n}$ and $\int_{u_0}^{\pi} + \int_{-\pi}^{-u_0}$
for some $\delta_n \downarrow 0$ and $u_0\in \intervalleoo{0}{\pi}$ appropriately chosen.
\paragraph*{First part: the main term.}
We want to compute the asymptotics of the term
\begin{align}
	\int_{-\delta_n}^{\delta_n}
	&M^n_{\lfloor nt \rfloor}(h+iu)
	\exp \left(
	\log n \cdot (\eta(h+iu)-\eta(h)-i\eta'(h) u )
	\right)\mathrm{d}u \notag
\end{align}
for some appropriately chosen sequence $(\delta_n)$. 
For that, the first step
is to re-write the integral of the last display as 
\begin{align}
	&M_{\infty} (h) \int_{-\delta_n}^{\delta_n} \exp(\log n \cdot (\eta(h+iu)-\eta(h) - i \eta'(h) u)) \mathrm{d}u \label{eq:main part of main term}\\
	&\qquad + 
	\int_{-\delta_n}^{\delta_n} \left( M^n_{\lfloor nt \rfloor}(h+iu) - M_\infty(h)\right)
	\exp \left(\log n \cdot  (\eta(h+iu)-\eta(h)-i\eta'(h) u )\right)\mathrm{d}u 
	\label{eq:integral difference}
\end{align}
and handle the two terms \eqref{eq:main part of main term} and \eqref{eq:integral difference} separately.
We start with \eqref{eq:integral difference}.
First, let {$p\in \intervalleof{1}{2}$} be so that Theorem~\ref{th: convergence vers M infini} holds for $z=h$. We consider the $L^p$ norm of the random variable in \eqref{eq:integral difference}.
We first bound the modulus of the integral by the integral of the modulus of the integrand 
and then use Jensen's inequality (in the form of $\int f g \leq (\int g)^{\frac{p-1}{p}}\cdot \left(\int f^p g \right)^{\frac{1}{p}}$, valid for non-negative functions $f$ and $g$, with $g$ integrable),
\begin{align}
&\E_n\left[\left\vert\int_{-\delta_n}^{\delta_n} \left( M^n_{\lfloor nt \rfloor}(h+iu) - M_\infty(h)\right)
\exp \left(\log n \cdot (\eta(h+iu)-\eta(h)-i\eta'(h) u )\right)\mathrm{d}u \right\vert^p\right]\notag\\
&\le \E_n\left[\left(\int_{-\delta_n}^{\delta_n} \left| M^n_{\lfloor nt \rfloor}(h+iu) - M_\infty(h)\right|
\exp \left(\log n \cdot \Re(\eta(h+iu)-\eta(h)-i\eta'(h) u )\right)\mathrm{d}u \right)^p\right]\notag
\\
&\le \E_n \left[ \left( \int_{-\delta_n}^{\delta_n} \exp \left(\log n \cdot  \Re(\eta(h+iu)-\eta(h)-i\eta'(h) u )\right)\mathrm{d}u \right)^{p-1} \right. \notag\\
&\times \left.\left(\int_{-\delta_n}^{\delta_n} \left\vert M^n_{\lfloor nt \rfloor}(h+iu) - M_\infty(h)\right\vert^p
\exp \left(\log n \cdot \Re(\eta(h+iu)-\eta(h)-i\eta'(h) u )\right)\mathrm{d}u\right)\right]\notag\\
&\leq \sup_{u\in\intervalleff{-\delta_n}{\delta_n}} \E_n\left[  |M^n_{{\lfloor nt \rfloor}}(h+iu)-M_\infty(h) \vert^p \right] \cdot \left( \int_{-\delta_n}^{\delta_n} \exp \left(\log n \cdot \Re(\eta(h+iu)-\eta(h)-i\eta'(h) u )\right)\mathrm{d}u \right)^p \label{eq majoration ecart M h plus iu moins M h}
\end{align}
where for the last inequality, we first used Fubini and then upper-bounded the integrand uniformly. 
Note that 
\begin{align*}
	\E_n\left[  |M^n_{{\lfloor nt \rfloor}}(h+iu)-M_\infty(h) \vert^p \right] = o_\P(1)
\end{align*}

thanks to Theorem~\ref{th: convergence vers M infini}.

We should now understand the (deteministic) integral that appears in \eqref{eq majoration ecart M h plus iu moins M h}, as well as the very similar one that appears in \eqref{eq:main part of main term}.
We have $\eta(h + iu) -\eta(h) -i\eta'(h) u =-(u^2/2) \eta''(h) + O(u^3)$ and hence also $\Re(\eta(h + iu) -\eta(h) -i\eta'(h) u)= -(u^2/2) \eta''(h) + O(u^3)$ as $u\to 0$.
Therefore, by taking $\delta_n$ so that $\delta_n^{3} \cdot \log n \to 0$, we get
\begin{align}
\int_{-\delta_n}^{\delta_n} \exp\left(
	\log n \cdot (\eta(h + iu) -\eta(h) -i\eta'(h) u)
	\right) \mathrm{d}u
	&= 
	\int_{-\delta_n}^{\delta_n} \exp\left(
	 -\frac{u^2}{2} \cdot  \eta''(h) \cdot \log n + o(1)
	\right) \mathrm{d}u \notag \\
	&= (1+o(1))\cdot 
	\int_{-\delta_n}^{\delta_n} \exp\left(
	-\frac{u^2}{2} \cdot  \eta''(h) \cdot \log n
	\right) \mathrm{d}u,\label{eq DL eta}
\end{align}
and similarly 
\begin{equation}\label{eq DL partie reelle de eta}
\int_{-\delta_n}^{\delta_n} \exp \left(\log n \cdot \Re(\eta(h+iu)-\eta(h)-i\eta'(h) u )\right)\mathrm{d}u = (1+o(1))\cdot 
\int_{-\delta_n}^{\delta_n} \exp\left(
-\frac{u^2}{2} \cdot  \eta''(h) \cdot \log n
\right) \mathrm{d}u.
\end{equation}
Besides, using a change of variable $v=u\cdot \sqrt{\eta''(h)\cdot \log n}$ we obtain
\begin{align}
	\int_{-\delta_n}^{\delta_n} \exp\left(
	-\frac{u^2}{2} \cdot \eta''(h)\cdot \log n
	\right) \mathrm{d}u
	&= 
	\int_{-\delta_n \sqrt{\eta''(h) \log n}}^{\delta_n \sqrt{\eta''(h) \log n}} \exp\left(
	-\frac{v^2}{2}
	\right) \frac{\dd v}{\sqrt{\eta''(h) \log n}} \notag \\
	&=\frac{1}{\sqrt{\eta''(h) \log n}}
	\left(\int_{-\infty}^\infty \exp\left(-\frac{v^2}{2}\right)\dd v + o(1)\right)\notag\\
	&= (1+o(1))\cdot \sqrt{\frac{2\pi}{\eta''(h)\log n}} ,\label{eq asymptotique integrale}
\end{align}
where in the second equality we assume that we take $(\delta_n)$ so that $\delta_n \sqrt{\log n} \to \infty$. 
For the rest of the proof, we fix $\delta_n= (\log n)^{-5/12}$ so that the results above hold. 
Now putting everything together, we get
\begin{align}\label{eq: equivalent integrale bulk}
	&\int_{-\delta_n}^{\delta_n} M^n_{\lfloor nt \rfloor} (h+iu) \exp(\log n \cdot (\eta(h+iu)-\eta(h) - i \eta'(h) u)) \mathrm{d}u \notag\\
	&= M_{\infty} (h) \int_{-\delta_n}^{\delta_n} \exp(\log n \cdot (\eta(h+iu)-\eta(h) - i \eta'(h) u)) \mathrm{d}u \notag\\
	&\qquad+ 	\int_{-\delta_n}^{\delta_n} \left( M^n_{\lfloor nt \rfloor}(h+iu) - M_\infty(h)\right)
	\exp \left(\log n \cdot  (\eta(h+iu)-\eta(h)-i\eta'(h) u )\right)\mathrm{d}u \notag\\
	&=  M_\infty(h) \cdot (1+o(1))\cdot \sqrt{\frac{2\pi}{\eta''(h)\log n}} + o_\P\left(\frac{1}{\sqrt{\log n}}\right),
\end{align}
{where the first term in the last line comes from \eqref{eq DL eta} and \eqref{eq asymptotique integrale} while the second term comes from \eqref{eq majoration ecart M h plus iu moins M h}, \eqref{eq DL partie reelle de eta} and \eqref{eq asymptotique integrale},}
and $M_\infty(h)>0$ by Proposition~\ref{prop: M infini ne s'annule pas}.

\paragraph*{Second part: the error terms.}
Now we need to show that the term
\begin{align*}
	\int_{\delta_n}^\pi
	M^n_{\lfloor nt \rfloor}(h+iu)
	\exp\left(
	\log n \cdot
	( \eta(h+iu)-\eta(h) - i \eta'(h) u)
	\right)
	\mathrm{d}u
\end{align*}
and the symmetrical integral are negligible compared to the main term \eqref{eq: equivalent integrale bulk}.
We only deal with this first integral since the other term is handled similarly. We reason in expectation (conditional on $S^n$) and start by writing
\begin{align}\label{eq:majoration integrale delta pi}
		&\E_n \left[  \left \vert
		\int_{\delta_n}^\pi
		M^n_{\lfloor nt \rfloor}(h+iu)
		\exp\left(
		\log n \cdot
		( \eta(h+iu)-\eta(h) - i \eta'(h) u)
		\right)
		\mathrm{d}u
		\right\vert  \right]\notag\\
		&\le 
		\E_n \left[{\int_{\delta_n}^\pi
			\left\vert M^n_{\lfloor nt \rfloor}(h+iu)\right\vert
			\exp\left(
			\log n \cdot
			\Re( \eta(h+iu)-\eta(h) - i \eta'(h) u)
			\right)
			\mathrm{d}u}\right]\notag\\
		&\le {\int_{\delta_n}^\pi
			\E_n \left[\left\vert M^n_{\lfloor nt \rfloor}(h+iu)\right\vert \right] \cdot 
			\exp\left(
			\log n \cdot
			\Re( \eta(h+iu)-\eta(h) - i \eta'(h) u)
			\right)
			\mathrm{d}u}.
\end{align}
From Proposition~\ref{proposition majoration increments Mn}, it holds that for any compact $K \subset \{z \in \mathbb{C}, \ \Re z>0\}$ {such that $K \subset \mathcal{E}(\lambda)$ for all $\lambda \in I$, for all $p \in (1,2]$}, uniformly in $z\in K$ and $k\leq nt$, we have 
\begin{align*}
	\Ecp{n}{\lvert M^n_{k+1 }(z)- M^n_k(z)\rvert^p} \le \exp \left(
	\left( -p+ \gamma\left( e^{p \Re z} -1 -p (\Re(e^z) -1)\right)
	\right) \log k + O_\P(1)
	\right)
\end{align*}
so that, {taking $p>1$ close enough to $1$ so that} the quantity $1-p+ \gamma ( e^{p \Re z} -1 -p (\Re(e^z) -1))$ is strictly negative on $K$ {for all $\lambda \in I$}, using Lemma~A.2 of \cite{Sen21}, we have 
\begin{align*}
	\E_n\left[\lvert M^n_k(z)\rvert^p\right] {\leq \Ecp{n}{\lvert M^n_1(z)\rvert^p}+ 2^p \cdot \sum_{i=1}^{k-1} \Ecp{n}{\lvert M^n_{i+1}(z)- M^n_{i}(z)\rvert^p}} = O_\P(1),
\end{align*}
uniformly on $z\in K$. The inequality $h<z_\lambda$ {for all $\lambda \in I$} ensures that there exists $p\in \intervalleof{1}{2}$ such that {for all $\lambda \in I$, we have} $1-p+ \gamma\left(e^{ph} -1 -p (e^h -1)\right) <0$. 
By continuity, this ensures that for small enough $u\geq 0$, say smaller or equal than some $u_0 = u_0(h)>0$, we have, locally uniformly in $h<z-\lambda$ and $u\in\intervalleff{0}{u_0(h)}$,
\begin{align}\label{eq:expectation modulus martingale O(1)}
	\E_n\left[\lvert M^n_{\lfloor nt \rfloor}(h+iu)\rvert^p\right] = O_\P(1).
\end{align}
In general, without assuming anything on the sign of $1-p+ \gamma( e^{p \Re z} -1 -p (\Re(e^z) -1))$, we can also get the following for any {compact set $K\subset\{z \in \mathbb{C}, \ \mathrm{Re}(z)>0\}$ such that $K \subset \mathcal{E}(\lambda)$ for all $\lambda \in I$, uniformly in $z \in K$,}
\begin{align}\label{eq:expectation modulus martingale power}
	\E_n\left[\left\vert M^n_{\lfloor nt \rfloor}(z)\right\vert^p\right] 
	&\leq  {(1+\log \lfloor nt \rfloor)} \cdot e^{(\log n) \cdot 
	\left(1-p+ \gamma\left( e^{p \Re z} -1 -p (\Re(e^z) -1) \right)
	\right)\vee 0
	+O_\P(1)
	}.
\end{align}
This is done using the fact that for any $\alpha\in \R$ and $n \in \N$, we have the inequality
\begin{align}\label{eq:upper bound sum k alpha}
	\sum_{k=1}^{n-1} k^\alpha \leq 1+ n^{0\vee (\alpha+1)}\cdot  \log n \leq (1+\log n) \cdot  n^{0\vee (\alpha+1)}.
\end{align}
Indeed, for any $\alpha\in \R$ {and $n \in \N$,}
\begin{align*}
	\sum_{k=2}^{n-1} k^\alpha \leq \int_{1}^{n}x^\alpha \dd{x} \leq \left\lbrace \begin{aligned}
	& \frac{n^{\alpha +1} -1}{\alpha+1} = (\log n) \cdot \frac{\exp((\alpha +1) \log n) -1}{(\alpha+1)\log n} \leq (\log n )\cdot n^{\alpha +1} \quad &\text{if } \alpha > -1, \\
	& \log n \quad &\text{if } \alpha \leq -1,
	\end{aligned}
\right. 
\end{align*}
where we used the inequality $\frac{e^x-1}{x} \leq e^x$, valid for any $x>0$, for  $x=(\alpha+1) \log n$.
In what follows, we then split the integral $\int_{\delta_n}^\pi$ and deal with the term $\int_{\delta_n}^{u_0}$ using \eqref{eq:expectation modulus martingale O(1)} and then the term $\int_{u_0}^\pi$ using \eqref{eq:expectation modulus martingale power}. 
\paragraph*{First error term.}
Since $\Re( \eta(h+iu)-\eta(h) - i \eta'(h) u)=\gamma e^h (\cos u - 1) \leq -\gamma e^h \frac{u^2}{8}$ for $u\in \intervalleff{-\pi}{\pi}$ we have 
\begin{align*}
	&\int_{\delta_n}^{u_0} \E_n \left[\left\vert M^n_{\lfloor nt \rfloor}(h+iu)\right\vert\right]
	\exp\left(
	\log n \cdot
	\Re( \eta(h+iu)-\eta(h) - i \eta'(h) u)
	\right)
	\dd u\\
	&\le O_\P(1)\int_{\delta_n}^{u_0} \exp\left(\log n \cdot \gamma e^h (\cos u -1) \right)\dd u  \qquad \qquad \qquad \text{{by \eqref{eq:expectation modulus martingale O(1)}}}\\
	&\le O_\P(1)\int_{\delta_n}^{u_0}
	e^{-\log n \cdot \gamma e^h u^2/8}
	\dd u \\
	&=O_\P(1) \int_{\delta_n \sqrt{  \log n \cdot \gamma e^h/4}}^{u_0\sqrt{  \log n \cdot \gamma e^h/4}} e^{-v^2/2} \frac{\dd v}{\sqrt{  \log n \cdot \gamma e^h/4}} \\
	&=o_\P(1/\sqrt{\log n}),
\end{align*}
where the last line comes from the fact that $\delta_n \sqrt{\log n} \to \infty$, so this term is of smaller order than the main term (and similarly for the symmetric term).
\paragraph*{Second error term.}
Now we take care of the last term $\int_{u_0}^{\pi}$ using \eqref{eq:expectation modulus martingale power}. 
Using Jensen's inequality, and writing $z=h+iu$ we get, {uniformly in $u \in [u_0,\pi]$,}
\begin{align*}
	\E_n&\left[\left\vert M^n_{\lfloor nt \rfloor}(h+iu)\right\vert\right] \leq \E_n\left[\left\vert M^n_{\lfloor nt \rfloor}(h+iu)\right\vert^p\right]^{\frac{1}{p}}\\ 
	&\qquad \leq {(1+\log \lfloor nt \rfloor)}^{1/p}\cdot \exp\left(\log n \cdot 
	\left(
	\left(\frac{1-p}{p} + \frac{\gamma}{p}(e^{ph}-1)- \gamma(e^h\cos u - 1) \right)\vee 0 
	\right)
	+O_\P(1)
	\right).
\end{align*}
Recall that $\Re( \eta(h+iu)-\eta(h) - i \eta'(h) u)=\gamma e^h \cos u - \gamma e^h$, so, plugging this into the integrand of \eqref{eq:majoration integrale delta pi}, we get {uniformly in $u \in [u_0,\pi]$,}
\begin{align*}
	\E_n&\left[\left\vert M^n_{\lfloor nt \rfloor}(h+iu)\right\vert\right] 
	\exp\left(
	\log n \cdot
	\Re( \eta(h+iu)-\eta(h) - i \eta'(h) u)
	\right)\\
	\leq &{(1+\log \lfloor nt \rfloor)}^{1/p} \\
	&\cdot \exp\left(\log n \cdot
	\left(
	\left(\frac{1-p}{p} + \frac{\gamma}{p}(e^{ph}-1)- \gamma e^h \cos u + \gamma \right)\vee 0
	\right) + O_\P(1)+
	\log n \cdot(\gamma e^h\cos u - \gamma e^h) \right)\\
	 \leq &{(1+\log \lfloor nt \rfloor)}^{1/p} \exp\left(\log n \cdot
	\left(
	\left(\frac{1-p}{p} + \frac{\gamma}{p}(e^{ph}-1) + \gamma(1-e^h)\right)\vee \gamma e^h(\cos u - 1) 
	\right)
	+ O_\P(1)\right).
\end{align*}
Note that from our choice of $p$ and $h$ the expression $\frac{1-p}{p} + \frac{\gamma}{p}(e^{ph}-1) + \gamma(1-e^h) $ is negative and $u \mapsto \gamma e^h(\cos u - 1)$ is negative {for all $\lambda \in I$} and decreasing on $[u_0 ,\pi]$. Hence the expression in the last display is bounded above, uniformly in $u\in \intervalleff{u_0}{\pi}$ by some term
\begin{align*}
	n^{\beta +o_\P(1)} = o_\P\left(\frac{1}{\sqrt{\log n}}\right),
\end{align*}
where $\beta= \left(\frac{1-p}{p} + \frac{\gamma}{p}(e^{ph}-1) + \gamma(1-e^h)\right)\vee \gamma e^h(\cos u_0 - 1)<0 $.

\paragraph*{Conclusion.}
Thus, starting from \eqref{eq:normalized profile expressed using integral}, combining our results \eqref{eq: equivalent integrale bulk} for $\int_{-\delta_n}^{\delta_n}$ and the controls on the error terms $\int_{\delta_n}^{u_0}$ and $\int_{u_0}^\pi$ (and their symmetric {counterparts}), we get 
\begin{equation*}
\mathbb{L}^n_{\lfloor nt \rfloor}(\theta \log n ) = e^{-F(\theta) \log n - \frac{1}{2}\log \log n + O_\P(1)}.
\end{equation*}
This concludes the proof thanks to \eqref{eq lien F f lambda}.
\end{proof}

Now, let us prove Proposition~\ref{prop: profil faible} using Theorem~\ref{thm:profil}.
	We rely here on the coupled construction of the process of Section~\ref{sous-section preliminaire couplage}, so that the number of infected individuals in the infection process $(I^n_k)_{k\geq 0}$ is given here by $(S^n_{k\wedge \tau'_n})_{k\geq 0}$ where $\tau'_n=\inf\enstq{k\geq 0}{S^n_k=0}$, so that on the event $\{\tau'_n \geq k\}$ we have the equality  $I^n_k=S^n_k$.

\begin{proof}[Proof of Proposition~\ref{prop: profil faible}]
In the statement of the proposition, we have $\lambda_n \sim {\lambda}/{n}$ with a fixed $\lambda>1$. 
	We will use the previous results, which assume that $\lambda_n = {\lambda}/{n}$ but hold uniformly in $\lambda$ contained in a compact interval, by applying them for $\lambda =n \lambda_n$.  
	Note that for any {compact interval} $I$ containing $\lambda$ in its interior, we have $n\lambda_n \in I$, provided that $n$ is large enough.

	Let $t \in (0,t_\lambda)$, $x\in \intervalleoo{0}{z_\lambda}$ and $y \in (x,\infty]$. 
	We define $\theta_1,\theta_2$ as $\theta_1=\gamma e^x$ and  $\theta_2=\gamma e^y$. {Recall that we write $\mathbb{E}_{n}$ for $\Ecsq{\, \cdot}{S^{n}}=\Ecsq{\, \cdot}{\XXb^{n}}$.} 
Writing $h=h(\theta_{1})$,
  {on the event $\{\tau'_{n} \geq  \lfloor n t\rfloor \} \cap \{J_n \leq \lfloor nt \rfloor \}$ we have}
	\begin{align*}
		\E_n \left[\mathbb{L}^n_{\lfloor nt \rfloor}([\theta_1 \log n,\theta_2 \log n])\right]&\le
		\E_n\left[ \frac{1}{S^n_{\lfloor nt \rfloor}} \sum_{v \text{ active}} e^{h(\haut(v) - \theta_1  \log n)}\right]\\
		&\underset{\eqref{eq:productform}}{=} e^{-h\theta_1 \log n} \cdot {\Ecp{n}{\mathcal{L}(z,\T^n_{J^n})}} \cdot  C^n_{\lfloor nt \rfloor}(h)\\
		&\underset{Lem.~\ref{lem:asymptotics Cnkz}}{=}e^{(-h\theta_1 + \eta(h) ) \log n+ O_\P(1)} \\
		&= e^{(-F(\theta_1)-\eta(h)  + \eta(h))\log n + O_\P(1)} \\
		&= e^{-F(\theta_1) \log n  + O_\P(1)},
		\end{align*}
	so using Markov's inequality with the probability measure $\P_n$, we get that
	\begin{equation}
	\label{eq:ineg}\frac{\log \mathbb{L}^n_{\lfloor nt \rfloor} ([\theta_1 \log n,\theta_2 \log n])}{ \log n} {\leq} -F(\theta_1)+o_\P(1)
	\end{equation} as $n\rightarrow \infty$
	  {on the event $\{\tau'_{n} \geq  \lfloor n t\rfloor \} \cap \{J_n \leq \lfloor nt \rfloor \}$.
Because of Lemma~\ref{lem:JnJ}, we have $J_n\ind{\tau'_n \geq  \lfloor n t\rfloor}=O_\P(1)$ and so $\Pp{\{\tau'_{n} \geq  \lfloor n t\rfloor \}\cap \{J_n > \lfloor nt \rfloor\}} = o(1)$ as $n\rightarrow \infty$,  and so the inequality \eqref{eq:ineg} holds on the event $\{\tau'_{n} \geq  \lfloor n t\rfloor \}$.}
{A matching} lower bound follows from Theorem~\ref{thm:profil}.

Now using the fact that, by \eqref{eq:limflu}, on the event $\{\tau'_{n} \geq  \lfloor n t\rfloor\}$ we have $\log I^{n}_{\lfloor nt \rfloor}= \log n + O_{\P}(1)$, we get  
\begin{align*}
		\frac{\log \mathbb{A}^n_{\lfloor nt \rfloor} ([\theta_1 \log n,\theta_2 \log n])}{ \log n} 
		&= \frac{\log \mathbb{L}^n_{\lfloor nt \rfloor} ([\theta_1 \log n,\theta_2 \log n]) + \log I^{n}_{\lfloor nt \rfloor} }{ \log n}\\
		&= -F(\theta_1)+1 +o_\P(1)\\
		&= -F(\gamma e^ x )+1 +o_\P(1)\\
		&= f_\lambda (x) + o_\P(1),
\end{align*}
on the event $\{\tau'_{n} \geq  \lfloor n t\rfloor\}$, where the last line follows from \eqref{eq lien F f lambda}. 
This completes the proof.
\end{proof}

\section{Time dependent P\'olya urns with removals and application to frozen recursive trees}
\label{sec:polya}
Let $(\X_k)_{k\ge 1}$ be a (deterministic) sequence in $\mathbb{Z}_{\ge -1}$. Let $s_0\ge 2$. 
For all $k\ge 0$, let 
$
s_k=s_0+\X_1+\ldots +\X_k
$. We assume that for all $k\ge 0$, we have $s_k \ge {1}$.
We start with an urn with $b_0\ge1$ blue balls and $r_0\ge1$ red balls such that $r_0+b_0 = s_0$. 
At each step $k\ge1$, we draw a ball at random in the urn. 
If $\X_k\ge 0$, then we put the ball back in the urn together with $\X_k$ new balls of the same color. 
If $\X_k=-1$, then we remove the ball. 
In other words, if $(R_k)_{k\ge 0}$ denotes the number of red balls in the urn, the sequence $(R_k)_{k\ge 0}$ is a time-{inhomogeneous} Markov chain which evolves as follows: for all $k\ge 0$, conditionally on $R_0,\ldots, R_k$, we have
\begin{equation*} 
R_{k+1}=  R_k+\X_{k+1} \cdot B_{k+1} \quad \text{where} \quad B_{k+1} = 
\left\lbrace 
\begin{aligned}
&1 &\quad &\text{with probability} &  \frac{R_k}{s_k}&, \\
&0 &\quad &\text{with probability }&  1-\frac{R_k}{s_k}&.
\end{aligned}
\right.
\end{equation*} 
	We start with a simple convergence result.
\begin{lemma}\label{lem:convergence proportion urn}
We have the almost sure convergences 
\begin{align*}
	\frac{R_k}{s_k} \underset{k\rightarrow \infty}{\rightarrow} Z \qquad \text{and} \qquad \frac{1}{k}\sum_{i=1}^kB_i \underset{k\rightarrow \infty}{\rightarrow} Z,
\end{align*} 
for some random variable $Z$ with values in $\intervalleff{0}{1}$. 
\end{lemma}
\begin{proof}
For any $k\ge 0$, we compute
	\begin{align*} 
		\Ecsq{\frac{R_{k+1}}{s_{k+1}}}{R_0,\ldots,R_k} =
		\frac{R_k +\X_k}{s_{k+1}} \cdot \frac{R_k}{s_k} + \frac{R_k}{s_{k+1}}\cdot \left(
		1 - \frac{R_k}{s_k}
		\right)
		= \frac{R_k}{s_k}.
	\end{align*}
This ensures that $(R_k/s_k)_{k\ge 0}$ is a martingale. By construction it only takes values in $\intervalleff{0}{1}$ so it converges a.s.\@ towards some random variable $Z$ with values in $\intervalleff{0}{1}$.

We now prove the second convergence.
On the event $\{\sum_{k\ge 0} R_k/s_k = \infty\}$, it follows from the third Borel-Cantelli lemma (see \cite[Theorem~4.5.5]{Dur19}) together with the fact that, by the previous convergence, we have 
\begin{align*}
	\frac{1}{k}\sum_{i=1}^k\Ecsq{B_i}{R_1,\dots ,R_{i-1}} = \frac{1}{k}\sum_{i=1}^k \frac{R_{i-1}}{s_{i-1}} \underset{k\rightarrow \infty}{\rightarrow} Z
\end{align*} 
almost surely by Ces\`aro convergence.
On the event $\{\sum_{k\ge 0} R_k/s_k <\infty\}$, the Borel-Cantelli lemma \cite[Theorem~4.3.4]{Dur19} ensures that $\sum_{k\ge 1} B_k <\infty$ so that the convergence also holds in this case.
\end{proof}

The following proposition is a generalization of Theorem~2 of \cite{Pem90}:
\begin{proposition}\label{prop urnes}
	The almost sure limit ${Z}$ of $(R_k/s_k)_{k\ge 0}$ is a Bernoulli random variable if and only if
	\begin{align*}
		\prod_{j=1}^\infty \left(1- \frac{\X_j^2}{s_j^2} \right)=0.
	\end{align*}
Equivalently, we have $\P(Z\in \intervalleoo{0}{1})>0$ if and only if
	\begin{align*} 
	\forall k\geq 0, \ s_k\geq 2 \qquad \text{and} \qquad \sum_{k\ge 0} \left(\frac{\X_k}{s_k}\right)^2 < \infty.
	\end{align*} 
\end{proposition}
\begin{proof}
	The proof follows from the same ideas as in \cite{Pem90}. 
	Let {$Z$} be the almost sure limit of  the bounded martingale $(R_k/s_k)_{k\ge 0}$. 
	The idea is to observe that since $0 \leq {Z} \leq 1$ and $\mathbb{E}[{Z}]=r_{0}/s_{0}$, we have $\mathbb{E}[{Z}^{2}] \leq r_{0}/s_{0}$ with equality if and only if ${Z}$ is a Bernoulli random variable.  
	We compute
	\begin{align*}
		\Ecsq{\left(\frac{R_{k+1}}{s_{k+1}} \right)^2 }{R_0,\ldots ,R_k}
		&=
		\left(\frac{R_k +\X_{k+1}}{s_{k+1}}\right)^2\frac{R_k}{s_k} + \left(\frac{R_k}{s_{k+1}}\right)^2\left(
		1 - \frac{R_k}{s_k}
		\right)\\
		&=\left(
		\frac{R_{k}}{s_{k+1}}
		 \right)^2
		 + \frac{2\X_{k+1} R_k^2}{s_{k+1}^2 s_k}
		 + \frac{\X_{k+1}^2R_k}{s_{k+1}^2 s_k} \\
		 &=
		 \left(\frac{R_k}{s_k} \right)^2 - \frac{R_k^2 \X_{k+1}^2}{s_{k+1}^2 s_k^2} + \frac{R_k\X_{k+1}^2}{s_{k+1}^2 s_k}.
	\end{align*}
 So, if we set $u_k = \Ec{(R_k/s_k)^2}$ for all $k \ge 0$, then we have for all $k\ge 0$,
 \begin{align*} 
 u_{k+1} = \left(1-\frac{\X_{k+1}^2}{s_{k+1}^2}\right)u_k+ \frac{r_0 \X_{k+1}^2}{s_0 s_{k+1}^2}.
 \end{align*} 
 So one finds that for all $k\ge 0$,
 \begin{align*} 
 u_k = \frac{r_0}{s_0} + \left( \frac{r_0^2}{s_0^2} - \frac{r_0}{s_0} \right) \prod_{j=1}^k \left(1- \frac{\X_{j}^2}{s_j^2} \right).
 \end{align*} 
 So 
 \begin{align*} 
\mathbb{E}[{Z}^{2}]= \lim_{k \to \infty} \Ec{\left( \frac{R_k}{s_k} \right)^2 }
 =
  \frac{r_0}{s_0} + \left( \frac{r_0^2}{s_0^2} - \frac{r_0}{s_0} \right) \prod_{j=1}^\infty \left(1- \frac{\X_j^2}{s_j^2} \right).
 \end{align*} 
 This entails the desired result.
\end{proof}

\section{Appendix: bounds on the Lambert function}
\label{sec:lambert}

Here we prove Lemma~\ref{lem:Lambert}.

\begin{proof}[Proof of Lemma~\ref{lem:Lambert}(i).]
Using the identity $W(xe^{x})=x$ for $x \geq -1$, and since $-1+\sqrt{2e}\sqrt{x+\frac{1}{e}}- \frac{2}{3}e \left( x+ \frac{1}{e}\right) \geq -1$ for $-1/e \leq x \leq 0$, since $W$ is increasing, it is enough to show that
\begin{equation}
\label{eq:lambert1}
x \geq  \left(-1+\sqrt{2e}\sqrt{x+\frac{1}{e}}- \frac{2}{3}e \left( x+ \frac{1}{e}\right)\right) e^{-1+\sqrt{2e}\sqrt{x+\frac{1}{e}}- \frac{2}{3}e \left( x+ \frac{1}{e}\right)}
\end{equation}
for $-1/e \leq x \leq 0$.
 Setting $y=x+1/e$ with $y \geq 0$, this is equivalent to showing that
\begin{align*} f(y)=y- \frac{1}{e}- \left(-1+\sqrt{2ey}- \frac{2}{3}e y\right) e^{-1+\sqrt{2ey}- \frac{2}{3}e y} \geq 0\end{align*} 
for every $y \geq 0$. To this end, we show that $f$ is increasing, and since $f(0)=0$ this will entail the result.

We have
\begin{align*} f'(y)=1+ \frac{1}{9}e^{\sqrt{2ey}- \frac{2}{3}ey}  \left( -9+9 \sqrt{2ey}-4ey\right).\end{align*} 
Setting $u=\sqrt{2ey}$, we show that for $u \geq 0$
\begin{align*} g(u)=1+ \frac{1}{9} e^{u- \frac{1}{3}u^{2}} (-9+9u-2u^{2}) \geq 0.\end{align*} 

\emph{Step 1: $0 \leq u \leq 3/2$.} Observe that $u \mapsto u- \frac{1}{3}u^{2}$ is a one-to-one function from $[0,3/2]$ to $[0,3/4]$. Using the change of variable $x=u- \frac{1}{3}u^{2}$, we get that $g(u)$ is equal to
\begin{align*} 
h(x)=1-\frac{1}{6} e^x \left(-4 x+\sqrt{9-12 x}+3\right), \qquad 0 \leq x \leq 3/4.
\end{align*} 
To show that $h(x) \geq 0$ for $0 \leq x \leq 3/4$ we show that $h$ is increasing, and since $h(0)=0$ this will entail the result. We have
\begin{align*} h'(x)=\frac{e^x }{6 \sqrt{3-4 x}} \left(4 \sqrt{3} x-\sqrt{3} +   (4x+1) \sqrt{3-4 x}\right).\end{align*} 
It is a simple matter to  check that $4 \sqrt{3} x-\sqrt{3} +   (4x+1) \sqrt{3-4 x} \geq 0$ for $0 \leq x \leq 3/4$ (e.g.~by differentiating this function is increasiong on $[0,(2 \sqrt{5}+3)/12]$ and decreasing on $[(2 \sqrt{5}+3)/12,3/4]$).

\emph{Step 2: $3/2 \leq u \leq 3$.} For $3/2 \leq u \leq 3$, we have $-9+9u-2u^{2} \geq 0$, so $g(u) \geq 0$.

\emph{Step 3: $u \geq 3$.} Using the inequality $-9+9u-2u^{2} \geq 6 (u-u^{2}/3)$ valid for $u \geq 3$, we get
\begin{align*} g(u) \geq 1+ \frac{2}{3} e^{u- \frac{1}{3}u^{2}} (u- \frac{1}{3}u^{2}).\end{align*} 
The fact that $g(u) \geq 0$ then comes from the fact that $1+ \frac{2}{3} x e^{x} \geq 0$ for every $x \leq 0$ (this function attains its infimum at $x=-1$).
\end{proof}

To establish Lemma~\ref{lem:Lambert}(ii) we use the following bounds: for all $h\ge 0$,
\begin{equation}
\label{eq:input}
1 - \frac{1}{2}h^2 + \frac{1}{3} h^3 - \frac{1}{8}h^4 \leq e^{-h}(1+h)  \leq1- \frac{h^2}{2}  + \frac{h^3}{3} - \frac{h^4}{8}  + \frac{h^5}{30} - \frac{h^6}{144} + \frac{h^7}{840},   \end{equation}
which can be seen by using alternating series.
In particular, this implies that for $0 \leq h \leq 1$,
\begin{equation}
\label{eq:input2}\sqrt{2-2e^{-h}(1+h)} \geq h-  \frac{1}{3}  h^{2} +  \frac{5}{72} h^{3}-  \frac{11}{1080} h^{4}.
\end{equation}
Indeed, by \eqref{eq:input}, for all $h\ge 0$,
$$
2- 2e^{-h}(1+h) \ge h^2 - \frac{2h^3}{3} + \frac{h^4}{4} - \frac{h^5}{15} + \frac{h^6}{72} - \frac{h^7}{420} .
$$
Besides,
$$
\left(h-\frac{h^2}{3} + \frac{5h^3}{72} - \frac{11h^4}{1080}\right)^2 = h^2 - \frac{2h^3}{3} + \frac{h^4}{4}- \frac{h^5}{15} + \frac{301h^6}{25920} - \frac{11h^7}{7776} + \frac{121h^8}{1166400}.
$$
Therefore, to prove \eqref{eq:input2} it suffices to check that for all $h \in [0,1]$,
$$ \frac{1}{72} - \frac{h}{420}  -\left( \frac{301}{25920} - \frac{11h}{7776} + \frac{121h^2}{1166400}\right)\ge0,
$$
which is elementary since it is a polynomial of degree $2$: one root is on the left of $(0,1)$, one root is on the right.

\begin{proof}[Proof of Lemma~\ref{lem:Lambert}(ii).]
By Lemma~\ref{lem:Lambert}(i), it is enough to prove that for $\lambda \geq 1$ we have
\begin{align*} -1+\sqrt{2e}\sqrt{-\lambda e^{-\lambda }+\frac{1}{e}}- \frac{2}{3}e \left( -\lambda e^{-\lambda }+ \frac{1}{e}\right) \geq (\lambda -1) \sqrt{2-2\lambda +\lambda ^{2}}-2+2\lambda -\lambda ^{2}.
\end{align*} 
Setting $\lambda=1+h$, this is equivalent to showing that
\begin{equation}
\label{eq:equiv}
h^2-h \sqrt{h^2+1} +\frac{2}{3} e^{-h} (h+1)+\sqrt{2-2 e^{-h} (h+1)}-\frac{2}{3} \geq 0.
\end{equation}

\emph{Step 1: $h \in [0,1]$.} We first show that \eqref{eq:equiv} holds for $ h \in [0,1]$. To this end, using \eqref{eq:input}, \eqref{eq:input2} and the inequality $\sqrt{1+h^{2}} \leq 1+ \frac{1}{2} h^{2}$, we get  
\begin{align*} h^2-h \sqrt{h^2+1} +\frac{2}{3} e^{-h} (h+1)+\sqrt{2-2 e^{-h} (h+1)}-\frac{2}{3} \geq  \frac{h^{2}}{1080}(360-225h-101h^{2}).\end{align*} 
The roots of the later second order polynomial are $\frac{3}{202} \left( \pm \sqrt{21785}-75\right)$, which implies that the inequality \eqref{eq:equiv} holds on $[0,1]$.

\emph{Step 2: $h \geq 1$}. We now show that   \eqref{eq:equiv} holds for $h \geq 1$. Using the inequality $h^2-h \sqrt{h^2+1}  \geq -1/2$ and the change of variable $x=e^{-h} (h+1) \in (0,2/e]$, we get
\begin{align*} 
h^2-h \sqrt{h^2+1} +\frac{2}{3} e^{-h} (h+1)+\sqrt{2-2 e^{-h} (h+1)}-\frac{2}{3}
\geq 
- \frac{7}{6}+ \frac{2}{3}x+\sqrt{2-2x}.
\end{align*} 
By differentiating, it is a simple matter to see that the latter function is decreasing in $x$ on $(0,2/e]$, and for $x=2/e$ it is equal to a positive real number (approximately equal to $0.05$). This completes the proof.
\end{proof}

\bibliographystyle{abbrv}

\bibliography{bibli}

\end{document}